\documentclass[11pt]{amsart}
\usepackage[T1]{fontenc}
\usepackage{graphicx} 
\usepackage{amsmath,amsfonts,amssymb,amsthm}
\usepackage{mathtools,mathrsfs} 
\usepackage{pifont} 

\usepackage{float}
\usepackage{tikz-cd}
\usepackage{quiver}
\usepackage{pbox} 
\usepackage{multicol}
\usepackage[margin=1.5in]{geometry}
\usepackage{array}
\usepackage{longtable}
\usepackage{multirow}
\usepackage{verbatim}
\usepackage{fancyvrb} 
\usepackage{fvextra}
\usepackage{soul} 
\usepackage{MnSymbol} 
\usepackage{hyperref}
\hypersetup{
    colorlinks=true,
    linkcolor=black,
    citecolor=black,
    filecolor=black,      
    urlcolor=black,
    }

\usepackage[dvipsnames]{xcolor}

\usepackage{cancel}  

\usepackage{pifont}

\newtheorem{thm}{Theorem}[section]
\newtheorem{lemma}[thm]{Lemma}
\newtheorem{prop}[thm]{Proposition}
\newtheorem{coro}[thm]{Corollary}

\newtheorem{remark}[thm]{Remark}
    
\newcommand{\am}{\operatorname{Am}}
\newcommand{\aut}{\operatorname{Aut}}

\newcommand{\pic}{\operatorname{Pic}}

\newcommand{\pp}{\mathbb{P}}
\newcommand{\cc}{\mathbb{C}}
\newcommand{\zz}{\mathbb{Z}}
\newcommand{\autp}{\operatorname{AutP}}

\numberwithin{equation}{section} 

\author{Shreya Sharma}
\address{Shreya Sharma, University of Basel, Switzerland.}
\email{shreya.sharma@unibas.ch}

\title{Amitsur groups of Imprimitive Fano Threefolds}

\begin{document}
\begin{abstract}
    We classify possible Amitsur groups of smooth imprimitive Fano threefolds defined over $\cc$ that admit a faithful action of a finite group.
\end{abstract}
\maketitle
\section{Introduction}   \label{section:intro}
Let $X$ be a proper variety over $\cc$ with a faithful action of a finite group $G$. The Amitsur group $\am(X,G)$ measures the obstruction to lifting the $G$-action to the total
spaces of $G$-invariant line bundles. The term was introduced in the arithmetic context in
\cite{Liedtke-Amitsur-arithmetic-version}, where it measures the failure of Galois descent for line bundles. The equivariant version was first defined in \cite{BCDP}.

\smallskip
A line bundle $\mathcal{L}$ on $X$ is $G$-invariant if there are isomorphisms $\phi_g:g^*\mathcal{L}\to \mathcal{L}$ for every $g\in G$. For a $G$-invariant line bundle $\mathcal{L}$, the lifting group $G_{\mathcal{L}}$ consists of  isomorphisms $\phi_g:g^*\mathcal{L}\to\mathcal{L}$ lifting $g$ over all $g\in G$, see \cite[2.1]{duncan2025numericalamitsurgroup}. It is an extension of $G$ by $\cc^\times$, which splits if and only if $\mathcal{L}$ is $G$-linearizable. There is an exact sequence of abelian groups 
\begin{equation}   \label{sequence:cohomological-for-amitsur}
1 \rightarrow \operatorname{Hom}(G,\mathbb{C}^\times) \rightarrow \pic(X,G) \rightarrow \pic(X)^G \xrightarrow{\partial} H^2(G,\mathbb{C}^\times),
\end{equation}
see also \cite[\S A]{BCDP}. 

The \emph{Amitsur group} is defined as 
\begin{equation}   \label{defn:Amitsur}
  \am(X,G):= \operatorname{Im}\partial,   
\end{equation}
or equivalently, it is isomorphic to $\operatorname{coker}(\pic(X,G) \to \pic(X)^G)$. 

\smallskip
The Amitsur group has been shown to be an equivariant birational invariant in \cite{BCDP}, where the authors use it to distinguish conjugacy classes of finite subgroups of $\operatorname{Cr}_3(\cc)$. It provides an obstruction to equivariant versions of rationality,
such as linearizability; see for example, \cite[Lemma 2.22]{cheltsov2026equivariantunirationalityfanothreefolds}. 

\smallskip
The Amitsur group is known for smooth curves
\cite{dolga-modular-curves}, and its possibilities for rational surfaces were classified in
\cite[Proposition A.7]{BCDP}. 
The author classified the possible Amitsur groups of smooth primitive Fano threefolds in
\cite{sharma2026amitsurgroupsprimitivefano}.
In addition, the methods for computing these groups for toric varieties
were developed in
\cite{duncan2025numericalamitsurgroup,
kresch2025equivariantunirationalitytoricvarieties}, see \cite[\S 3]{sharma2026amitsurgroupsprimitivefano} for their application to smooth toric Fano threefolds. 

\medskip
In this article, we classify the possible Amitsur groups for the remaining families of smooth imprimitive Fano threefolds. Our main result is the following. 
\begin{thm}
    Let $X$ be a smooth Fano threefold with Picard rank $\rho\leq 5$ and $G\subseteq \aut(X)$ be a finite group. Then the Table \ref{table:amitsur} lists the largest possible Amitsur group $\am(X,G)$. 
    Moreover, for each entry in the table, except possibly for family \textnumero 4.2, there is an $X$ and a group $G$ for which that possibility occurs.   
\end{thm}

\begin{remark}
 \normalfont \hfill
 \begin{enumerate}
 \item[(i)] If $X$ is a smooth Fano threefold in family \textnumero 4.1 with $\autp(X,G)=(\zz/2\zz)^2$, we only provide an upper bound on $\am(X,G)\subseteq \zz/2\zz$; see Remark \ref{remark:for4.1} (iv). The question of whether this bound can be realized is not yet answered. 
     \item[(ii)] If $X$ is a smooth Fano threefold in family \textnumero 4.2, then Lemma \ref{Am:4.2} provides an upper bound on $\am(X,G) \subseteq (\zz/2\zz)^2$. It is not yet known whether this bound is achieved. 
     \item[(ii)] A smooth Fano threefold $X$ with $\rho(X)\geq 6$ belongs to one of the families 
     \begin{center}
        \textnumero 6.1,\; \textnumero 7.1,\; \textnumero 8.1,\; \textnumero 9.1,\; \textnumero 10.1 
     \end{center} 
     and is isomorphic to $\pp^1 \times S_d$, where $S_d$ is a del Pezzo surface of degree $d\leq 5$. For all these families, the computations of Amitsur groups is reduced to the study of the induced actions on the two factors, see \S \ref{section:directproducts}. 
 \end{enumerate}
\end{remark}

Our approach uses the natural action of $\aut(X)$ on $\pic(X)$, 
\begin{equation}  \label{autpaction}
    \operatorname{Aut}(X) \xrightarrow{\alpha} \operatorname{Aut}\operatorname{Pic}(X).
\end{equation}
For a finite subgroup $G\subseteq\aut(X)$, set
$\autp(X,G):=\alpha(G)$, and write $\autp(X):=\alpha(\aut(X))$.
Since
\[
\pic(X)^G=\pic(X)^{\autp(X,G)},
\]
the image $\autp(X,G)$ determines the $G$-invariant Picard subgroup, $\pic(X)^G$. We use the results on $G$-actions on $\pic(X)$ obtained by the author in
\cite{sharma2026actionspicardgroupsmooth}, see Table \ref{table:amitsur} for a summary. In many of our examples, we begin by identifying classes in $\pic(X)^G$ that are known to be $G$-linearizable, such as the class of canonical divisor. Quotienting by the subgroup they generate restricts the possibilities for $\am(X,G)$ and yields upper bounds for this group.
We then construct explicit members $X$ of the deformation families, together with finite subgroups $G\subseteq\aut(X)$, to determine which of the nontrivial possibilities are realized.
 
\medskip
\subsection*{Outline} The document is organized as follows. In \S 2, we recall the basic definitions and machinery concerning linearizations of line bundles and Amitsur groups. We also establish some general results concerning imprimitive Fano threefolds and complete intersections. In \S 3, we consider families whose members arise as blow-ups of del Pezzo and quadric threefolds, and in \S 4 those that arise as blow-ups of projective space. Sections 5 and 6 concern Fano threefolds described as divisors and complete intersections, respectively, in products of projective spaces. Sections 7 and 8 deal with blow-ups of $\pp^1\times\pp^2$ and $\pp^1\times\pp^1\times\pp^1$, respectively. In \S 9, we treat the remaining imprimitive Fano threefolds. In \S 10, we consider families containing a member of the form $\pp^1\times S_d$, where $S_d$ is a del Pezzo surface of degree $d$. Finally, in \S 11, we summarize the classification in a table listing the largest possible Amitsur groups for all primitive and imprimitive Fano threefolds.

\subsection*{AI Declaration} The author declares that generative AI has not been used in any part of this work.  

\section{Preliminaries}
\subsection{Notation and Conventions} Throughout, all varieties are assumed to be smooth, projective, and defined over $\cc$. 
A deformation family of smooth Fano threefolds will be denoted by \textnumero $\rho$.N, where $\rho$ is the Picard rank of its members and $N$ is its number in the Mori-Mukai classification \cite{MM81classtable, MMerratum}. So, by a member $X$ of a deformation family \textnumero $\rho.N$, we would mean a smooth Fano threefold belonging to the family. 

\smallskip
Let $X$ be a $G$-variety, that is, a variety admitting action by a finite group. 

We write $\pic(X)^G$ for the group of isomorphism classes of $G$-invariant line bundles on $X$ and $\pic(X,G)$ for the group of isomorphism classes of $G$-linearizable line bundles for a fixed choice of linearization. Every $G$-linearizable line bundle is $G$-invariant, so the forgetful homomorphism $\pic(X,G) \to \pic(X)^G$ fits into the exact sequence \eqref{sequence:cohomological-for-amitsur}. We recall that the Amitsur group is defined as 
\[
\am(X,G):= \operatorname{Im}\partial \cong \operatorname{coker}(\pic(X,G)\to \pic(X)^G).
\]

\smallskip
We use additive notation for $\pic(X)$. When no confusion can arise, we use the same notation $D$ for a Cartier divisor, the line bundle $\mathcal{L}:=\mathcal{O}_X(D)$, and its class $[\mathcal{L}] \in \pic(X)$. A divisor $D\in \pic(X)^G$ is $G$-linearizable if corresponding line bundle class $\mathcal{O}_X(D)$ is $G$-linearizable, and we write $\overline{D}:= \partial(D) \in \am(X,G)$ for the image of divisor $D\in \pic(X)^G$. In particular, $\overline{D}=0$ means that $D$ admits a $G$-linearization. This notation requires the class of $D$ to be $G$-invariant but the divisor itself need not be $G$-invariant.

\subsection{Linearization Results and Amitsur Period}

\begin{prop}\cite[Proposition 2.11]{BCDP}   \label{KXlinearizable}
    If $X$ is a smooth $G$-variety, then the canonical bundle has a canonical linearization. 
\end{prop}
Since linearizations are preserved under pullbacks along $G$-equivariant morphisms, canonical divisors and their equivariant pullbacks provide relations in $\am(X,G)$. 

\smallskip
The following is another consequence of this proposition. 
\begin{prop}\cite[Proposition A.5]{BCDP}   \label{Prop:rkPic_G=1}
    If $X$ is a smooth Fano variety of index $1$ with $\operatorname{rk}\pic(X)^G=1$, then $\am(X,G)=0$ for all finite groups $G\subseteq \aut(X)$.
\end{prop}
\begin{proof}
    The hypotheses imply $\pic(X)^G=\zz[K_X]$. The result follows from Proposition \ref{KXlinearizable}.
\end{proof}

Let $\mathcal{L}\in \pic(X)^G$. The \emph{Amitsur period} of $\mathcal{L}$, denoted $m(\mathcal{L})$, is the order of $\partial(\mathcal{L})$ in $H^2(G,\cc^\times)$, see \cite[Definition 2.5]{duncan2025numericalamitsurgroup}. If $D$ is a divisor and $\mathcal{L}=\mathcal{O}_X(D)$, we write $m(D):=m(\mathcal{L})$. 

\begin{prop}(\cite[Proposition 2.6]{duncan2025numericalamitsurgroup}) \label{Gsubrepn_result}
    Let $X$ be a $G$-variety and let $\mathcal{L}$ be a $G$-invariant line bundle on $X$ with Amitsur period $m(\mathcal{L})$. If there exists a $G_{\mathcal{L}}$-subrepresentation of $H^0(X,\mathcal{L})$ of dimension $d$, then $d\partial(\mathcal{L})=0$. In particular, $m(\mathcal{L})$ divides $d$.
\end{prop}

For Fano threefolds, intersection products impose further restrictions on these periods, providing a uniform bound.

\begin{prop}( \cite[Proposition 4.8]{duncan2025numericalamitsurgroup})   \label{prop:amitsur-period-congruences}
    Suppose $X$ is a smooth Fano threefold and $D\in \pic(X)^G$. Then $m(D)$ divides $12$ and 
    \[
    2D^3 \equiv 6 \pmod{m(D)}, \quad D^2\cdot K_X \equiv 4 \pmod{2m(D)}.
    \]
\end{prop}

\subsection{Equivariant Morphisms and Blow-ups}
Let $X$ be a $G$-variety. Then $\am(X,-)$ is a contravariant functor. 

\begin{remark}  \label{remark:equivariant-restrction-Amitsur}
\normalfont Let $H\subseteq G$ be a subgroup. Then a line bundle $\mathcal{L}\in \pic(X)^G$ is not $G$-linearizable if it is not $H$-linearizable. 
\end{remark}
See \cite[Proposition 2.12]{duncan2025numericalamitsurgroup} for a generalization to any group homomorphism $\phi: H \to G$. 

\begin{prop}  \label{prop:g-eq-morphism-sameAmitsur}
    Let $f: X \to Y$ be a $G$-equivariant morphism such that $f^*: \pic(Y) \to \pic(X)$ is an isomorphism. Then $\am(X,G)=\am(Y,G)$. 
\end{prop}
\begin{proof}
    See \cite[Proposition 2.6]{sharma2026amitsurgroupsprimitivefano} and \cite[Lemma 2.1]{kresch-unramified-brauergroup} for a justification.
\end{proof}

A Fano threefold $X$ is called \emph{imprimitive} if it is isomorphic to the blow-up of a Fano threefold $Y$ along a smooth irreducible curve $C\subset Y$ \cite{MM81classtable}, otherwise it is called primitive. The notation $\aut(Y;C)$ will be used for the group of automorphisms of $Y$ that preserve $C$. 

By the universal property of blow-ups, the following lifting lemma provides a necessary and sufficient condition for automorphisms of $Y$ to lift to $X$.

\begin{lemma}    \label{Lemma:autblowup}
   Let $X$ be an imprimitive Fano threefold obtained by blowing-up a Fano threefold $Y$ in a curve $C$, and let $E\subset X$ denote the exceptional divisor. There is a natural map $\aut(Y;C) \to \aut(X)$. Conversely, $G \subseteq \aut(X)$ is a lift of a subgroup of $\aut(Y)$ if and only if $g^*E=E$ for all $g\in G$. 
\end{lemma}

Let $G \subseteq \aut(X)$ be a finite subgroup such that $E$ is $G$-invariant, and let $H$ denote a descended image of $G$ in $\aut(Y;C)$. Then $H$ is unique and we will identify it with the group $G$. 

\smallskip
The following result will be of practical importance in our examples. It says that $\am(X,G)$ is an equivariant birational invariant.

\begin{thm}\cite[Theorem A.1]{BCDP} \label{imprimitivebcdp}
    If $X$ and $Y$ are smooth projective $G$-varieties that are $G$-equivariantly birationally equivalent, then 
    \[
    \am(X,G)=\am(Y,G).
    \]
\end{thm}
In particular, if $X$ is the blow-up of $Y$ in a $G$-invariant subvariety and $E$ is the exceptional divisor, then $\mathcal{O}_X(E)$ is $G$-linearizable, see the proof of \cite[Theorem A.1]{BCDP}.

When working with blow-up descriptions, we use the classification of Amitsur groups of smooth primitive Fano threefolds from \cite{sharma2026amitsurgroupsprimitivefano}. Table \ref{table:amitsur} summarizes this classification.  

\subsection{Divisors and Complete Intersections}
We first recall when the natural embedding of a smooth complete intersection into its ambient variety induces an isomorphism of Picard groups. This is a consequence of the Lefschetz hyperplane
theorem; see \cite[Proposition 2.1]{sharma2026actionspicardgroupsmooth} for a proof.
\begin{prop}\label{thm:lefschetz}
   Let $Y$ be a smooth Fano variety of dimension at least $4$. Let $X$ be a smooth complete intersection of ample divisors on $Y$ such that $X$ is a Fano threefold. Then $\pic(Y)\cong \pic(X)$.
\end{prop}

The following observation will be useful to work with the automorphisms of the ambient projective variety. 

\begin{remark}  \label{remark:equivarince-of-embedding}
  \normalfont Let $Y:=\pp^{n_1}\times \dots \times \pp^{n_r}$, and let $\iota : X \to Y$ denote a smooth projective subvariety. Denote by $\operatorname{pr}_i\colon Y\to\pp^{n_i}$ the
natural projections and set $H_i :=\operatorname{pr}_i^*\mathcal{O}_{\pp^{n_i}}(1)$. Let $G\subseteq \aut(X)$ be a finite group. Suppose that $\operatorname{pr}_i|_X: X \to \pp^{n_i}$ is the morphism induced by the linear system $|\iota^*H_i|$. If $G$ acts on the set $\{ \iota^*H_1, \dots , \iota^*H_r\}$ by permutations, then the $G$-action extends to $Y$. In particular, $\iota$ is $G$-equivariant. 
\end{remark}

The following lemma was communicated to the author by Alexander Duncan. 
\begin{lemma}  \label{lemma:G-invariance-for-CI}
Let $R=\bigoplus R_D$ be a finitely generated $\cc$-algebra multigraded nonnegatively\footnote{A ring $R$ is multigraded nonnegatively by a free abelian group $P$ if for any $D\in P\setminus \{0\}$ either $R_D=0$ or $R_{-D}=0$. }. Let $\widetilde{G}$ be a finite group that acts on $R$ by graded algebra automorphisms. Suppose that $I=\langle f_1,\dots , f_r\rangle$ is a $\widetilde{G}$-invariant homogeneous ideal minimally generated by homogeneous polynomials $f_1,\dots , f_r$ with $\operatorname{deg}f_i=D_i$, $1\leq i \leq r$. 
\begin{enumerate}
    \item[(i)] If $R_{D_i-D_j}=0$ whenever $i\neq j$, then $f_1, \dots , f_r$ are $\widetilde{G}$-semi-invariant. 
    \item[(ii)] If the degrees $D_1, \dots , D_r$ are distinct in $P$, then there exist a choice of $\widetilde{G}$-semi-invariant homogeneous polynomials $f_1',\dots , f_r'$ with $\deg f_i'=D_i$ such that $I=\langle f_1',\dots , f_r'\rangle$.
\end{enumerate}
\end{lemma}
\begin{proof}
    Since $I$ is $\widetilde{G}$-invariant and the $\widetilde{G}$-action preserves degrees, then $I\cap R_{D_i}$ is a one-dimensional $\widetilde{G}$-invariant space generated by $f_i$. This shows (i). 

    For (ii), let us first order the generators $f_i$ so that $R_{D_i-D_j}=0$ for $i<j$. For $1\leq i \leq r$, set 
    $I_i:= \langle f_1,\dots , f_i\rangle$, and consider the sequence of ideals 
    \[
    (0)=: I_0\subsetneq I_1 \subsetneq I_2 \subsetneq \dots \subsetneq I_{r-1}\subsetneq I.
    \]
    Observe that $R_{D_i}\cap I\subseteq I_i$ and $I_i$ is $\widetilde{G}$-invariant. Let $U_i:= I\cap R_{D_i}$ and $V_i:= I_{i-1}\cap R_{D_i}$. Then we have $U_i=V_i+\cc f_i$. Moreover, $U_i/V_i$ is a one-dimensional $\widetilde{G}$-invariant vector space due to minimality of generators $f_i$'s. Now since $\widetilde{G}$ is finite, $V_i$ has a one-dimensional $\widetilde{G}$-stable complement $W\subset U_i$ such that $U_i=V_i\oplus W$. Choose a nonzero element $f'_i\in W$. It is $\widetilde{G}$-semi-invariant with degree $D_i$, and after rescaling, if necessary, satisfies $f'_i-f_i\in I_{i-1}$. 
    Now by induction on $i$, 
    \[
    \langle f'_1,\dots , f'_i\rangle = \langle f'_1,\dots , f'_{i-1}\rangle + \langle f'_i\rangle = I_{i-1} + \langle f'_i\rangle = I_i
    \]
    for every $i$. This completes the proof of (ii).
\end{proof}

In our application, $Y$ is a product of projective spaces, and $R:=\operatorname{Cox}(Y)$ is its Cox ring. For a finite group $G\subseteq \aut(Y)$, let $\widetilde{G}$ be a finite lifting group acting on $R$ such that it induces the $G$-action on $Y$. We apply the above lemma to the $\widetilde{G}$-action on $R$ and to a $\widetilde{G}$-invariant complete intersection ideal  $I$ of a curve $C$ cut out by hypersurfaces in $Y$. 

\begin{remark}
\normalfont If the $G$-action on $\pic(Y)$ is nontrivial, the $\widetilde{G}$-action on $R$ may permute its graded components. In each application of Lemma \ref{lemma:G-invariance-for-CI}, we verify that the classes of $D_1, \dots , D_r$ are individually $G$-invariant. Thus, the graded components of $R$ are $\widetilde{G}$-stable and the proof of the lemma remains valid in this setting. 

\end{remark}

Under the hypotheses of lemma, the divisors $D_i$s may be replaced by $G$-invariant linearly equivalent divisors without changing the complete intersection $C$. 
\begin{remark} \label{remark:CI}
 \normalfont Let $D_i:=\mathbb{V}(f_i)$. Suppose that Lemma \ref{lemma:G-invariance-for-CI} (ii) yields a $\widetilde{G}$-semi-invariant polynomial $f'_i$ with $\operatorname{deg}f'_i=D_i$. Let $D'_i:=\mathbb{V}(f'_i)$. Then $D'_i$ is a $G$-invariant divisor, $D'_i \in |D_i|$, and for $g\in G$, the isomorphisms $\phi_g:g^*\mathcal{O}(D'_i)\to \mathcal{O}(D'_i)$ satisfy the compatibility condition of $2$-cocycles. As a result, $\mathcal{O}(D'_i)$, and hence $\mathcal{O}(D_i)$ admits a $G$-linearization. In (i), the divisors $D_i$ are themselves $G$-invariant, hence, are $G$-linearizable by the same argument. 
\end{remark}

\section{Blow-ups of del Pezzo and quadric threefolds} 

Let $Y$ be a Fano threefold with Picrd rank $\rho(X)=1$ belonging to one of the families
\begin{center} \textnumero 1.11, \textnumero 1.12, \textnumero 1.13, \textnumero 1.14, \textnumero 1.15. \end{center}
Then $Y$ has index $r=2$ and is called a del Pezzo threefold. A member of family \textnumero 1.16 is a quadric threefold and has index $r=3$. In this section, we determine the possible Amitsur groups of Fano threefolds $X$ obtained by blowing up these threefolds along smooth irreducible curves. More precisely, we classify possible $\am(X,G)$
\begin{enumerate}
    \item[(i)] if $X$ belongs to one of the families
\begin{center}
    \textnumero 2.1, \textnumero 2.3, \textnumero 2.5, \textnumero 2.10, \textnumero 2.11, \textnumero 2.14, \textnumero 2.16, \textnumero 2.19, \textnumero 2.20, \textnumero 2.22, \textnumero 2.26.
    \end{center} 
    Here $X$ is the blow-up of a del Pezzo threefold along a smooth
    irreducible curve; 
   \item[(ii)] if $X$ belongs to one of the families
   \begin{center}
   \textnumero 2.7, \textnumero 2.13, \textnumero 2.17, \textnumero 2.29, \textnumero 2.31.
\end{center}
Here $X$ is the blow-up of the quadric threefold \textnumero 1.16 along a smooth irreducible curve.
\end{enumerate}

Note that $\autp(Y)=0$. Let $\pi : X \to Y$ denote the blow-up morphism with exceptional divisor $E$. Let $L$ denote a general hyperplane on $Y$, then $\pic(X)=\zz[\pi^*L]\oplus \zz[E]$. 

\subsection{Blow-ups of del Pezzo threefolds}
Let $Y$ be a del Pezzo threefold that belongs to family \textnumero 1.11, it is a hypersurface of degree $6$ in $\pp(1,1,1,2,3)$, and let $G\subseteq \aut(Y)$ be a finite group. Then $\am(Y,G)=0$ (Table \ref{table:amitsur}). 

\smallskip
A Fano threefold $X$ in family \textnumero 2.1 is the blow-up of $Y$ along a smooth elliptic curve $C$ \cite[Table 2]{MM81classtable}.

\begin{lemma}   \label{Am:2.1}
    Let $X$ be a Fano threefold in family \textnumero 2.1. Let $G\subseteq \aut(X)$ be a finite group. Then $\am(X,G)=0$.
\end{lemma}
\begin{proof}
    Since $\autp(X)=0$ (Table \ref{table:amitsur}), Lemma \ref{Lemma:autblowup} implies that we may regard $G$ as a subgroup of $\aut(Y;C)$, and $\pi$ is $G$-equivariant. By Theorem \ref{imprimitivebcdp}, $\am(X,G)=\am(Y,G)$. Since $\am(Y,G)=0$, we get $\am(Y,G)=0$.
\end{proof}

Suppose that $Y$ is a Fano threefold in family \textnumero 1.12. Then $Y$ is a double cover of $\mathbb{P}^3$ branched along a smooth quartic surface $B$ \cite{isko2,MM81classtable}. Let $\phi: Y \to \pp^3$ denote this double covering map. Let $D_1, D_2 \in |-\frac{1}{2} K_Y|$ and let $C := D_1\cap D_2$ be their intrsection, which is a smooth elliptic curve on $Y$. A Fano threefold $X$ in family \textnumero 2.3 is obtained by blowing up $Y$ along $C$.

\begin{lemma}   \label{Am:2.3}
    Let $X$ be a Fano threefold in family \textnumero 2.3. Let $G\subseteq \aut(X)$ be a finite group. Then $\am(X,G)$ is isomorphic to $\zz/2\zz$ or trivial.
\end{lemma}
\begin{proof}
From Table \ref{table:amitsur}, we have $\autp(X)=0$. So, by Lemma \ref{Lemma:autblowup}, $G$ descends to a subgroup of $\aut(Y;C)$ and we may assume that $\pi$ is $G$-equivariant. Since $\am(Y,G)\subseteq \zz/2\zz$ for every $G\subseteq \aut(Y)$ (Table \ref{table:amitsur}), Theorem \ref{imprimitivebcdp} gives 
\[
\am(X,G) \subseteq \zz/2\zz.
\]
We now show that this upper bound can be realized. Let $\mathcal{G}\subset \aut(\pp^3)$ be the group generated by 
\begin{equation*}      
\begin{split}
   \sigma :[x_0:x_1:x_2:x_3] &\mapsto [x_0:ix_1:-x_2:-ix_3]\\
    \tau : [x_0:x_1:x_2:x_3] &\mapsto [x_2:x_3: x_0:x_1]  
\end{split}
\end{equation*}
Let $B:= \{x_0^4 + x_1^4 + x_2^4 + x_3^4 =0\} \subset \pp^3$, it is $\mathcal{G}$-invariant, and can be verified to be both smooth and irreducible. So, $\mathcal{G}\subset \aut(\pp^3;B)$. Denote by $\overline{\mathcal{G}} \subset \aut(Y)$ the preimage of $\mathcal{G}$.

Observe that $\mathcal{G}$-action on $\pp^3$ does not lift to a homomorphism to $\operatorname{GL}_4(\cc)$:  if $S,T\in \operatorname{GL}_4(\cc)$ are linear lifts of $\sigma$ and $\tau$, then we have $ST = -TS$. Hence $\mathcal{O}_{\pp^3}(1)$ is not $\mathcal{G}$-linearizable. However, $\mathcal{O}_{\pp^3}(2)$ is $\mathcal{G}$-linearizable since the $\mathcal{G}$-action on $\pp^9$ lifts to a linear action. Thus, $\am(\pp^3,\mathcal{G})= \zz/2\zz$.

By the same argument as in \cite[Lemma 4.4]{sharma2026amitsurgroupsprimitivefano}, the natural map 
\[ 
\am(\pp^3,\mathcal{G})\longrightarrow \am(Y,\overline{\mathcal{G}})
\] 
is surjective with kernel generated by $[\mathcal{O}_{\pp^3}(2)]$. Since this class is already trivial in $\am(\pp^3,\mathcal{G})$, we obtain
\[
\am(Y,\overline{\mathcal{G}})\cong\zz/2\zz. 
\]

Denote the hyperplane class on $\pp^3$ by $H$. By the Hurwitz formula, we have $-\frac{1}{2}K_Y = \phi^*H$. Let $D_1 := \{x_0=0\}$ and $D_2:= \{x_2=0\}$ be two hyperplanes in $\pp^3$. 
Then 
\[
D_1 \cap D_2=\{ x_0=0,\; x_2=0\} \subset \pp^3
\]
is $\mathcal{G}$-invariant, smooth, and irreducible. Set $C:=\phi^{-1}(D_1\cap D_2)$. Then $C$ is $\overline{\mathcal{G}}$-invariant and smooth on $Y$. Thus, $\overline{\mathcal{G}}\subset \aut(Y;C)$ and it lifts to a subgroup of $\aut(X)$, making $\pi$ a $\overline{\mathcal{G}}$-equivariant map. By Theorem \ref{imprimitivebcdp}, it follows that 
\[
\am(X,\overline{\mathcal{G}})= \am(Y,\overline{\mathcal{G}}) \cong \zz/2\zz.
\]
\end{proof}

A del Pezzo threefold $Y$ in family \textnumero 1.13 is a cubic threefold in $\pp^4$ \cite{MM81classtable}. We have $\autp(Y)=0$ and $\am(Y,G)=0$ for every finite group $G\subseteq \aut(Y)$; see Table \ref{table:amitsur}. Let $X$ be a Fano threefold belonging to  \textnumero 2.5 or \textnumero 2.11. Then it is the blow-up of $Y$ along a plane cubic or a line, respectively. 
\begin{lemma}    \label{Am:2.5,2.11}
    Let $X$ be a Fano threefold belonging to family \textnumero 2.5 or \textnumero 2.11. Let $G\subseteq \aut(X)$ be a finite group. Then $\am(X,G)=0$.
\end{lemma}
\begin{proof}
     In both cases, $\autp(X)=0$; see Table \ref{table:amitsur}. Hence, by Lemma \ref{Lemma:autblowup}, the $G$-action on $X$ descends to $Y$, and we may identify $G$ with a subgroup of $\aut(Y;C)$. Theorem \ref{imprimitivebcdp} then gives $\am(X,G)=\am(Y,G)$. Since $\am(Y,G)=0$, we conclude $\am(X,G)=0$.
\end{proof}

A del Pezzo threefold $Y$ in family \textnumero 1.14 is the complete intersection of two quadrics in $\mathbb{P}^5$ \cite{isko2,MM81classtable}. Let $X$ be a Fano threefold $X$ belonging to family \textnumero 2.10 or \textnumero 2.16, then it is the blow-up of $Y$ along a smooth curve $C$. If $X$ belongs to \textnumero 2.10, then $C$ is an elliptic curve cut out by two hyperplanes, whereas if $X$ belongs to  \textnumero 2.16, then $C$ is a conic \cite[Table 2]{MM81classtable}.

\begin{remark}   \label{remark:blow-ups-of-1.14}
\normalfont In both cases, $\autp(X)=0$ (Table \ref{table:amitsur}). Hence, by Lemma \ref{Lemma:autblowup} we may identify $G$ as a subgroup of $\aut(Y;C)$ so that $\pi$ is $G$-equivariant. Theorem \ref{imprimitivebcdp} then gives
\begin{equation}  \label{eqn:amitsur-bound-for-1.14}
\am(X,G)=\am(Y,G) \subseteq \zz/2\zz,
\end{equation}
where the inclusion follows from the classification of Amitsur groups of $Y$; see Table \ref{table:amitsur}.
\end{remark}

It remains to check whether $\am(X,G)\cong \mathbb{Z}/2\mathbb{Z}$ can be realized for some $X$ and finite group $G\subseteq \aut(X)$. More concretely, it remains to determine whether $\pi^*L$ is $G$-linearizable, where $\pi: X \to Y$ is the blow-up map.

\begin{lemma}     \label{Am:2.10}
   Let $X$ be a Fano threefold in family \textnumero 2.10. Let $G\subseteq \aut(X)$ be a finite group. Then $\am(X,G)$ is isomorphic to $\zz/2\zz$ or trivial.
\end{lemma}
\begin{proof}

We recall the following construction of $Y$ from \cite[Lemma 4.7]{sharma2026amitsurgroupsprimitivefano}. Let $\mathcal{G}:=\langle\sigma$, $\tau\rangle \subset \aut(\pp^5)$, where
\begin{align}
\begin{split}
    \sigma : [x_0:x_1:x_2:x_3:x_4:x_5] &\mapsto [-x_0:-x_1:-x_2:x_3:x_4:x_5]\\
    \tau : [x_0:x_1:x_2:x_3:x_4:x_5] &\mapsto [\zeta x_3:\zeta x_4:\zeta x_5: x_0:x_1:x_2],  
\end{split}
\end{align}
and $\zeta$ is a primitive sixth root of unity.
Consider the quadrics
\begin{align} 
\begin{split}
    Q_1 &: x_0^2 -x_1^2 +  2x_2^2 -2 x_3^2 +  3 x_4^2 -3x_5^2=0\\
    Q_2 &: -2x_0^2 +  3x_1^2 -3x_2^2 + \zeta^2(x_3^2 -x_4^2 + 2x_5^2)=0.  
\end{split}
\end{align}
Furthermore, the complete intersection $Y:=Q_1\cap Q_2$ is $\mathcal{G}$-invariant and defines a smooth Fano threefold in family \textnumero 1.14 such that $\am(Y,\mathcal{G})\cong\zz/2\zz$. 
  
Let 
$H_1:=\{x_0=0\}$ and $H_2:=\{x_3=0\}$ be hyperplanes in $\mathbb{P}^5$, and set 
\[
C:=Y \cap H_1 \cap H_2 = \{Q_1= Q_2 =x_0 =x_3=0\}.
\]
The subspace $H_1\cap H_2$ is $\mathcal{G}$-invariant: it is preserved by $\sigma$, while $\tau$ interchanges the two hyperplanes. Since $Y$ is also $\mathcal{G}$-invariant, the curve $C$ is $\mathcal G$-invariant. A direct computation in Magma shows that $C$ is smooth and irreducible \cite{magma}. Let $X$ denote the blow-up of $Y$ in $C$, then $X$ is a Fano threefold in \textnumero 2.10. Since $C$ is $\mathcal{G}$-invariant, $\mathcal{G}$ lifts to a subgroup of $\aut(X)$. 
Finally, Theorem \ref{imprimitivebcdp} implies that 
\[
\am(X,\mathcal{G}) \cong \mathbb{Z}/2\mathbb{Z}.
\]
\end{proof}

\begin{lemma}      \label{Am:2.16}
     Let $X$ be a Fano threefold in family \textnumero 2.16. Let $G\subseteq \aut(X)$ be a finite group. Then $\am(X,G)=0$.
\end{lemma}
\begin{proof}
 Since a conic $C\subset \pp^5$ uniquely spans a $2$-dimensional space, it is cut out by three linearly independent hyperplanes, $H_1, H_2,$ and $H_3$ in $\pp^5$. Then 
\[
C:= Y \cap S, \text{ where } S := H_1 \cap H_2 \cap H_3.
\]
Since $C$ is $G$-invariant (Remark \ref{remark:blow-ups-of-1.14}), $S$ is also $G$-invariant. If $G_{\mathcal{O}_{\pp^5}(1)}$ denotes the lifting group of $\mathcal{O}_{\pp^5}(1)$, there is a $3$-dimensional $G_{\mathcal{O}_{\pp^5}(1)}$-subrepresentation of $H^0(\pp^5,\mathcal{O}_{\pp^5}(1))$. It now follows from Proposition \ref{Gsubrepn_result} that $3\overline{H}=0$, where $H$ denotes a general hyperplane on $\pp^5$. Consequently, $\overline{3L}=0$. Since $\pi$ is $G$-equivariant, $3\overline{\pi^*L}=0$. Together with $2\overline{\pi^*L}=0$ obtained from \eqref{eqn:amitsur-bound-for-1.14}, this gives $\overline{\pi^*L}=0$, and consequently, $\am(X,G)=0$.
\end{proof}

Let $Y$ denote a Fano threefold in family \textnumero 1.15. Then $\autp(Y)=0$ and $\am(Y,G)=0$ for any finite subgroup $G\subseteq \aut(Y)$, Table \ref{table:amitsur}. A Fano threefold $X$ in any of the following families 
\begin{center}
    \textnumero 2.14,\; \textnumero 2.20,\; \textnumero  2.22,\; \textnumero 2.26
\end{center} 
is the blow-up of $Y$ along a smooth curve $C$, where $C$ is, a complete intersection of two hyperplanes, a twisted cubic, a conic, or a line, respectively \cite[Table 2]{MM81classtable}.
\begin{lemma}    \label{Am:2.14,2.20,2.22,2.26}
    Let $X$ be a Fano threefold belonging to one of the following families 
    \begin{center}
        \textnumero 2.14, \textnumero 2.20, \textnumero  2.22, \textnumero  2.26.
    \end{center}
    Let $G\subseteq \aut(X)$ be a finite group. Then $\am(X,G)=0$.
\end{lemma}
\begin{proof}
  From Table \ref{table:amitsur}, we recall $\autp(X)=0$. By Lemma \ref{Lemma:autblowup}, we may regard $G\subseteq \aut(Y;C)$. Since $\am(Y,G)=0$, Theorem \ref{imprimitivebcdp} yields $\am(X,G)=0$. 
\end{proof}

\subsection{Blow-ups of a quadric threefold}
Let $Q$ denote a smooth quadric threefold, that is, a member of family \textnumero 1.16. From Table \ref{table:amitsur}, we know that $\autp(Q)=0$ and 
\begin{equation}   \label{eqn:AmQ}
    \am(Q,G)=0
\end{equation}
for every finite subgroup $G\subseteq \aut(Q)$. 
As an application of Theorem \ref{imprimitivebcdp}, we obtain the classification of Amitsur groups for several families with Picard rank $\rho=2$ and $3$. 

\begin{lemma}  \label{Am:2.7,2.13,2.17,2.21,2.23,2.29,2.31}
    Let $X$ be a Fano threefold in any one of the following families
        \begin{center}
       \textnumero 2.7,\; \textnumero 2.13,\; \textnumero 2.17,\; \textnumero 2.21,\; \textnumero 2.23,\; \textnumero 2.29,\; \textnumero 2.31. 
        \end{center}
        Let $G\subseteq \aut(X)$ be a finite group. Then $\am(X,G)=0$. 
\end{lemma}
\begin{proof}
By Table \ref{table:amitsur}, we have $\autp(X)=0$ for all families under consideration except \textnumero 2.21, for which $\autp(X)=\zz/2\zz$.

First, suppose that $X$ belongs to family \textnumero 2.21 and $\autp(X,G)=\zz/2\zz$. Then $\operatorname{rk}\pic(X)^G=1$, so Proposition \ref{Prop:rkPic_G=1} gives $\am(X,G)=0$. In all remaining cases, $\autp(X,G)=0$. Hence, the $G$-action descends to $Q$ and $\pi$ is $G$-equivariant. Theorem \ref{imprimitivebcdp} and \eqref{eqn:AmQ} then yield $\am(X,G)=0$.
\end{proof}

 Let $X$ be a Fano threefold that belongs to one of the families
 \begin{center}
     \textnumero 3.10,\; \textnumero 3.15,\; \textnumero 3.20.
 \end{center} 
Then it is the blow-up of $Q$ in the disjoint union $C:=C_1\sqcup C_2$. Let $\phi: X \to Q$ denote this blow-up, $E_1$ and $E_2$ are exceptional divisors, and write $H:=\phi^*\mathcal{O}_Q(1)$. From Table \ref{table:amitsur} and \cite[Table 3]{MM81classtable}, we have 
\begin{center}
\renewcommand{\arraystretch}{1.2}
\begin{tabular}{c l c}
\hline
Family & $C_1\sqcup C_2\subset Q$ & $\autp(X)$ \\
\hline
\textnumero $3.10$ & Two disjoint conics & $\zz/2\zz$ \\
\textnumero $3.15$ & A disjoint line and conic & $0$ \\
\textnumero $3.20$ & Two disjoint lines & $\zz/2\zz$ \\
\hline
\end{tabular}
\end{center}

The previous lemma yields the following result.

\begin{lemma}  \label{Am:3.10,3.15,3.20}
Let $X$ be a Fano threefold belonging to one of the following families
   \begin{center}
   \textnumero 3.10,\; \textnumero 3.15,\; \textnumero 3.20.
   \end{center}
Let $G\subseteq \aut(X)$ be a finite group. Then $\am(X,G)=0$.
\end{lemma}
\begin{proof}
Suppose first that $X$ belongs to \textnumero 3.10 (resp. \textnumero 3.20) and $\autp(X,G)=\zz/2\zz$. By \cite[Lemmas 4.11, 4.13]{sharma2026actionspicardgroupsmooth}, the nontrivial element of $\autp(X,G)$ exchanges the classes of $E_1,E_2$, while keeping the hyperplane class fixed. Since the total exceptional divisor $E:=E_1\cup E_2 \in \pic(X)^G$, it follows by Lemma \ref{Lemma:autblowup} that $G$ descends to a subgroup of $\aut(Y;C)$ and $\phi$ is $G$-equivariant. From Theorem \ref{imprimitivebcdp} and \eqref{eqn:AmQ}, we obtain $\am(X,G)=0$.

If $\autp(X,G)=0$ in these two cases, or if $X$ is in \textnumero 3.15, then each exceptional divisor is $G$-invariant. It again follows that $\phi$ is $G$-equivariant, and consequently, by Theorem \ref{imprimitivebcdp} and \eqref{eqn:AmQ}, $\am(X,G)=0$.
\end{proof}

Next suppose that $X$ is a Fano threefold in family \textnumero 3.19. Then $X$ is the blow-up of $Q$ in two non collinear points $p$ and $q$. Denote the blow-up map by $\phi:X \to Q$, and the exceptional divisors of $p$ and $q$ by $E_p$ and $E_q$, respectively. 

\begin{lemma}   \label{Am:3.19}
    Let $X$ be a Fano threefold in \textnumero 3.19. Let $G\subseteq \aut(X)$ be a finite group. Then $\am(X,G)=0$.
\end{lemma}
\begin{proof}
    From Table \ref{table:amitsur}, $\autp(X)=\zz/2\zz$. If $\autp(X,G)=\zz/2\zz$, then $\pic(X)^G=\zz[H,E_p+E_q]$. Since $E_p+E_q$ is $G$-invariant, $G$ descends to a subgroup of $\aut(Q;\{p,q\})$, and $\phi$ is $G$-equivariant. From Theorem \ref{imprimitivebcdp} and \eqref{eqn:AmQ}, we have $ \am(X,G)=0$.
    If $\autp(X,G)=0$. Then both $E_p$ and $E_q$ are $G$-invariant, hence, it again follows that $\phi$ is $G$-equivariant. Thus, $ \am(X,G)=0$.
\end{proof}

\section{Blow-ups of projective space}
In this section, we discuss the following families of Fano threefolds
\begin{center}
    \textnumero 2.4,\; \textnumero 2.12,\; \textnumero 2.15,\;  \textnumero 2.19,\; \textnumero 2.25,\; \textnumero 2.27,\; \textnumero 2.28,\; \textnumero 2.30,\; \textnumero 2.35,\;
     \textnumero 3.6,\; \textnumero 3.11,\; \textnumero 3.12,\; \textnumero 3.14,\; \textnumero 3.16,\; \textnumero 3.23,\; \textnumero 3.26,\; \textnumero 3.29,\; \textnumero 3.30.
\end{center}

\medskip
Let $X$ denote a member of any of above listed families and suppose that $\rho(X)=2$. By Mori-Mukai classification, $X$ is obtained by blowing-up $\mathbb{P}^3$ along a smooth curve $C$. Let $H$ denote the hyperplane class on $\pp^3$, and let $\pi : X \to \mathbb{P}^3$ be the blow-up map with exceptional divisor $E$. Then $\pic(X) = \zz [\pi^*H,E]$ and $-K_X=4\pi^*H -E$. 

\smallskip
Unless noted otherwise, let $x_0,x_1,x_2,x_3$ be coordinates on $\pp^3$. 

\begin{remark}  \label{remark:blow-ups-P3}
\normalfont Let $G\subseteq \aut(X)$ be a finite group, and suppose $\autp(X,G)=0$. Then $E$ is $G$-invariant. By Lemma \ref{Lemma:autblowup}, we may assume that $G\subseteq \aut(\pp^3;C)$, and in particular, $\pi$ is $G$-equivariant. It now follows from Theorem \ref{imprimitivebcdp} and the classification of $\am(\pp^3,G)$ discussed in \cite[Example 2.3]{sharma2026amitsurgroupsprimitivefano} that
\begin{equation}   \label{Amitsur-bound-Bl-of-P3}
\am(X,G) = \am(\pp^3,G)  \subseteq \zz/4\zz,
\end{equation}
\end{remark}
and $E$ is $G$-linearizable. For all families considered here other than \textnumero $2.12$,
we have $\autp(X)=0$, so $\pi$ is $G$-equivariant for every
finite subgroup $G\subseteq\aut(X)$. We shall therefore use
this remark without further comment for these families.
For family \textnumero $2.12$, we apply it only when
$\autp(X,G)=0$. In these cases, $\am(X,G)=\langle\overline{\pi^*H}\rangle$, so it suffices to determine the
Amitsur period of $\pi^*H$.

\medskip
   In many of the following examples, we determine the Amitsur period of $\pi^*H$ by instead considering the divisor $D:=\pi^*H-E$. Suppose $E$ is $G$-linearizable, then $D$ is $G$-linearizable if and only if $\pi^*H$ is, and $m(\pi^*H)=m(D)$. Denote it by $m$. Then the upper bounds on $m$ are obtained using Proposition \ref{prop:amitsur-period-congruences}. 
\medskip

A Fano threefold $X$ in family \textnumero $2.4$ is the blow-up of $\pp^3$ along a curve $C$ that is a complete intersection of two cubics \cite[Table 2]{MM81classtable}.

\begin{lemma}   \label{Am:2.4}
Let $X$ be a Fano threefold in family \textnumero 2.4. Let $G\subseteq \aut(X)$ be a finite group. Then $\am(X,G)$ is isomorphic to $\zz/2\zz$ or trivial.
\end{lemma}
\begin{proof}
Let $m$ denote the Amitsur period of $D:=\pi^*H-E$. Using \cite[Lemma 1.22]{prokhorov-fano-threefolds}, we determine
\[
\begin{aligned}
(\pi^*H)^3             &= 1,
&\qquad\qquad
(\pi^*H)^2\cdot E      &= 0, \\
\pi^*H\cdot E^2        &=-9,
&\qquad\qquad
E^3                    &= -54.
\end{aligned}
\]
Then $D^3= 28$. From Proposition \ref{prop:amitsur-period-congruences}, $m$ divides $\gcd(50,12)$=2. This implies that $\am(X,G)\subseteq \zz/2\zz$.

To show that this bound is realized, consider the smooth curve $C$ from \cite[\S 4.5]{calabi} defined by $C:=\{f_1=f_2=0\}$, where
\begin{align*}
 f_1: x_0^3 + x_1^3 +2 (x_2^3+x_3^3) &=0,\\
 f_2:   2(x_0^3-x_1^3) + x_2^3 - x_3^3 &=0. 
\end{align*}
Let $\mathcal{G}:= \langle \sigma , \tau\rangle \cong (\zz/2\zz)^2$, where 
\begin{align*}
 \sigma : ([x_0:x_1:x_2:x_3]) &\mapsto ([x_2:-x_3:x_0:-x_1])\\
   \tau : ([x_0:x_1:x_2:x_3]) &\mapsto ([x_1:x_0:x_3:x_2]).
\end{align*}
We observe that $\sigma$ exchanges the two cubics, while $\tau$ preserves them. Consequently, their intersection $C$ is $\mathcal{G}$-invariant. Since the $\mathcal{G}$-action on $\pp^3$ does not lift to a linear representation. $H$ is not $\mathcal{G}$-linearizable. However, $2H$ is $\mathcal{G}$-linearizble since the induced $\mathcal{G}$-action on $\pp(V^\vee)$ has a linear lift, where $V=H^0(\pp^3,\mathcal{O}_{\pp^3}(2)$. Thus, $\am(\mathbb{P}^3, \mathcal{G})\cong \mathbb{Z}/2\mathbb{Z}$, and from \eqref{Amitsur-bound-Bl-of-P3}, we obtain $\am(X,\mathcal{G}) \cong \mathbb{Z}/2\mathbb{Z}$. 
\end{proof}

Let $X$ be a Fano threefold in family \textnumero 2.12, it is the blow-up of $\pp^3$ along a smooth curve of degree $6$ and genus $3$, which is an intersection of four cubic hypersurfaces \cite[Table 2]{MM81classtable}. We use an alternative description of $X$ from \cite[§5.4]{calabi}, in which $X$ is realized as the intersection of three divisors of bidegree $(1,1)$ in $\pp^3\times\pp^3$; see also \cite{cheltsovs2-12family}. Let $\iota : X \hookrightarrow \mathbb{P}^3 \times \mathbb{P}^3$ denote the natural embedding of $X$, let $H_i:=\operatorname{pr}_i^*\mathcal{O}_{\pp^3}(1)$, where $\operatorname{pr}_i$ is the projection of $\pp^3\times\pp^3$ onto its $i$th factor, $i=1,2$. Then $X :=D_1\cap D_2\cap D_3$ for $D_1, D_2,D_3 \in |H_1+H_2|$, and 
\[
\pic(X)= \zz[\iota^*H_1]\oplus\zz[\iota^*H_2], \quad -K_X=\iota^*(H_1+H_2).
\]

Let $G\subseteq \aut(X)$ be a finite group. Since the $G$-action swaps the two generators of $\pic(X)$ or preserves each of them (\cite[Lemma 4.8]{sharma2026actionspicardgroupsmooth}), Remark \ref{remark:equivarince-of-embedding} allows us to identify $G$ as a subgroup of $\aut(\pp^3\times\pp^3;X)$, so that $\iota$ is $G$-equivariant.

\begin{lemma}  \label{Am:2.12}
    Let $X$ be a Fano threefold in family \textnumero $2.12$. Let $G\subseteq \aut(X)$ be a finite group. Then $\am(X,G)$ is isomorphic to $\zz/4\zz$, or $\zz/2\zz$, or trivial. 
\end{lemma}
\begin{proof}
From Table \ref{table:amitsur}, $\autp(X)=\zz/2\zz$. Suppose first $\autp(X,G)=\zz/2\zz$. Then $\pic(X)^G=\zz[-K_X]$, and we obtain $\am(X,G)=0$ using Proposition \ref{KXlinearizable}.

Now suppose $\autp(X,G)=0$. Let $V:= H^0(\pp^3,\mathcal{O}_{\pp^3}(1))$. From \cite[\S 5.4]{calabi}, we recall the matrices 
\[
A = \begin{pmatrix} 1 & 0 & 0 & 0\\ 0 & \zeta & 0 & 0\\  0 & 0 & \zeta^4 & 0\\ 0 & 0 & 0 & \zeta^2 \end{pmatrix} 
\text{ and } B= \frac{i}{\sqrt{7}} \begin{pmatrix} 1 &\sqrt{2} & \sqrt{2} & \sqrt{2} \\ \sqrt{2} & \zeta^2 + \zeta^5 & \zeta^3 +\zeta^4 & \zeta+\zeta^6 \\ \sqrt{2} & \zeta^3+\zeta^4 & \zeta+\zeta^6 & \zeta^2+\zeta^5 \\ \sqrt{2} & \zeta+\zeta^6 & \zeta^2+\zeta^5 & \zeta^3+\zeta^4  &   \end{pmatrix},
\]
where $\zeta$ is a primitive seventh root of unity. These generate the group $\mathcal{G}_1 \cong\operatorname{SL}_2(\mathbb{F}_7)$.  Let $\overline{A}, \overline{B} \in\operatorname{PSL}_2(\mathbb{F}_7)$ denote projective images of $A$ and $B$, respectively, and set $\overline{\mathcal{G}_1}:=\langle \overline{A}, \overline{B}\rangle$. Then  $\overline{\mathcal{G}_1} \cong \operatorname{PSL}_2(\mathbb{F}_7)$. Furthermore, the group $\overline{\mathcal{G}_1}$ acts on $\pp^3$ while preserving a smooth curve $C_1 \subset \pp^3$ of degree $6$ and genus $3$ \cite{edge-kleingroupfor2.12family}. The blow-up of $\pp^3$ along $C_1$ produces a Fano threefold $X_1$ in family \textnumero 2.12. Note that $\autp(X_1,\overline{\mathcal{G}_1})=0$.

We claim that $\mathcal{O}_{\pp^3}(1)$ is not $\overline{\mathcal{G}_1}$-linearizable, but $\mathcal{O}_{\pp^3}(2)$ is. Suppose that the projective representation $\overline{\mathcal{G}_1} \to \operatorname{PGL}(V)$ lifts to a linear representation. Then there is a map $s: \overline{\mathcal{G}_1} \to \mathcal{G}_{\mathcal{O}_{\pp^3}(1)}$, where $\mathcal{G}_{\mathcal{O}_{\pp^3}(1)}$ is the lifting group of $\mathcal{O}_{\pp^3}(1)$. Since $\overline{\mathcal{G}}_1$ is simple, $s(\overline{\mathcal{G}}_1)\subseteq \mathcal{G}_1$, which splits the following sequence
\[
1\to \{\pm I_4\} \to \mathcal{G}_1 \to \overline{\mathcal{G}_1} \to 1
\]
Now since $B^2=-I$, $\overline{B}$ has order $2$, and so $s(\overline{B})$ also has order $2$. However both of its preimages satisfy $B^2=-I_4$ or $(-B)^2=-I_4$, and hence have order $4$, a contradiction. Thus, $\mathcal{O}_{\pp^3}(1)$ is not $\overline{\mathcal{G}}_1$-linearizable. On the other hand, via the Veronese embedding $\nu: \pp^3\hookrightarrow \pp(\operatorname{Sym}^2(V))=\pp^9$, $\overline{\mathcal{G}_1}$ induces an action on $\pp^9$ and $\nu$ is $\overline{\mathcal{G}_1}$-equivariant. Since 
\[
\operatorname{Sym}^2(I_4) = \operatorname{Sym}^2(-I_4)= I_{10},
\]
the kernel $\{\pm I_4\}$ acts trivially on $\operatorname{Sym^2(V)}$. Therefore, the $\overline{\mathcal{G}_1}$-action lifts to a linear representation on $\operatorname{Sym}^2(V)$. Hence, $\mathcal{O}_{\pp^9}(1)$, and consequently, $\nu^*\mathcal{O}_{\pp^9}(1)=\mathcal{O}_{\pp^3}(2)$ and $\pi^*\mathcal{O}_{\pp^3}(2)$ are both $\overline{\mathcal{G}_1}$-linearizable. Thus, 
\[
\am(X_1,\overline{\mathcal{G}_1}) \cong \zz/2\zz.
\]

From \cite[\S 2.2]{cheltsovs2-12family}, we recall the group $\mathcal{G}_2:=\langle B,C\rangle \subset \operatorname{GL}_4(\cc)$, where
\[
B = \begin{pmatrix}
    0 & i & 0 & 0\\
    1 & 0 & 0 & 0\\
    0 & 0 & 0 & -i\\
    0 & 0 & 1 & 0
\end{pmatrix}, \hspace{0.5cm} 
C = \begin{pmatrix}
    0 & 0 & 0 & 1\\
    0 & 0 & 1 & 0\\
    0 & i & 0 & 0\\
    -i & 0 & 0 & 0
  \end{pmatrix},
\]
where $i$ is a primitive fourth root of unity. Let $\overline{\mathcal{G}_2}$ denote its projective image. Then the authors show that there is a smooth $\overline{\mathcal{G}_2}$-invariant curve $C_2 \subset \pp^3$, whose blow-up gives a Fano threefold $X_2$ in family \textnumero 2.12. So, we can assume that $\overline{\mathcal{G}_2}$ lifts to a subgroup of $\aut(X_2)$ and $\pi$ is $\overline{\mathcal{G}_2}$-equivariant. Note that $BC = i CB$, so the $\overline{\mathcal{G}_2}$-action on $\pp^3$ does not lift to a linear representation. Thus, $\mathcal{O}_{\pp^3}(1)$ is not $\overline{\mathcal{G}_2}$-linearizable. 

Now let $B':=\operatorname{Sym}^2(B)$, $C':=\operatorname{Sym^2(C)} \in \operatorname{GL}_{10}(\cc)$, and note that $B'C'= -C'B'$. So, the induced $\overline{\mathcal{G}_2}$-action on $\mathbb P(\operatorname{Sym}^2(V))$ does not lift to a linear representation on $\operatorname{Sym}^2(V)$. Thus, $\mathcal{O}_{\pp^3}(2)$, and consequently, $\pi^*\mathcal{O}_{\pp^3}(2)$ is not $\overline{\mathcal{G}_2}$-linearizable. It now follows from \eqref{Amitsur-bound-Bl-of-P3} that 
\[
\am(X_2,\overline{\mathcal{G}_2})=\am(\pp^3,\overline{\mathcal{G}_2})\cong \zz/4\zz.
\]
\end{proof} 

A Fano threefold $X$ in family \textnumero $2.15$ is the blow-up of $\pp^3$ along a curve $C := Q \cap \mathscr{C}$, where $Q$ is a quadric and $\mathscr{C}$ is a cubic. So, $\deg C= 6$ and $g(C)=4$. There are two cases--- a) $Q$ is smooth, b) $Q$ is singular but reduced (\cite[Table 2]{MM81classtable}). 

\begin{lemma}  \label{Am:2.15}
     Let $X$ be a Fano threefold in family \textnumero 2.15. Let $G\subseteq \aut(X)$ be a finite group. Then $\am(X,G)=0$.
\end{lemma} 
\begin{proof}
We determine the Amitsur period of $\pi^*H$ as follows. Set $D:\pi^*H-E$, then using \cite[Lemma 1.22]{prokhorov-fano-threefolds},
\begin{align*}
    D^2\cdot K_X = (\pi^*H-E)^2\cdot(-4\pi^*H+E) &= -4(\pi^*H)^3 + 9(\pi^*H)^2E-6(\pi^*H)E^2 +E^3\\
    &= -4+36-30=2.
\end{align*}
Then by Proposition \ref{prop:amitsur-period-congruences}, $2m$ divides $2$, or simply, $m=1$. This implies that $D$, hence, $\pi^*H$ is $G$-linearizable. Thus, we get $\am(X,G)=0$. 
 
\end{proof}

Let $X$ denote a Fano threefold in \textnumero $2.19$, then it is the blow-up of a Fano threefold $Y$ from family \textnumero $1.14$ in a line \cite[Table 2]{MM81classtable}. Here we use an alternative description of $X$ following \cite[\S 4.4]{calabi}: $X$ is realized as the blow-up of $\pp^3$ along a curve $C$ of bidegree $(2,3)$ in $\mathbb{P}^1_{x_0,x_1}\times \mathbb{P}^1_{y_0,y_1}\hookrightarrow \mathbb{P}^3$. Then $\deg C =5$ and $g(C)=2$. 

\begin{lemma}   \label{Am:2.19}
    Let $X$ be a Fano threefold in family \textnumero $2.19$. Let $G \subseteq \aut(X)$ be a finite group. Then $\am(X,G)$ is isomorphic to $\zz/2\zz$ or trivial.
\end{lemma}
\begin{proof}
Let $D:=\pi^*H-E$. Using  \cite[Lemma 1.22]{prokhorov-fano-threefolds}, we determine 
\[
D^3 = (\pi^*H)^3-E^3+3\pi^*H E^2 = 1+22-15 = 8
\]
So, by Proposition \ref{prop:amitsur-period-congruences}, $m$ divides $\gcd(12,10)=2$. Thus, $\am(X,G)\subseteq \zz/2\zz$.
We claim that the nontrivial possibility for $\am(X,G)$ can be realized.

Let $\mathcal{G}:=\langle \sigma , \tau \rangle \subset \aut(\pp^1\times\pp^1)$ generated by 
the involutions 
\begin{align*}
   \sigma : ([x_0:x_1],[y_0:y_1]) &\mapsto ([-x_0:x_1],[y_0:y_1]) \\
  \tau :  ([x_0:x_1],[y_0:y_1]) &\mapsto ([x_1:x_0],[y_1:y_0]).
\end{align*}
Under the Segre embedding $\nu: \pp^1_{x_0,x_1}\times\pp^1_{y_0,y_1} \hookrightarrow \pp^3$, these induce the involutions 
\begin{align*}
   \sigma': [z_0:z_1:z_2:z_3] &\mapsto [-z_0:-z_1:z_2:z_3]\\
    \tau':[z_0:z_1:z_2:z_3] &\mapsto [z_3:z_2:z_1:z_0], 
\end{align*}
where $\pp^3$ is coordinated by $z_0:=x_0y_0, z_1:=x_0y_1,z_2:=x_1y_0,z_3:=x_1y_1$. Let $\mathcal{G}':=\langle \sigma' , \tau' \rangle$. Then $\mathcal{G}$ induces the subgroup $\mathcal{G}'$ of $\aut(\pp^3)$ such that $\nu$ is $(\mathcal{G},\mathcal{G}')$-equivariant. 

Let $C$ be the curve defined by the equation 
\begin{equation*}
\begin{split}
x_0^2(a_0y_0^3 + a_1y_1^3 + a_2y_0^2y_1 + a_3y_0y_1^2) + x_1^2(a_1y_0^3 + a_0y_1^3 + 
a_3y_0^2y_1 + a_2y_0y_1^2)=0.
\end{split}
\end{equation*}
It is $\mathcal{G}$-invariant. The smoothness and irreducibility of $C$ can be verified for the choice of coefficients $(a_0,a_1,a_2,a_3)=(2,2,2,-1)$. Thus, its image $\nu(C)\subset \pp^3$ is $\mathcal{G}'$-invariant. 

Observe that the $\mathcal{G}'$-action on $\pp^3$ does not lift to a linear representation to $\operatorname{GL}_4(\cc)$, so $\mathcal{O}_{\mathbb{P}^3}(1)$ is not $\mathcal{G}'$-linearizable. But $\mathcal{O}_{\pp^3}(2)$ is $\mathcal{G}'$-linearizable. Thus, $\am(\pp^3,\mathcal{G}')\cong \zz/2\zz$. Together with \eqref{Amitsur-bound-Bl-of-P3}, we obtain
\[
\am(X,\mathcal{G}') \cong \mathbb{Z}/2\mathbb{Z}.
\]
\end{proof}

Let $X$ be a Fano threefold in family \textnumero 2.25. In the Mori-Mukai description, $X$ is the blow-up of $\pp^3$ in an elliptic curve $C$ which is an intersection of two quadrics. 

\begin{lemma}   \label{Am:2.25}
    Let $X$ be a Fano threefold in \textnumero $2.25$. Let $G\subseteq \aut(X)$ be a finite group. Then $\am(X,G)$ is isomorphic to $\zz/4\zz$, $\zz/2\zz$, or trivial.
\end{lemma}
\begin{proof}
By Remark \ref{remark:blow-ups-P3}, we assume that $G\subseteq \aut(\pp^3;C)$. It would be sufficient to determine possible groups $\am(\pp^3,G)$.

First, consider the group $\mathcal{G}_1\cong (\zz/4\zz)^2 $ generated by 
\begin{align*}
      \sigma :  [x_0:x_1:x_2:x_3] &\mapsto [x_0:i x_1:-x_2:-i x_3]\\
      \tau :  [x_0:x_1:x_2:x_3] &\mapsto [x_3:x_0:x_1:x_2],
\end{align*}
and let $C\subset \pp^3$ be the intersection of quadrics
\[
\begin{cases}
    x_0^2 + x_2^2 + x_1x_3 =0\\
    x_1^2 + x_3^2 + x_0x_2 =0.
\end{cases}
\]
Note that $\sigma$ preserves each hypersurface, while $\tau$ interchanges them, so $C$ is $\mathcal{G}_1$-invariant. It can also be verified to be both smooth and irreducible \cite{magma}. If $A,B \in \operatorname{GL}_4(\cc)$ denote the matrix lifts of $\sigma$ and $\tau$, respectively, then $AB=iBA$. So, the $\mathcal{G}_1$-action on $\pp^3$ does not lift to a linear representation, and $H$ is not $\mathcal{G}_1$-linearizable. In addition, $2H$ is not $\mathcal{G}_1$-linearizable since $\operatorname{Sym}^2A$ and $\operatorname{Sym}^2 B$ anticommute. Thus, the induced $\mathcal{G}_1$-action on $\pp^{9}$ does not lifts to a $10$-dimensional linear representation. Hence, $\am(X,\mathcal{G}_1)\cong \zz/4\zz$.

Now consider the group $\mathcal{G}_2\cong (\zz/2\zz)^2$ generated by involutions 
\begin{align*}
[x_0:x_1:x_2:x_3] &\mapsto [x_1:x_0:x_3:x_2]\\
[x_0:x_1:x_2:x_3] &\mapsto  [x_0:-x_1:x_2:-x_3], \end{align*}
and curve $C$ defined by the intersection
\[
\begin{cases}
   x_0^2+x_1^2+2 (x_2^2+x_3^2) =0\\
   x_0^2 - x_1^2 +x_2^2-x_3^2  =0.
\end{cases}
\]
Then $C$ is $\mathcal{G}_2$-invariant since both quadrics are preserved by $\mathcal{G}_2$, and it is also smooth and irreducible \cite{magma}. The $\mathcal{G}_2$-action on $\pp^3$ does not lift to a linear representation, but its action on $\pp^9$ does have a linear lift, therefore, $\am(X,\mathcal{G}_2)=\zz/2\zz$.

\end{proof}

Let $X$ be a Fano threefold in family \textnumero $2.27$. Mori-Mukai describe $X$ as the blow-up of $\pp^3$ in a twisted cubic $C$ \cite[Table 3]{MM81classtable}, so $\deg C =3$ and $g(C)=0$.

\begin{lemma}   \label{Am:2.27}
Let $X$ be a Fano threefold in family \textnumero 2.27. Let $G\subseteq \aut(X)$ be a finite group. Then $\am(X,G)$ is isomorphic to $\zz/2\zz$ or trivial.
\end{lemma}
\begin{proof} 
We first show that $\am(X,G)\subseteq \zz/2\zz$. Set $D:=\pi^*H-E$. Then 
\[
D^3 = (\pi^*H)^3 -E^3 + 3\pi^*HE^2 = 1+10-9 =2,
\]
and Proposition \ref{prop:amitsur-period-congruences} implies that $m$ divides $2$. Thus, by Remark \ref{remark:blow-ups-P3} $\am(X,G) = \langle \pi^*H\rangle \subseteq\zz/2\zz$. 

The curve $C$ can be parametrized by the image of the embedding $\mathbb{P}^1 \hookrightarrow \mathbb{P}^3$, equivalently, it can be given by the intersection 
\[
x_0x_3=x_1x_2,\; x_1^2=x_0x_2,\; x_2^2=x_1x_3.
\]
This is smooth and irreducible. Consider $\mathcal{G} \subset \aut(\pp^3)$ generated by the following involutions
\begin{equation}   \label{otherdihedralorder8}
\begin{split}
    [x_0:x_1:x_2:x_3]  &\mapsto  [x_0:-x_1:x_2:-x_3]\\
     [x_0:x_1:x_2:x_3] &\mapsto  [x_3:x_2:x_1:x_0]
\end{split}
\end{equation}

Then $C$ is $\mathcal{G}$-invariant. Blowing-up $\pp^3$ in $C$ now gives a Fano threefold $X$ in family \textnumero 2.27. The $\mathcal{G}$-action on $\pp^3$ does not lift to a linear representation, whereas its induced action on $\pp^9$ via the Veronese embedding does. So, $H$ is not $\mathcal{G}$-linearizable but $2H$ is. Hence, $\am(\pp^3,\mathcal{G})\cong \zz/2\zz$. From \eqref{Amitsur-bound-Bl-of-P3}, it follows that 
\[
\am(X,\mathcal{G})\cong\mathbb{Z}/2\mathbb{Z}.
\]
\end{proof}

Let $X$ be a Fano threefold in \textnumero $2.28$, then $X$ is the blow-up of $\pp^3$ is a smooth planar cubic $C:=\mathscr{H}\cap S$, where $\mathscr{H}$ is a hyperplane and $S$ is a cubic surface in $\mathbb{P}^3$.

\begin{lemma}  \label{Am:2.28}
     Let $X$ be a Fano threefold in family \textnumero 2.28. Let $G\subseteq \aut(X)$ be a finite group. Then $\am(X,G)=0$.
\end{lemma} 
\begin{proof}
By Remark \ref{remark:blow-ups-P3}, we assume that $G\subseteq \aut(\pp^3;C)$. Set $V:=H^0(\pp^3,\mathcal{O}_{\pp^3}(1))$. Since $C$ is $G$-invariant and $\mathscr{H}$ is the unique hyperplane containing $C$, $\mathscr{H}$ is also $G$-invariant. As a consequence, the natural $G_H$-representation on $V$ contains a $G_H$-subrepresentation of dimension $1$. By Proposition \ref{Gsubrepn_result}, $H$ is $G$-linearizable on $\pp^3$, consequently, $\pi^*H$ is $G$-linearizable on $X$. Since $E$ is $G$-linearizable by Remark \ref{remark:blow-ups-P3}, we get $\am(X,G)=0$.
\end{proof}   

Let $X$ be a Fano threefold in family \textnumero $2.30$. Then $X$ is the blow-up of $\pp^3$ in a conic $C$, which is an intersection of a quadric and a hyperplane, so $\deg C = 2$ and $g(C)=0$. 
\begin{lemma}  \label{Am:2.30}
    Let $X$ be a smooth Fano threefold in family \textnumero 2.30. Let $G\subseteq \aut(X)$ be a finite group. Then $\am(X,G)=0$.
\end{lemma} 
\begin{proof}
Let $D:= \pi^*H-E$. Using \cite[Lemma 1.22]{prokhorov-fano-threefolds}, we determine
\begin{align*}
    D^2\cdot K_X &= -4(\pi^*H)^3 + 9(\pi^*H)^2E-6(\pi^*H)E^2 +E^3\\ 
    &= -4+12 -6 = 2.
\end{align*}
By Proposition \ref{prop:amitsur-period-congruences}, $2m$ divides $2$, thus, $m=1$. Since $E$ is $G$-linearizable by Remark \ref{remark:blow-ups-P3}, we get $\am(X,G)=0$. 
\end{proof}

A Fano threefold $X'$ in \textnumero 3.18 is the blow-up of a Fano threefold $X$ from \textnumero 2.30 along the strict transform $\widetilde{\ell}\subset X$ of a line $\ell\subset \pp^3$ disjoint from the conic blown-up to get $X$. Let $\phi: X'\to X$ denote this blow-up. 

\begin{coro}    \label{Am:3.18}
    Let $X'$ be a Fano threefold in family \textnumero 3.18. Let $G\subseteq \aut(X')$ be a finite group. Then $\am(X',G)=0$. 
\end{coro}
\begin{proof}
    Since $\autp(X')=0$, $\phi$ is $G$-equivariant using Lemma \ref{Lemma:autblowup}. Then by Theorem \ref{imprimitivebcdp} and the previous lemma, it follows that $\am(X',G)=0$. 
\end{proof}

A Fano threefold $X$ in family \textnumero $2.35$ is the blow-up of $\pp^3$ at a point $p$ \cite{MM81classtable}. The threefold $X$ is toric, so its Amitsur groups can be computed using the methods of \cite{duncan2025numericalamitsurgroup}; for example, see
\cite[\S3]{sharma2026amitsurgroupsprimitivefano}. We give an alternative proof using Proposition \ref{Gsubrepn_result}.

\begin{lemma}  \label{Am:2.35}
    Let $X$ be a Fano threefold in family \textnumero 2.35. Let $G\subseteq \aut(X)$ be a finite group. Then $\am(X,G)=0$.
\end{lemma}
\begin{proof}
    By Remark \ref{remark:blow-ups-P3}, we assume that $G\subseteq \aut(\pp^3;\{p\})$. Set $V:=H^0(\pp^3,\mathcal{O}_{\pp^3}(1))$. Since $G$ fixes a point $p$, the natural $G_{H}$-representation on $V^\vee$ splits into a direct sum of $G_{H}$-subrepresentations of dimensions $1$ and $3$. By Proposition \ref{Gsubrepn_result}, $H$ is $G$-linearizable. Since $E$ is $G$-linearizable (Remark \ref{remark:blow-ups-P3}), we get $\am(X,G)=0$.
\end{proof}

A Fano threefold $X'$ in any of the families 
\begin{center}
 \textnumero 3.11, \textnumero 3.14, \textnumero 3.16, \textnumero 3.23, \textnumero 3.26, \textnumero 3.29, \textnumero 3.30
\end{center}
can be obtained by blowing up the Fano threefold $X$ from family \textnumero 2.35 in a curve \cite[Table 3]{MM81classtable}. Note that $X'$ is toric if it belongs to families \textnumero 3.26, \textnumero 3.29, or \textnumero 3.30, and the Amitsur groups for these families have been classified in \cite{sharma2026amitsurgroupsprimitivefano}. Let $\phi: X' \to X$ denote the blow-up map with the exceptional divisor $F$. Then $\pic(X')=\phi^*\pic(X)\oplus \zz[F]$. 

\begin{coro}    \label{Am:blowup_of_2.35}
Let $X'$ be a Fano threefold belonging to deformation families 
 \begin{center}
    \textnumero 3.11,\; \textnumero 3.14,\; \textnumero 3.16,\; \textnumero 3.23,\; \textnumero 3.26,\; \textnumero 3.29,\; \textnumero 3.30.
\end{center}
Let $G \subseteq \aut(X')$ be a finite group. Then $Am(X',G)=0$. 
\end{coro}
\begin{proof}
  In all these cases, we have $\autp(X')=\autp(X)=0$, so we may assume that $\phi$ is $G$-equivariant. By Theorem \ref{imprimitivebcdp}, $\am(X',G)=\am(X,G)$. From the previous Lemma, it follows that $\am(X',G)=0$. 
\end{proof}

Suppose now that $X$ denotes a Fano threefold in family \textnumero 2.33. Then $X$ is the blow-up of $\mathbb{P}^3$ along a line $\ell$. It also admits a toric description, and it was shown in \cite[\S 3]{sharma2026amitsurgroupsprimitivefano} that
\[
\am(X,G)\subseteq \zz/2\zz
\]
for every finite subgroup $G\subseteq \aut(X)$, where the nontrivial possibility is realizable.

We consider Fano threefolds that can be obtained by blowing up $X$ along a curve $C$:
\begin{center}
   \textnumero 3.6,\; \textnumero 3.12,\; \textnumero 3.25. 
\end{center} 
Let $X'$ denote a member of these families. If $X'$ belongs to family \textnumero $3.25$, it has a toric description and $\am(X',G)$ has been classified in \cite{sharma2026amitsurgroupsprimitivefano}, therefore we will forego its discussion.   

\medskip
Suppose that $X'$ belongs to \textnumero 3.6 or \textnumero 3.12, then it is the blow-up of $X$ in an elliptic curve or a twisted cubic, respectively. Let $\phi: X' \to X$ denote the blow-up map and $F$ be the exceptional divisor. Then $\pic(X')=\phi^*\pic(X)\oplus \zz[F]$. 

We also recall that $\autp(X')=\autp(X)=0$ (Table \ref{table:amitsur}). Let $G\subseteq \aut(X')$ be a finite group.  By Lemma \ref{Lemma:autblowup}, we may assume that $G$ is a lift of a subgroup of $\aut(X;C)$, and consequently, $\phi$ is $G$-equivariant. From Theorem \ref{imprimitivebcdp}, 
\[
\am(X',G) = \am(X,G) \subseteq \zz/2\zz
\] 
For these two families, we now verify whether $\am(X',G)=\zz/2\zz$ can be realized for some Fano threefold $X'$ and a finite group $G\subseteq \aut(X')$. 

\begin{lemma}   \label{Am:3.6,3.12}
    Let $X'$ be a Fano threefold belonging to \textnumero 3.6 or \textnumero 3.12. Let $G\subseteq \aut(X')$ be a finite group. Then $\am(X',G)$ is isomorphic to $\zz/2\zz$ or trivial.
\end{lemma}
\begin{proof}
First, suppose that $X'$ is in family \textnumero $3.6$, then it is the blow-up of $\mathbb{P}^3$ along the disjoint union of a line, $\ell$ and a quartic elliptic curve, $C$. In this notation, we write $X:=\operatorname{Bl}_{\ell}(\pp^3)$. Let $\mathcal{G} :=\langle \sigma , \tau\rangle $, where 
\begin{equation}      \label{eqn:group-for-blow-up-2.33}
\begin{split}
 \sigma \colon [x_0:x_1:x_2:x_3] &\mapsto [x_0:-x_1:x_2:-x_3] \\
 \tau \colon [x_0:x_1:x_2:x_3] &\mapsto [x_1:x_0:x_3:x_2].
 \end{split}
\end{equation}
A quartic elliptic curve in $\pp^3$ is given by the intersection of two quadrics 
\[
\begin{cases}
    x_0^2+x_1^2+2(x_2^2+x_3^2) =0,\\
    2(x_0^2 - x_1^2) +x_2^2-x_3^2=0,
\end{cases} 
\]
and let the line be given by $\ell:=\{ x_2=x_3=0\} \subset \pp^3$. Then both $C$ and $\ell$ are $\mathcal{G}$-invariant, so $C \sqcup \ell$ is also $\mathcal{G}$-invariant. From \cite[Example 2.3]{sharma2026amitsurgroupsprimitivefano}, we know that $\am(\pp^3,\mathcal{G})=\zz/2\zz$. It now follows that 
\[
\am(X',\mathcal{G})=\am(X,\mathcal{G}) = \am(\pp^3,\mathcal{G}) \cong \zz/2\zz.
\]

If $X'$ is in family \textnumero $3.12$, it is the blow-up of $\pp^3$ in the disjoint union of a line $\ell$ and a twisted cubic $C$. Let $\ell:=\{ x_2=x_3=0  \}$ and $C$ be the given by the intersection 
\[
x_0x_1=x_2x_3, \; x_2^2 = x_1x_3, \; x_3^2 = x_0x_2.
\] 
Then the group $\mathcal{G}\subset \aut(\pp^3)$ generated by involutions in \eqref{eqn:group-for-blow-up-2.33} keeps both $\ell$ and $C$ invariant. Thus, we again conclude
\[
\am(X',\mathcal{G}) = \am(X,\mathcal{G}) = \am(\pp^3,\mathcal{G})\cong \zz/2\zz.
\]
\end{proof}

\section{Divisors} \label{divisors}
We discuss the families
\begin{center}
    \textnumero 3.3, \; \textnumero 3.17,\;  \textnumero $4.1$.
\end{center}
\medskip
A Fano threefold $X$ in these families can be realized as a hypersurface in a fourfold that is a product of projective spaces.

Consider families \textnumero 3.3 and \textnumero 3.17. Their members are hypersurfaces of tridegrees $(1,1,2)$ and $(1,1,1)$ in $Y:=\mathbb{P}^1 \times \mathbb{P}^1 \times \mathbb{P}^2$, respectively \cite[Table 3]{MM81classtable}. Let $\operatorname{pr}_i$ denote the projection of $Y$ onto its $i$th factor, and set $H_i:= \operatorname{pr}_i^*\mathcal{O}(1)$. Let $\iota : X \to Y$ be the natural embedding. Proposition \ref{thm:lefschetz} gives
$ \pic(X)= \zz[\iota^*H_1]\oplus[ \iota^*H_2]\oplus[\iota^*H_3]$. 

For both families, $\autp(X)=\zz/2\zz$, and $G$-action on $\pic(X)$ fixes $\iota^*H_3$ and either fixes or swaps $\iota^*H_1$ and $\iota^*H_2$. Therefore, $\operatorname{pr}_3$ is $G$-equivariant. 
Let $G\subseteq \aut(X)$ be a finite group. By Remark \ref{remark:equivarince-of-embedding}, we may regard $G\subseteq \aut(Y;X)$ as a subgroup, so that $\iota$ is $G$-equivariant. By Proposition \ref{KXlinearizable}, we have 
\begin{equation}  \label{relations:3.3,3.17}
    \overline{3\iota^*H_3} =0, \quad \overline{-K_X}=0 
\end{equation}
in $\am(X,G)$.

\medskip
First, let $X$ denote a Fano threefold in \textnumero 3.3. Since $\iota$ is $G$-equivariant, Proposition \ref{KXlinearizable} implies  
\begin{align}  \label{relations:3.3}
    \overline{\iota^*(H_1+H_2+H_3)} =0, \qquad
    \overline{\iota^*H_3}=0 ,
\end{align}
where we use $-K_X=\iota^*(H_1+H_2+H_3)$ and $\iota^*H_3=\iota^*(-K_Y)+2K_X$.

\medskip
Let $([x_0:x_1],[y_0:y_1],[z_0:z_1:z_2])$ be coordinates on $Y:=\pp^1\times\pp^1\times\pp^2$.

\begin{lemma}   \label{Am:3.3}
    Let $X$ be a Fano threefold in family \textnumero $3.3$. Let $G\subseteq \aut(X)$ be a finite group. Then $\am(X,G)$ is isomorphic to $\zz/2\zz$ or trivial. 
\end{lemma}
\begin{proof} 
Suppose that $\autp(X,G)=\zz/2\zz$. Then $\pic(X)^G= \zz[\iota^*(H_1+H_2)]\oplus\zz[\iota^*H_3]$. From \eqref{relations:3.3,3.17} and \eqref{relations:3.3}, we get
\begin{align*}  
\begin{split}
\overline{\iota^*H_3} = 0, \quad \overline{\iota^*(H_1+H_2)}=0. 
\end{split}
\end{align*}
Thus, $\am(X,G)=0$.

Now suppose $\autp(X,G)=0$. Then $G$ fixes each factor of $Y$, so $\operatorname{pr}_i$ is $G$-equivariant for $1\leq i \leq 3$. By Proposition \ref{KXlinearizable}, $2\iota^*H_1$, $2\iota^*H_2$ are $G$-linearizable. Together with the relations in \eqref{relations:3.3,3.17} and \eqref{relations:3.3},  
\begin{equation}  \label{Amitsur-bound-3.3}
  \am(X,G)= \langle \overline{\iota^*H_1}\rangle = \langle \overline{\iota^*H_2}\rangle \subseteq \zz/2\zz. 
\end{equation}
We show that there is a Fano threefold $X$ in this family and a finite group $G\subseteq \aut(X)$ for which this upper bound is achieved. 

Consider the group $\mathcal{G}\cong (\zz/2\zz)^2$ generated by
\begin{equation}    \label{eqn:group-for3.3}
\begin{split}
    ([x_0:x_1],[y_0:y_1],[z_0:z_1:z_2]) &\mapsto  ([x_0:-x_1],[y_0:-y_1],[z_0:-z_1:z_2]) \\
    ([x_0:x_1],[y_0:y_1],[z_0:z_1:z_2]) &\mapsto ([x_1:x_0],[y_1:y_0],[z_2:z_1:z_0])
\end{split}
\end{equation}
and a hypersurface $X$ defined by the equation
\[
\begin{aligned}
x_0y_0(az_0^2 + bz_1^2 + cz_2^2 + dz_0z_2) + x_0y_1z_0z_1+ \\
x_1y_0z_1z_2 + x_1y_1(cz_0^2 + bz_1^2 + az_2^2 + dz_0z_2) =0
\end{aligned}
\]
where $(a,b,c,d)=(1,2026,100,4)$. The equation is $\mathcal{G}$-invariant, and defines a smooth and irreducible hypersurface in $Y$ which can be verified using \cite{magma}. Hence, $X$ is a smooth Fano threefold in family \textnumero 3.3 and $\mathcal{G}$ acts faithfully on $X$.  

Finally, we observe that the induced $\mathcal{G}$-action on $\pp^1_x$ and $\pp^1_y$ does not lift to a linear representation, therefore, $\iota^*H_1$ and $\iota^*H_2$ are not $\mathcal{G}$-linearizable. Combined with \eqref{Amitsur-bound-3.3}, we obtain $\am(X,\mathcal{G})\cong \zz/2\zz$. 

\end{proof}

Let $X$ be a Fano threefold in family \textnumero $3.17$. Then $-K_X = \iota^*H_1 + \iota^*H_2+2\iota^*H_3$ and 
\begin{equation*}
    \iota^*H_3 = \iota^*(-K_{Y})+2K_X, \qquad
\iota^*H_1+\iota^*H_2+\iota^*H_3 = \iota^*(-K_{Y})+K_X.
\end{equation*}
Combining the $G$-equivariance of $\iota$,  Proposition \ref{KXlinearizable}, and \eqref{relations:3.3,3.17} now imply that 
\begin{equation}  \label{Amitsur-relations-for-3.17}
    \overline{\iota^*H_3}=0, \qquad \overline{\iota^*(H_1+H_2+H_3)}=0.
\end{equation}

\begin{lemma}   \label{Am:3.17}
    Let $X$ be a Fano threefold in \textnumero $3.17$. Let $G\subseteq \aut(X)$ be a finite group. Then $\am(X,G)$ is isomorphic to $\zz/2\zz$ or trivial.
\end{lemma}
\begin{proof}
Suppose first that $\autp(X,G)=\zz/2\zz$. Then $\pic(X)^G=\zz[\iota^*(H_1+H_2),\iota^*H_3]$. The
relations in \eqref{Amitsur-relations-for-3.17} yield $\overline{\iota^*(H_1+H_2)}=0$. Thus, $\am(X,G)=0$. 

Suppose $\autp(X,G)=0$. Then for $1\leq i \leq 3$, $\iota^*H_i$ is $G$-invariant, and so, $\operatorname{pr}_i$ is $\operatorname{G}$-equivariant. Consequently
\[
2\overline{\iota^*H_1}=2\overline{\iota^*H_2}=0.
\]
Combined with relations in \eqref{Amitsur-relations-for-3.17}, we get
\[
\am(X,G)=  \langle\overline{\iota^*H_1}\rangle = \langle\overline{\iota^*H_2}\rangle \subseteq \zz/2\zz. 
\]

Consider the group $\mathcal{G}\cong (\zz/2\zz)^2$ defined in \eqref{eqn:group-for3.3} and the hypersurface $X$ defined by the equation 
\[
(x_0y_0z_1 + x_1y_1z_1) +2(x_0y_1z_0+ x_1y_0z_2) +3(x_0y_1z_2+x_1y_0z_0)=0.
\]
This equation is $\mathcal{G}$-invariant, and we verify it to be smooth and irreducible; for example using \cite{magma}. Hence, $X$ is a member of \textnumero 3.17. Since the $\mathcal{G}$-action on $\pp^1_x$ and $\pp^1_y$ does not lift to a homomorphism to $\operatorname{GL}_2(\cc)$, $\iota^*H_1$ and $\iota^*H_2$ are not $\mathcal{G}$-linearizable. Thus, $\am(X,\mathcal{G})\cong \zz/2\zz$. 
\end{proof}

\medskip
For family \textnumero $4.1$, set $Y:=(\pp^1)^4$.
Let $\operatorname{pr}_i : Y\to\pp^1$ denote the
projection onto the $i$-th factor, and set
$H_i:=\operatorname{pr}_i^*\mathcal{O}_{\pp^1}(1)$,
$1\leq i\leq 4$. A Fano threefold $X$ in family \textnumero $4.1$ is a member of $|H_1+H_2+H_3+H_4|$. Let $\iota : X \hookrightarrow Y$ be the natural embedding. 

We have  $\pic(X)=\oplus_{i=1}^4 \zz[\iota^*H_i]$, and $-K_X = \iota^*(H_1+H_2+H_3+H_4)$. Recall $\autp(X)=S_4$, so a finite group $G\subseteq \aut(X)$ acts on $\pic(X)$ by permuting its generators. Remark \ref{remark:equivarince-of-embedding} now allows us to assume that $G$-action extends to $Y$ and $\iota$ is $G$-equivariant. 
We identify $\autp(X,G)$ with $\autp(Y,G)$, viewed as a subgroup of $S_4$. Up to conjugacy, their possibilities are in the table below. 
\begin{table}[h]
\centering
\renewcommand{\arraystretch}{1.4}
\begin{tabular}{c|c}
\hline
$\autp(X,G)$ & Action on $\{H_1,H_2,H_3,H_4\}$ \\
\hline
$0$ 
& $\{1\}$ \\
\hline
$\zz/2\zz$
& $\langle(12)\rangle$ \\
\hline
$(\zz/2\zz)^2$
& $\langle(12)(34)\rangle,\;\langle(12),(34)\rangle$ \\
\hline
$\zz/3\zz,\;S_3$
& $\langle(123)\rangle,\;\langle(12),(123)\rangle$ \\
\hline
$\begin{gathered}
\zz/4\zz,\;(\zz/2\zz)^2,\\
D_8,\;A_4,\;S_4
\end{gathered}$
&
$\begin{gathered}
\langle(1234)\rangle,\;
\langle(12)(34),(13)(24)\rangle,\\
\langle(1234),(13)\rangle,\;A_4,\;S_4
\end{gathered}$ \\
\hline
\end{tabular}
\label{table:subgroups-S4}
\end{table}
Here $(ijk)$ denotes the permutation that sends $H_i\mapsto H_j\mapsto H_j\mapsto H_k$, fixing the remaining generators. We use the analogous notation for other permutations. 

We organize the proof according to the rows of the table, since the subgroups in the first column of each row produce the same $\pic(X)^G$. The resulting upper bounds on $\am(X,G)$ therefore apply to each subgroup in that row. 

\begin{remark}   \label{remark:for4.1}
\hfill \normalfont
\begin{enumerate}
\item[(i)] Suppose $\autp(X,G)=\autp(Y,G)=0$. Then each $\iota^*H_i$ is $G$-invariant, and $\operatorname{pr}_i$ is $G$-equivariant. By Proposition \ref{KXlinearizable}, $-K_X$ and $\iota^*\operatorname{pr}_i^*(-K_{\pp^1})$ for $1\leq i\leq 4$ are $G$-linearizable. Consequently,
\[
\overline{2\iota^*H_i}=0,\quad \overline{\iota^*(H_1+H_2+H_3+H_4)}=0.
\]
Thus 
\[
\am(X,G)= \langle \overline{\iota^*H_1}, \overline{\iota^*H_2}, \overline{\iota^*H_3} \rangle \subseteq (\zz/2\zz)^3.
\]

\item[(ii)] Suppose now that $\autp(X,G)=\autp(Y,G)=\zz/2\zz$. We may assume that $G$-action exchanges $\iota^*H_1$, $\iota^*H_2$ and fixes $\iota^*H_3$, $\iota^*H_4$. Then $\pic(X)^G=\zz[\iota^*(H_1+H_2)]\oplus\zz[\iota^*H_3]\oplus\zz[\iota^*H_4]$ and $\operatorname{pr}_3$ and $\operatorname{pr}_4$ are $G$-equivariant. Consequently, $2\iota^*H_3$, $2\iota^*H_4$ are $G$-linearizable, and since $\overline{-K_X}=0$, we obtain 
\[
\overline{2\iota^*H_3}= \overline{2\iota^*H_4}=0,\qquad\overline{\iota^*(H_1+H_2)} = \overline{\iota^*H_3}+\overline{\iota^*H_4}.
\]
Thus, we get $\am(X,G)=\langle \overline{\iota^*H_3}, \overline{\iota^*H_4}\rangle  \subseteq (\zz/2\zz)^2$.

\item[(iii)] Suppose that $\autp(X,G)=\autp(Y,G) = \zz/3\zz$. So, after relabeling if necessary, we may assume that $G$ cyclically permutes the generators, $\iota^*H_1, \iota^*H_2$, and $\iota^*H_3$. Then $\pic(X)^G= \zz[\iota^*(H_1+H_2+H_3)]\oplus\zz[\iota^*H_4]$. So, $\operatorname{pr}_4$ is $G$-equivariant and $\overline{2\iota^*H_4}=0$. Since $-K_X$ is also $G$-linearizable, we get 
\[
\overline{\iota^*(H_1+H_2+H_3)}=\overline{\iota^*H_4}.
\]
Thus, $\am(X,G)=\langle \overline{\iota^*H_4}\rangle \subseteq \zz/2\zz$.

\item[(iv)] Suppose that $\autp(X,G)=\autp(Y,G)=(\zz/2\zz)^2$. There are two cases: either this group is isomorphic to $\langle (12)(34)\rangle$ or $\langle (12), (34)\rangle$. In both cases, we have 
\[
\pic(X)^G\cong \pic(Y)^G=\zz[H_1+H_2]\oplus\zz[H_3+H_4].
\]
Note that $m(H_1+H_2)$, $m(H_3+H_4)\;|\; 2$. Since $-K_X=\iota^*(H_1+H_2)+\iota^*(H_3+H_4)$ is $G$-linearizable, we have
\[
\overline{2\iota^*(H_1+H_2)} = \overline{2\iota^*(H_3+H_4)}=0, \quad \overline{\iota^*(H_1+H_2)}=\overline{\iota^*(H_3+H_4)}. 
\]
Thus, $\am(X,G) = \langle \overline{\iota^*(H_1+H_2)} \rangle \subseteq \zz/2\zz$. It is not yet known whether there exists a Fano threefold $X$ and a group $G$ for which $\am(X,G)\cong \zz/2\zz$.

\item[(v)] The remaining possibilities for $\autp(X,G)$ are
\[
\zz/4\zz,\; A_4,\; \langle (12)(34),(13)(24),(14)(23)\rangle ,\; D_8,\; S_4.
\]
In all these cases, $G$ acts transitively on $\pic(X)$, so $\pic(X)^{G} = \zz[-K_X]$. Since $-K_X$ is $G$-linearizable, it follows that $\am(X,G)=0$.
    \end{enumerate}
\end{remark}

By Proposition \ref{prop:g-eq-morphism-sameAmitsur}, we have 
\begin{equation}  \label{eqn:for4.1-Amitsurs-same}
\am(X,G)=\am(Y,G). 
\end{equation}

Let $[x_0:x_1],[y_0:y_1],[z_0:z_1],[w_0:w_1]$ be coordinates on $(\pp^1)^4$.

\begin{lemma}  \label{Am:4.1}
    Let $X$ be a Fano threefold in family \textnumero $4.1$. Let $G\subseteq \aut(X)$ be a finite group. Then $\am(X,G)$ is isomorphic to $(\zz/2\zz)^3$, $(\zz/2\zz)^2$, $\zz/2\zz$, or trivial.
\end{lemma}
\begin{proof}
Suppose $\autp(Y,G)=0$. 
We show that the upper bound in Remark \ref{remark:for4.1} (i) can be realized for a suitable $X$ and a finite group $G\subseteq \aut(X)$. 

Consider the group $\mathcal{G}_1\cong (\zz/2\zz)^3 $ generated by involutions
\[
\begin{aligned}
\sigma_1: \ &
([x_0:x_1],[y_0:y_1],[z_0:z_1],[w_0:w_1])\\
&\qquad  \longmapsto
([x_0:-x_1],[y_0:-y_1],[z_0:-z_1],[w_0:-w_1]),\\[4pt]
\sigma_2:\ &
([x_0:x_1],[y_0:y_1],[z_0:z_1],[w_0:w_1])\\
&\qquad \longmapsto
([x_1:x_0],[y_1:y_0],[z_1:z_0],[w_1:w_0]),\\[4pt]
\sigma_3:\ &
([x_0:x_1],[y_0:y_1],[z_0:z_1],[w_0:w_1])\\
&\qquad \longmapsto
([x_1:-x_0],[y_0:-y_1],[z_0:z_1],[w_1:w_0]).
\end{aligned}
\]
Note that $\autp(Y,\mathcal{G}_1)=0$. The hypersurface defined by
\begin{multline} 
    a_0(x_0y_0z_0w_0 + x_1y_1z_1w_1) +a_1(x_0y_0z_1w_1 + x_1y_1z_0w_0) +\\
    ia_1(x_0y_1z_0w_1 + x_1y_0z_1w_0) +ia_0(x_0y_1z_1w_0 + x_1y_0z_0w_1) =0,
\end{multline}
is $\mathcal{G}_1$-invariant and for $(a_0,a_1)=(2,-3)$, it is smooth and irreducible \cite{magma}. Thus, the hypersurface defines a smooth Fano threefold $X$ in family \textnumero 4.1 and $\mathcal{G}_1\subset \aut(X)$. Observe that $\langle \sigma_1,\sigma_2\rangle$-action on $\pp^1_z$ does not lift to a linear representation, so $H_3$ is not $\langle \sigma_1,\sigma_2\rangle$-linearizable. Similarly, the $\langle \sigma_2,\sigma_3\rangle$-action on $\pp^1_y$ does not lift, whereas its actions on $\pp^1_x$ and $\pp^1_z$ do, so $H_2$ is not $\langle \sigma_2,\sigma_3\rangle$-linearizable. Finally, the $\langle \sigma_1,\sigma_3\rangle$-action on $\pp^1_x$ does not lift but its actions on $\pp^1_y$ and $\pp^1_z$ do. So, only $H_1$ is not $\langle \sigma_1,\sigma_3\rangle$-linearizable. 
Each of $H_1,H_2,H_3$ is not linearizable for a suitable subgroup $H\subseteq \mathcal{G}_1$, therefore they are not $\mathcal{G}_1$-linearizable by Remark \ref{remark:equivariant-restrction-Amitsur} and have independent obstruction classes in $H^2(\mathcal{G}_1,\cc^\times)$. It follows that 
\[
\am(X,\mathcal{G}_1) = \am(Y,\mathcal{G}_1) \cong (\zz/2\zz)^3. 
\]

\medskip
Suppose now that $\autp(Y,G)=\zz/2\zz$. The bound on $\am(X,G)$ attained in Remark \ref{remark:for4.1} (ii) is sharp as we show now. Consider the group $\mathcal{G}_2$ generated by automorphisms
\begin{equation*}
\begin{aligned}
\tau_1: \ &
([x_0:x_1],[y_0:y_1],[z_0:z_1],[w_0:w_1])\\
&\qquad\longmapsto
([y_0:-y_1],[x_0:x_1],[z_1:z_0],[w_0:iw_1]),\\[4pt]
\tau_2: \ &
([x_0:x_1],[y_0:y_1],[z_0:z_1],[w_0:w_1])\\
&\qquad \longmapsto
([y_1:-y_0],[x_0:x_1],[z_0:-iz_1],[w_1:-w_0]).
\end{aligned}
\end{equation*}
and note that $\pic(Y)^{\mathcal{G}_2} = \zz[H_1+H_2, H_3, H_4]$. Let $\tau_3:=\tau_1^2\tau_2^2$, explicitly,  
\[
\begin{aligned}
\tau_3:\ &
([x_0:x_1],[y_0:y_1],[z_0:z_1],[w_0:w_1])\\
&\qquad \longmapsto
([x_1:x_0],[y_1:y_0],[z_0:-z_1],[w_0:-w_1]).
\end{aligned}
\]
Consider the hypersurface $X$ defined by 
\begin{multline}   \label{eqn-second-4.1}
  x_0y_0(z_0w_0 + z_1w_0)  + x_0y_1(z_0w_1 -iz_1w_1) \\
  +x_1y_0(z_0w_1 +iz_1w_1) -x_1y_1(z_0w_0-z_1w_0)=0.
\end{multline}
The hypersurface is $\mathcal{G}_2$-invariant, smooth and irreducible \cite{magma}. So, $X$ defines a smooth Fano threefold in family \textnumero 4.1 and $\mathcal{G}_2\subset \aut(X)$. Observe that $\tau_1\tau_3=\tau_3\tau_1$ but their matrix representatives on $\pp^1_z$ anticommute. Thus, the induced $\langle \tau_1, \tau_3\rangle$-action on $\pp^1_z$ does not lift to a linear representation, thus $H_3$ is not $\langle \tau_1, \tau_3\rangle$-linearizable. 
Similarly, the matrix representatives of $\tau_2,\;\tau_3$ on $\pp^1_w$ anticommute. Consequently, the $\langle \tau_2, \tau_3\rangle$-action on $\pp^1_w$ does not lift to a linear representation, and $H_4$ is not $\langle \tau_2, \tau_3\rangle$-linearizable. 

Finally, the $\langle \tau_1, \tau_3\rangle$-action on $\pp^3$ induced via the segre embedding $\pp^1_x\times \pp^1_y \hookrightarrow \pp^3$ does not lift, therefore $H_1+H_2$ is not $\langle \tau_1, \tau_3\rangle$-linearizable, either. By Remark
\ref{remark:equivariant-restrction-Amitsur}, $H_3$, $H_4$, and $H_1+H_2$ are not $\mathcal{G}_2$-linearizable. Moreover, they non-linearizability correspond to three distinct obstructions classes in $H^2(\mathcal{G}_2,\cc^\times)$. Using \eqref{eqn:for4.1-Amitsurs-same}, we conclude 
\[
\am(X,\mathcal{G}_2)=\am(Y,\mathcal{G}_2)=(\zz/2\zz)^2.
\]

\medskip
Suppose that $\autp(Y,G) = \zz/3\zz$.  We show that the upper bound in Remark \ref{remark:for4.1} (iii) can be realized. Consider the group $\mathcal{G}_3\cong \zz/3\zz\times (\zz/2\zz)^2$ generated by 
\[
\begin{aligned}
&([x_0:x_1],[y_0:y_1],[z_0:z_1],[w_0:w_1])\\
&\qquad\longmapsto
([z_0:-z_1],[x_0:-x_1],[y_0:-y_1],[w_0:-w_1]),\\[4pt]
&([x_0:x_1],[y_0:y_1],[z_0:z_1],[w_0:w_1])\\
&\qquad\longmapsto
([z_1:z_0],[x_1:x_0],[y_1:y_0],[w_1:w_0]).
\end{aligned}
\]
and observe $\pic(Y)^{\mathcal{G}_3}=\zz[H_1+H_2+H_3,H_4]$. Consider the hypersurface $X$ defined by
\begin{multline}   \label{eqn:hypersurface-forG3-fano4.1}
    a_0(x_0y_0z_0w_0 + x_1y_1z_1w_1) + a_1(x_0y_0z_1w_1 + x_0y_1z_0w_1 + \\ x_1y_0z_0w_1 + x_0y_1z_1w_0 + x_1y_0z_1w_0 + x_1y_1z_0w_0) =0.
\end{multline}
It is $\mathcal{G}_3$-invariant. The smoothness and irreduciblity can be verified for $(a_0,a_1)=(2,3)$; for instance using \cite{magma}. Observe that the $\mathcal{G}_3$-action on $\pp^1_w$ does not lift to a linear representation, therefore, $H_4$ is not $\mathcal{G}_3$-linearizable. 
Using \eqref{eqn:for4.1-Amitsurs-same}, we have 
\[
\am(X,\mathcal{G}_3)=\am(Y,\mathcal{G}_3)\cong \zz/2\zz.
\]
Finally, consider the involution
\[
\begin{aligned}
    \sigma : \ & ([x_0:x_1],[y_0:y_1],[z_0:z_1],[w_0:w_1]) \\
    &\qquad\longmapsto ([y_0:y_1],[x_0:x_1],[z_0:z_1],[w_0:w_1]),
\end{aligned}
\]
and let $\mathcal{G}'_3:=\langle\mathcal{G}_3,\sigma \rangle \cong S_3$. Then the equation in \eqref{eqn:hypersurface-forG3-fano4.1} is also $\mathcal{G}_3'$-invariant, and therefore $\mathcal{G}_3'$ lifts to a subgroup of $\aut(X)$ such that $\autp(X,\mathcal{G}_3')=S_3$. By Remark \ref{remark:equivariant-restrction-Amitsur}, since $\iota^*H_4$ is not $\mathcal{G}_3$-linearizable, it is not $\mathcal{G}_3'$-linearizable. Combined with the upper bound in Remark \ref{remark:for4.1} (iii), we get
\[
\am(X,\mathcal{G}_3') \cong \zz/2\zz.
\]
\end{proof}

\section{Complete Intersections} We discuss the following families 
\begin{center}
   \textnumero 2.9, \; \textnumero 3.13, \; \textnumero 3.24, \; \textnumero 4.7.
\end{center}
A member of each of these families can be realized as a complete intersection of divisors on $Y$, where $Y=\pp^3\times\pp^2$ for \textnumero 2.9, $Y:=(\pp^2)^3$ for \textnumero 3.13, $Y:=\pp^1\times\pp^2\times\pp^2$ for \textnumero 3.24, and $Y:=\pp^2\times\pp^2\times\pp^1\times \pp^1$ for \textnumero 4.7. 

In each case, let $\iota: X \hookrightarrow Y$ denote the natural embedding, and $\operatorname{pr}_i$ denote the projection of $Y$ onto its $i$th factor. Set $H_i:= \operatorname{pr}_i^*\mathcal{O}(1)$. 

\begin{remark}   \label{remark:complete-int-equivarianceofiota-and-sameAm}
   \normalfont Suppose that $\iota$ is $G$-equivariant and it induces the isomorphism $\pic(X)\cong \pic(Y)$. Then Proposition \ref{prop:g-eq-morphism-sameAmitsur} gives $\am(X,G)=\am(Y,G)$. Throughout this section, whenever these hypotheses hold, we use this identification without explanation. 
\end{remark}
 
\smallskip
A Fano threefold $X$ in family \textnumero 2.9 is a complete intersection of ample divisors of bidegrees $(1,1)$ and $(2,1)$ in $Y=\pp^3 \times \pp^2$ \cite[Proposition 7.5]{MM86}. By Proposition \ref{thm:lefschetz}, we have $\pic(X)=\zz[\iota^*H_1]\oplus\zz[\iota^*H_2]$.  

\begin{lemma}    \label{Am:2.9}
    Let $X$ be a Fano threefold in family \textnumero $2.9$. Let $G\subseteq \aut(X)$ be a finite group. Then $\am(X,G)=0$.
\end{lemma}
\begin{proof}
Recall from Table \ref{table:amitsur}, $\autp(X)=0$. So, Remark \ref{remark:equivarince-of-embedding} allows us to assume that $G\subseteq \aut(Y;X)$ and $\iota$ is $G$-equivariant. Since $\autp(Y)=0$, both $\operatorname{pr}_1$ and $\operatorname{pr}_2$ are $G$-equivariant. Hence, by Proposition \ref{KXlinearizable},
\[
\iota^*\operatorname{pr}_1^*(-K_{\pp^3})=4\iota^*H_1, \quad  \iota^*\operatorname{pr}_2^*(-K_{\pp^2})=3\iota^*H_2, \quad -K_X=\iota^*H_1 +  \iota^*H_2
\]
are all $G$-linearizable on $X$. Furthermore, we have $\iota^*H_2 = 4\iota^*H_1+3\iota^*H_2+4K_X$, therefore, both $\iota^*H_2$ and $\iota^*H_1$ are $G$-linearizable on $X$. It follows that $ \am(X,G)=0$.
\end{proof}

Let $X$ be a Fano threefold in family \textnumero $3.13$. Mori-Mukai describe $X$ as the blow-up of a smooth divisor $W\subset \pp^2\times\pp^2$ of bidegree $(1,1)$ along a smooth curve $C$ of bidegree $(2,2)$, such that the composition $C\hookrightarrow \mathbb{P}^2\times \mathbb{P}^2 \to \mathbb{P}^2$ is an embedding for $i=1,2$.

We recall an alternative description of $X$ and the classification of its automorphism groups from \cite[Lemma 5.97]{calabi}. Let $Y:= \mathbb{P}^2\times \mathbb{P}^2\times \mathbb{P}^2$. Set 
$H_i := \operatorname{pr}_i^*\mathcal{O}_{\pp^2}(1)$. Then $X$ is the complete intersection of divisors $D_1, D_2, D_3$, where $D_1\in |H_1+H_2|$, $D_2\in |H_1+H_3|$, and $D_3\in |H_1+H_3|$. We have 
\[
\pic(X) = \zz[\iota^*H_1]\oplus\zz[\iota^*H_2]\oplus\zz[\iota^*H_3],
\]
and $-K_X = \iota^*(H_1+H_2+H_3)$. Moreover, one can choose coordinates on $Y$ such that $X$ is given by 

\begin{equation}   \label{eqn:Gm-3.13}
        \begin{cases}
        x_0y_0 + x_1y_1 + x_2y_2= 0,\\
        y_0z_0 + y_1z_1 + y_2z_2 =0,\\
        (1+s)x_0z_1 + (1-s)x_1z_0 - 2x_2z_2=0,
    \end{cases}
\end{equation}
    where $s\in \mathbb{C}$ such that $s\neq \pm 1$, or
\begin{equation}   \label{eqn:Ga-3.13}
        \begin{cases}
        x_0y_0 + x_1y_1 + x_2y_2= 0,\\
        y_0z_0 + y_1z_1 + y_2z_2 =0,\\
        x_0z_1+x_1z_0+x_1z_2-x_2z_1-2x_2z_2=0.
    \end{cases}
\end{equation}

If $X$ is given by \eqref{eqn:Gm-3.13} with $s\neq 0$, then $\aut(X)\cong \mathbb{G}_m \rtimes S_3$ where the $\mathbb{G}_m$-action on $Y$ is induced via the automorphisms 
\begin{equation}   \label{eqn:Gmautomorphism}
\begin{aligned}
\phi_{\lambda}\colon\ 
&([x_0:x_1:x_2],[y_0:y_1:y_2],[z_0:z_1:z_2]) \\
&\qquad\longmapsto
\left(
 \left[\lambda x_0:\frac{x_1}{\lambda}:x_2\right],
 \left[\frac{y_0}{\lambda}:\lambda y_1:y_2\right],
 \left[\lambda z_0:\frac{z_1}{\lambda}:z_2\right]
\right).
\end{aligned}
\end{equation}
for $\lambda \in \cc^*$, and if $s=0$, then $\aut(X)\cong \operatorname{PGL_2(\cc)\times S_3}$ where the $\operatorname{PGL}_2(\cc)$-action on $Y$ is given by the transformation 
\begin{align}   \label{eqn:PGL2-Ga-infiniteautomorphism}
 \phi_A : \begin{pmatrix}
     x_0\\
     x_1\\
     x_2 
 \end{pmatrix} \mapsto A \begin{pmatrix}
     x_0\\
     x_1\\
     x_2 
 \end{pmatrix},\quad \begin{pmatrix}
     y_0\\
     y_1\\
     y_2 
 \end{pmatrix} \mapsto (A^{-1})^T \begin{pmatrix}
     y_0\\
     y_1\\
     y_2 
 \end{pmatrix},\quad \begin{pmatrix}
     z_0\\
     z_1\\
     z_2 
 \end{pmatrix} \mapsto A \begin{pmatrix}
     z_0\\
     z_1\\
     z_2 
 \end{pmatrix},
\end{align}
where 
\[
A= \begin{pmatrix}
    a^2 & b^2 & 2ab \\
    c^2 & d^2 & 2cd \\
    ac & bd & ad+bc
\end{pmatrix} ,
\]
and $ad-bc\neq 0$. 

If $X$ is given by \eqref{eqn:Ga-3.13}, $\aut(X)\cong \mathbb{G}_a \rtimes S_3$ and the $\mathbb{G}_a$-action is given by the transformation $\phi_A$ in \eqref{eqn:PGL2-Ga-infiniteautomorphism} for the matrix 
\[
A= \begin{pmatrix}
    1 & a^2 & 2a \\
    0 & 1 & 0 \\
    0 & a & 1
\end{pmatrix},
\]
for $a\in \cc$. Moreover, we observe that the groups $\operatorname{PGL}_2(\cc)$, $\mathbb{G}_m$, and $\mathbb{G}_a$ act trivially on $\pic(X)$, and the nontrivial action comes from the second factor in the product $\aut(X)$. 

\medskip
Let $G\subseteq \aut(X)$ be a finite group. 
In \cite{sharma2026actionspicardgroupsmooth}, it was shown that $\autp(X)=S_3$. There are three cases to consider: $\autp(X,G)$ contains a $3$-cycle, is cyclic of order $2$, or is trivial. 

\begin{lemma}   \label{Am:3.13}
    Let $X$ be a Fano threefold in \textnumero $3.13$. Let $G\subseteq \aut(X)$ be a finite group. Then $\am(X,G)=0$. 
\end{lemma}
\begin{proof}
First, suppose that $\autp(X,G)\supseteq \zz/3\zz$. Then $\pic(X)^G=\zz[-K_X]$. By Proposition \ref{KXlinearizable}, we get $\am(X,G)=0$. 

Suppose $\autp(X,G) = \zz/2\zz$. Without loss of generality, we may assume that $G$ exchanges the first and third factors of $Y$, which gives $\pic(X)^G=\zz[\iota^*(H_1+H_3)]\oplus\zz[\iota^*H_2]$. Note that since $-K_X=\iota^*(H_1+H_3)+\iota^*H_2$ is $G$-linearizable, the $G$-linearizability of $\iota^*H_2$ implies that $\iota^*(H_1+H_3)$ is also $G$-linearizable. We claim that $\iota^*H_2$ is $G$-linearizable.    

Let $X$ be given by \eqref{eqn:Gm-3.13} with $s\neq 0$ and let the involution that acts nontrivially on $\pic(X)$ be 
\begin{equation}   \label{eqn:GmPGL2-involution}
\begin{aligned}
\tau_{x,z}\colon\ 
&([x_0:x_1:x_2],[y_0:y_1:y_2],[z_0:z_1:z_2])\\
&\qquad\longmapsto
([z_1:z_0:z_2],[y_1:y_0:y_2],[x_1:x_0:x_2]).
\end{aligned}
\end{equation}
Then we may assume that $G\subseteq \mathbb{G}_m\rtimes \langle \tau_{x,z}\rangle$. Note that we have the relations $\tau_{x,z}\phi_{\lambda}\tau_{x,z}^{-1}=\phi_{\lambda^{-1}}$ and $\tau_{x,z}^2=I$ in $\operatorname{GL}_3(\cc)^3$. Thus, $G$-action lifts to a linear representation, and $\am(X,G)=0$. 

If $X$ is given by \eqref{eqn:Gm-3.13} with $s= 0$, again we may assume that $G\subseteq \operatorname{PGL}_2(\cc)\times \langle \tau_{x,z}\rangle $ such that $\tau_{x,z}\phi_A\tau_{x,z}=\phi_A$ as matrices in $\operatorname{GL}_3(\cc)^3$. 

If $X$ is given by \eqref{eqn:Ga-3.13}, then we let $G\subseteq \mathbb{G}_a\rtimes \tau_{x,z}$, where 
\begin{equation}   \label{eqn:Ga-involution}
\begin{aligned}
\tau_{x,z}\colon\ 
&([x_0:x_1:x_2],[y_0:y_1:y_2],[z_0:z_1:z_2])\\
&\qquad\longmapsto
([z_0:z_1:-z_2],[y_0:y_1:-y_2],[x_0:x_1:-x_2]).
\end{aligned}
\end{equation}
The matrices satisfy the relations $\tau_{x,z}^2=I$ and $\tau_{x,z}\phi_a\tau^{-1}_{x,z}=\phi_{-a}$ in $\operatorname{GL}_3(\cc)^3$. In all three cases, the induced $G$-action on $\pp^2_y$ lifts to a linear representation, so $H_2$ is $G$-linearizable and consequently, $\overline{\iota^*H_2}=0$. Thus, $\am(X,G)=0$. 

Now suppose that $\autp(X,G)=0$. Then $G$-action is induced by the automorphisms in \eqref{eqn:PGL2-Ga-infiniteautomorphism} and \eqref{eqn:Gmautomorphism}. Since the induced $G$-action on each factor $\pp^2$ is also a genuine $3$-dimensional linear representation, $\iota^*H_i$ is $G$-linearizable for $1\leq i \leq 3$. Thus, we obtain $\am(X,G)=0$. 
\end{proof}

A Fano threefold $X$ in family \textnumero $3.24$ is the fiber product of a Fano threefold $W$ from \textnumero 2.32 and $\operatorname{Bl}_p\mathbb{P}^2$ over $\mathbb{P}^2$ \cite[Table 3]{MM81classtable}. In \cite[\S 77]{quantum}, the authors equivalently describe $X$ as follows. Let $Y:=\mathbb{P}^1\times \mathbb{P}^2\times \mathbb{P}^2$. Then $X:=D_1\cap D_2$ is the complete intersection of two divisors $D_1\in |H_1+H_2|$ and $D_2\in |H_2+H_3|$, respectively. We have 
\[
\pic(X) =\zz[\iota^*H_1]\oplus\zz[\iota^*H_2]\oplus\zz[\iota^*H_3],
\]
$-K_X = \iota^*(H_1+H_2+2H_3)$, and $\autp(X)=0$.. 

\begin{lemma}   \label{Am:3.24}
    Let $X$ be a Fano threefold in \textnumero $3.24$. Let $G\subseteq \aut(X)$ be a finite group. Then $\am(X,G)=0$.
\end{lemma}
\begin{proof}
By Remark \ref{remark:equivarince-of-embedding}, we may assume that $G\subseteq \aut(X)$ is a finite subgroup of $\aut(Y;X)$ and $\iota$ is $G$-equivariant. The descended $G$-action fixes each factor of $Y$, therefore, $\operatorname{pr}_i$ is $G$-equivariant for $\leq i \leq 3$. Then Proposition \ref{KXlinearizable} implies that the following divisors are $G$-linearizable on $X$: $\iota^*\operatorname{pr}_i^*(-K_{\pp^1})$, $\iota^*\operatorname{pr}_2^*(-K_{\pp^2})$, $\iota^*\operatorname{pr}_3^*(-K_{\pp^2})$, and $-K_X$, so
\[
2\overline{\iota^*H_1}= \; 3\overline{\iota^*H_2}= \; 3\overline{\iota^*H_3}= \; \overline{\iota^*(H_1+H_2+2H_3)}=0.
\]
Then the equation $2\iota^*H_2+\iota^*H_3 = -2K_X -2\iota^*H_1-3\iota^*H_3$ gives $\overline{2\iota^*H_2+\iota^*H_3}=0$. Furthermore, $\iota^*H_2+2\iota^*H_3 = 2(2\iota^*H_2+\iota^*H_3)-3\iota^*H_2$ yields $\overline{\iota^*H_2}+\overline{2\iota^*H_3}=0$. From these two equations, we get
\[
\overline{\iota^*H_1} = -\overline{K_X} - (\overline{\iota^*H_2+2\iota^*H_3}) =0, \qquad \overline{\iota^*H_2}= \overline{\iota^*H_3}. 
\]
It follows that 
\[
\am(X,G) = \langle \overline{\iota^*H_2}\rangle \subseteq \zz/3\zz. 
\]
Now since $X:=D_1\cap D_2$ is $G$-invariant, by Lemma \ref{lemma:G-invariance-for-CI} (i), $D_1$ and $D_2$ are individually $G$-invariant divisors. By Remark \ref{remark:CI}, $D_1$ and $D_2$ are $G$-linearizable on $Y$. 
Thus, $\iota^*D_1$ and $\iota^*D_2$ are also $G$-linearizable on $X$. Finally, since $\iota^*H_1$, $\iota^*H_2$, and $\iota^*H_3$ generate $\pic(X)^G$, and
\[
\iota^*H_3 = -K_X  -\iota^*D_1, \quad \iota^*H_2 = \iota^*D_2 - \iota^*H_3, 
\]
we obtain $\am(X,G)=0$.
\end{proof}

A Fano threefold $X$ in family \textnumero $4.7$ is the blow-up of a Fano threefold $W$ from family \textnumero 2.32 along the disjoint union $C_1\sqcup C_2$ of two curves $C_1$ and $C_2$ of degrees $(0,1)$ and $(1,0)$, respectively. We use the complete-intersection description from \cite{homogeneousvectorbundles}: let $Y:= \pp^2\times \pp^2 \times \pp^1 \times \pp^1$, and let $\operatorname{pr}_i$ denote the projection of $Y$ onto its $i$th factor. Set $H_i:=\operatorname{pr}_i^*\mathcal{O}(1)$. Then $X$ is the complete intersection of divisors $D_1, D_2$ and $D_3$, where 
$D_1 \in |H_1+H_2|$, $D_2\in |H_1+H_3|$, and $D_3\in |H_2+H_4|$. 

We have $\pic(X) = \bigoplus_{i=1}^4 \zz[\iota^*H_i]$, $-K_X = \iota^*(H_1+H_2+H_3+H_4)$, and $\autp(X)=\zz/2\zz$ (Table \ref{table:amitsur}). 
In \cite{sharma2026actionspicardgroupsmooth}, it was shown that $\autp(X)=\zz/2\zz$ using the Mori-Mukai description. It is just as easier to see this from the complete intersection description. 
The intersection defined by the set of equations
\begin{equation*}
    x_0y_0+x_1y_1+x_2y_2=0,\quad x_0z_0+x_1z_1=0, \quad y_0w_0+y_1w_1=0
\end{equation*}
is smooth and irreducible, thus defines a Fano threefold $X$ in \textnumero 4.7. The intersection is also invariant under the involution 
\[
\begin{aligned}
\sigma : ([x_0:x_1:x_2],[y_0:y_1:y_2],&[z_0:z_1],[w_0:w_1])\\
&\longmapsto
([y_0:y_1:y_2],[x_0:x_1:x_2],
 [w_0:w_1],[z_0:z_1]).
\end{aligned}
\]
So, $\sigma \in \aut(X)$ and it acts on $\pic(X)$ by swapping the generators $\iota^*H_1$ with $\iota^*H_2$, and $\iota^*H_3$ with $\iota^*H_4$. Hence, $\autp(X,\langle \sigma\rangle)=\zz/2\zz$, and $\autp(X)=\zz/2\zz$ by \cite[Proposition 3.3]{sharma2026actionspicardgroupsmooth}. 

By Remark \ref{remark:equivarince-of-embedding}, a finite group $G\subseteq \aut(X)$ may be assumed to be a subgroup of $\aut(Y;X)$, and $\iota$ is $G$-equivariant. 
\begin{lemma}  \label{Am:4.7}
    Let $X$ be a Fano threefold in \textnumero $4.7$. Let $G\subseteq \aut(X)$ be a finite group. Then $\am(X,G)=0$.
\end{lemma}
\begin{proof}
Suppose first that $\autp(X,G)=0$. Since the complete intersection $X:=D_1\cap D_2\cap D_3$ is $G$-invariant, Lemma \ref{lemma:G-invariance-for-CI} (i) applied to the lifted action on $\operatorname{Cox}(Y)$ shows that the divisors $D_1$, $D_2$, $D_3$ are $G$-invariant. By Remark \ref{remark:CI}, $D_i$ and hence, $\iota^*D_i$ is $G$-linearizable for each $i$. Thus, we have 
   \begin{equation}  \label{eqn:4.7-one}
   \overline{\iota^*(H_1+H_2)}=0, \quad \overline{\iota^*(H_1+H_3)}=0, \quad \overline{\iota^*(H_2+H_4)}=0.
   \end{equation}
   Moreover, $\operatorname{pr}_i$ is $G$-equivariant since the $G$-action on $Y$ fixes its each factor. So, we also have the following relations due to Proposition \ref{KXlinearizable}:
   \begin{equation}  \label{eqn:4.7-two}
   3\overline{\iota^*H_1} \;=\; 3\overline{\iota^*H_2} \;=\; 2\overline{\iota^*H_3} \;=\; \overline{2\iota^*H_4}=0. 
   \end{equation}
   Combining with relations from \eqref{eqn:4.7-one}, we get 
   \[ \overline{\iota^*H_1}=\overline{\iota^*H_2}=\overline{\iota^*H_3}=\overline{\iota^*H_4} = \overline{2\iota^*H_1},
   \]
   it follows that $\am(X,G)=0$. 
   
   Suppose $\autp(X,G)=\zz/2\zz$. Then $\pic(X)^G= \zz[\iota^*(H_1+H_2)]\oplus\zz[\iota^*(H_3+H_4)]$.
   In this case, the divisor $D_1$ is $G$-invariant, and by Remark \ref{remark:CI}, it is also $G$-linearizable. Thus, $\overline{\iota^*H_1+\iota^*H_2}=0$. Combined with 
   \[
   -\overline{K_X}=\overline{\iota^*H_1+\iota^*H_2+\iota^*H_3+\iota^*H_4}=0,
   \]
   we obtain
  $\overline{\iota^*H_3+\iota^*H_4}= 0$, thus $\am(X,G)=0$. 
\end{proof}


\section{Blow-ups of \texorpdfstring{$\mathbb{P}^1 \times \mathbb{P}^2$}{P1P2}}
We consider the families
\begin{center}
  \textnumero 3.5,\; \textnumero 3.21,\; \textnumero 3.22,
\end{center}
each of which contains a member $X$ obtained by blowing up $\mathbb{P}^1\times\pp^2$ along a curve $C$. Let $\pi: X \rightarrow \pp^1\times\pp^2$ denote the blow-up, with exceptional divisor $E$. For $i=1,2$, let $\operatorname{pr}_i$ denote the natural projection of $\pp^1\times\pp^2$ onto its $i$th factor, and set $H_1:=\operatorname{pr}_1^*\mathcal{O}_{\pp^1}(1)$ and $H_2:=\operatorname{pr}_2^*\mathcal{O}_{\pp^2}(1)$. Then $\pic(X)= \zz[\pi^*H_1]\oplus\zz[ \pi^*H_2]\oplus\zz[E]$, and $-K_X=2\pi^*H_1+3\pi^*H_2-E$. 

\begin{remark}  \label{remark:for3.5-and-3.21}
   \normalfont Let $G\subseteq \aut(X)$ be a finite group. Recall $\autp(X)=0$ for a Fano threefold $X$ of any of these three families by Table \ref{table:amitsur}. Hence, the $G$-action descends to $\pp^1\times\pp^2$, and we may regard $G \subseteq \aut(\pp^1\times\pp^2;C)$ (Lemma \ref{Lemma:autblowup}), and $\pi$ is $G$-equivariant. Thus, by Theorem \ref{imprimitivebcdp}, 
    \[
    \am(X,G)=\am(\pp^1\times\pp^2,G). 
    \]
\end{remark}

\medskip
Let $X$ be a Fano threefold in family \textnumero $3.5$. Then $X$ is the blow-up of $\pp^1\times \pp^2$ along a curve $C$ of bidegree $(5,2)$ such that $C \hookrightarrow \mathbb{P}^1\times \mathbb{P}^2 \xrightarrow{\operatorname{pr}_2}\mathbb{P}^2$ is an embedding. In particular, $\operatorname{pr}_2(C)$ is a smooth conic in $\mathbb{P}^2$ and 
\[
\operatorname{pr}_2^{-1}(\operatorname{pr}_2(C)) = \mathbb{P}^1 \times \operatorname{pr}_2(C) \cong \mathbb{P}^1\times \mathbb{P}^1.
\]
Following the proof of \cite[Lemma 8.7]{izv19}: Set $S:=\operatorname{pr}_2^{-1}(\operatorname{pr}_2(C)) \subset \pp^1\times\pp^2$. Then $C$ has bidegree $(1,5)$ on $S$. Observe that $S$ is $\aut(\mathbb{P}^1\times\mathbb{P}^2;C)$-invariant. Suppose that $G\subseteq \aut(\pp^1\times\pp^2;C)$ is a finite group. Then we may assume $G \subset \aut(S)$ since every automorphism acting trivially on $S$ acts trivially on the first factor $\pp^1$ and pointwise preserves $\operatorname{pr}_2(C)$, thus also acts trivially on the second factor $\pp^1\times\pp^2$. 

Therefore, the inclusion map $S\hookrightarrow \pp^1\times\pp^2$ is $G$-equivariant. By functoriality, it induces the injective homomorphism
\[
\am(\pp^1\times\pp^2,G) \hookrightarrow \am(S,G) \subseteq (\zz/2\zz)^2.
\]
Recall that we also have $\am(\pp^1\times\pp^2,G) \subseteq \zz/2\zz\times\zz/3\zz$; see the entry for family \textnumero 2.34 in Table \ref{table:amitsur}. Thus, $\am(\pp^1\times\pp^2,G) \subseteq \zz/2\zz$.

The following lemma shows that this upper bound is sharp. 

\begin{lemma}   \label{Am:3.5}
Let $X$ be a smooth member of family \textnumero $3.5$. Let $G \subseteq \aut(X)$ be a finite group. Then $\am(X,G)$ is isomorphic to $\mathbb{Z}/2\mathbb{Z}$ or $0$. 
\end{lemma}
\begin{proof}
We show that there exists a smooth conic $C\subset \pp^1\times\pp^1$ of bidegree $(1,5)$ and a group $\mathcal{G}\subset \aut(S)$ such that $C$ is $\mathcal{G}$-invariant and $\mathcal{G}$ lifts to a subgroup $\overline{\mathcal{G}}\subset \aut(\pp^1\times\pp^2)$ such that 
\[
\am(X,\overline{\mathcal{G}})=\am(\pp^1\times\pp^2,\overline{\mathcal{G}})\cong \zz/2\zz.
\] 
Let $([x_0:x_1],[y_0:y_1])$ be coordinates on $S$, and we have the embedding $S \hookrightarrow \pp^1\times\pp^2$
\[
([x_0:x_1],[y_0:y_1]) \longmapsto ([x_0:x_1],[y_0^2:y_0y_1:y_1^2]).
\]
Consider the group $\mathcal{G}:= \langle \sigma , \tau\rangle \cong (\zz/2\zz)^2$ where
\begin{equation}    \label{2-group}
\begin{split}
\sigma : ([x_0:x_1],[y_0:y_1]) &\mapsto ([x_0:-x_1],[y_0:-y_1]) \\
\tau : ([x_0:x_1],[y_0:y_1]) &\mapsto ([x_1:x_0],[y_1:y_0]).
\end{split}
\end{equation}      
Then $\am(S,\mathcal{G})=\mathbb{Z}/2\mathbb{Z}$. 

Consider the curve $C\subset S$ defined by the equation
\[
x_0(y_0^5+2y_0^3y_1^2)+x_1(y_1^5+2y_0^2y_1^3)=0.
\]
It is $\mathcal{G}$-invariant, and it is smooth and irreducible \cite{magma}. Let $ \overline{\sigma}$ and $\overline{\tau}$ denote lifts of $\sigma$ and $\tau$, respectively, and $\overline{\mathcal{G}}:=\langle \overline{\sigma}, \overline{\tau}\rangle \subset \aut(\pp^1\times\pp^2)$, where 
\begin{align*}
   \overline{\sigma}: ([x_0:x_1],[y_0^2:y_0y_1: y_1^2]) &\mapsto ([x_0:-x_1],[y_0^2:-y_0y_1: y_1^2])\\
   \overline{\tau}: ([x_0:x_1],[y_0^2:y_0y_1: y_1^2]) &\mapsto ([x_1:x_0],[y_1^2:y_0y_1: y_0^2]). 
\end{align*}
Now since the induced $\overline{\mathcal{G}}$-action on $\pp^1_{x_0,x_1}$ does not lift to a linear representation, we get 
$\am(\pp^1\times\pp^2, \overline{\mathcal{G}})\cong \zz/2\zz$.
\end{proof}

A Fano threefold $X$ in \textnumero $3.21$ is the blow-up of $\pp^1\times\pp^2$ in a curve $C$ of bidegree $(2,1)$, where $C$ is the complete intersection of two divisors $D_1\in |H_2|$ and $D_2\in |H_1+2H_2|$, respectively \cite[\S 74]{quantum}. 

\begin{lemma}  \label{Am:3.21}
    Let $X$ be a Fano threefold in \textnumero $3.21$. Let $G\subseteq \aut(X)$ be a finite group. Then  $\am(X,G)=0$. 
\end{lemma}
\begin{proof}
By Remark \ref{remark:for3.5-and-3.21}, let $G\subseteq \aut(\pp^1\times\pp^2;C)$. Therefore, Lemma \ref{lemma:G-invariance-for-CI} (ii) implies that there exist $G$-invariant divisors $D_1'\in |D_1|$ and $D_2'\in |D_2| $ such that $C=D_1'\cap D_2'$. Consequently, by Remark \ref{remark:CI}, $\overline{D_1}=\overline{D_2}=0$ in $\am(\pp^1\times\pp^2,G)$. 
Finally, since $H_1=D_2-2D_1$ and $H_2=D_1$, we obtain 
\[
\am(X,G)=\am(\pp^1\times \pp^2,G)=0.
\] 
\end{proof}

A Fano threefold $X'$ in family \textnumero $4.5$ can be described as the blow-up of $\pp^1\times\pp^2$ in the disjoint union $C:= C_1 \sqcup C_2$, where $C_1$ and $C_2$ are curves of bidegrees $(2,1)$ and $(1,0)$, respectively. Equivalently, $X'$ is the blow-up of a Fano threefold $X$ from \textnumero 3.21 in the proper transform, $\widetilde{C_2}$, of curve $C_2$. Let $\phi: X' \to X$ denote this blow with the exceptional divisor $F$. 
 
As a consequence of the previous lemma, we get

\begin{coro}   \label{Am:4.5}
    Let $X'$ be a Fano threefold in \textnumero $4.5$. Let $G\subseteq \aut(X')$ be a finite group. Then $\am(X',G)=0$. 
\end{coro}
 \begin{proof}
    From Table \ref{table:amitsur}, we have $\autp(X')=0$. So, by Lemma \ref{Lemma:autblowup}, we may regard $G\subseteq \aut(X;\widetilde{C_2})$ and $\phi$ is $G$-equivariant. By Theorem \ref{imprimitivebcdp} and Lemma \ref{Am:3.21}, we obtain 
    \[
    \am(X',G)=\am(X,G)=0.
    \]
 \end{proof}

A Fano threefold $X$ in family \textnumero 3.22 is the blow-up of $\pp^1\times\pp^2$ along a conic $C\simeq \{p\}\times \pp^1$ in the fiber of projection $\operatorname{pr}_1$. 

\begin{lemma}   \label{Am:3.22}
 Let $X$ be a Fano threefold in \textnumero 3.22. Let $G\subseteq \aut(X)$ be a finite group. Then $\am(X,G)=0$. 
\end{lemma}
\begin{proof}
    By Remark \ref{remark:for3.5-and-3.21}, let $G\subseteq \aut(\pp^1\times\pp^2;C)$ and $\pi$ is $G$-equivariant, so
    \[
    \am(X,G)=\am(\pp^1\times\pp^2,G) \subseteq \zz/2\zz\times\zz/3\zz.
    \]
     Since $\autp(\pp^1\times\pp^2,G)=0$, both projection maps $\operatorname{pr}_i$ are $G$-equivariant. Then the $G$-invariance of $C$ implies that $\operatorname{pr}_1(C)\in \pp^1$ and $\operatorname{pr}_2(C)\subset \pp^2$ are $G$-invariant. Consequently, there is a 1-dimensional $G_{\mathcal{O}_{\pp^1}(1)}$-subrepresentation of $H^0(\pp^1,\mathcal{O}_{\pp^1}(1))$, and a $1$-dimensional $G_{\mathcal{O}_{\pp^2}(2)}$-subrepresentation of $H^0(\pp^1,\mathcal{O}_{\pp^2}(2))$. By Proposition \ref{Gsubrepn_result}, $\mathcal{O}_{\pp^1}(1)$ and $\mathcal{O}_{\pp^2}(2)$ are $G$-linearizable, and so, $\overline{H_1}=\overline{2H_2}=0$. Finally, since $3H_2=\operatorname{pr}_2^*(-K_{\pp^2})$, Proposition \ref{KXlinearizable} yields $\overline{3H_2}=0$. Thus, $\overline{H_2}=0$ and we obtain
     \[
     \am(X,G)=\am(\pp^1\times\pp^2,G)=0.
     \]
\end{proof}

\section{Blow-ups of \texorpdfstring{$\mathbb{P}^1\times \mathbb{P}^1\times \mathbb{P}^1$}{P1P1P1}} 

We discuss the following families
\begin{center}
\textnumero 4.3,\; \textnumero 4.6,\; \textnumero 4.8,\; \textnumero 4.13. 
\end{center}
A Fano threefold $X$ belonging to any of these families is the blow-up of $(\pp^1)^3$ along a smooth curve $C$ \cite{MM81classtable,MMerratum}. For $1\leq i\leq 3$, let $\operatorname{pr}_i \colon (\pp^1)^3 \to \mathbb{P}^1$ 
denote the projection onto the $i$th factor, and set $H_i:=\operatorname{pr}_i^*\mathcal{O}_{\pp^1}(1)$. Let $\pi: X\to (\pp^1)^3$ be the blow-up map with the exceptional divisor $E$. Then 
\[
\pic(X) = \bigoplus_{i=1}^3 \zz[\pi^*H_i] \oplus \zz[E]
\]
and $-K_X = \pi^*(2H_1+2H_2+2H_3)-E$. 

\begin{remark} \label{remark:blow-P1cube-families}
 \normalfont Let $G\subseteq \aut(X)$ is a finite group. From \cite[Lemmas 4.14, 4.15, 4.16, 4.17]{sharma2026actionspicardgroupsmooth}, we recall that $G$ acts on the set $\{\pi^*H_1, \pi^*H_2, \pi^*H_3 \}$ by permutations, and $E$ is $G$-invariant. Thus, by Lemma \ref{Lemma:autblowup}, we may identify $G$ as a subgroup of $\aut((\pp^1)^3;C)$, and $\pi$ is $G$-equivariant. From Theorem \ref{imprimitivebcdp}, we have
 \begin{equation}    \label{eqn:Am-equal-for-blowupsofP1cube}
 \am(X,G)=\am((\pp^1)^3,G).
 \end{equation} 
\end{remark}

\smallskip
Thus, for each family below, the computation of $\am(X,G)$ reduces to that of $\am((\pp^1)^3,G))$ for the induced $G$-action on $(\pp^1)^3$.  

\vspace{0.2cm}

Let $X$ be a Fano threefold in family \textnumero 4.3, then the center of the blow-up is a curve $C$ of tridegree $(1,1,2)$, which can be described as the complete intersection of two divisors, $D_1 \in |H_1+H_2|$ and $D_2 \in |H_1+H_2+H_3|$ \cite[(7.22)]{MM86}. 

Let $([x_0:x_1],[y_0:y_1],[z_0:z_1])$ be coordinates on $(\pp^1)^3$. 

\begin{lemma}  \label{Am:4.3}
    Let $X$ be a Fano threefold in \textnumero $4.3$. Let $G \subseteq \aut(X)$ be a finite group. Then 
     $\am(X,G)$ is isomorphic to $\zz/2\zz$ or trivial. 
\end{lemma}
\begin{proof}
Recall that $\autp(X)=\zz/2\zz$ (Table \ref{table:amitsur}), with the nontrivial $G$-action exchanging $\pi^*H_1$ and $\pi^*H_2$ and fixing $\pi^*H_3$. Then the divisor classes $D_1, D_2$ are $G$-invariant for any $G\subseteq \aut(X)$. 
Since $C$ is $G$-invariant, Lemma \ref{lemma:G-invariance-for-CI} (ii) and Remark \ref{remark:CI} imply that both $D_1$ and $D_2$ are $G$-linearizable. Since $H_3=D_2-D_1$, it is also $G$-linearizable.  

Suppose first that $\autp(X,G)=\autp(\pp^1)^3,G)=\zz/2\zz$. Then $\pic((\pp^1)^3)^G=\zz[H_1+H_2]\oplus\zz[H_3]$.
However, both generators are $G$-linearizable. Thus, we get
\[
\am(X,G)=\am((\pp^1)^3,G)=0.
\]

Now suppose $\autp(X,G)=\autp((\pp^1)^3,G)=0$. Then each $\operatorname{pr}_i$ is $G$-equivariant, and so $\overline{2H_i}=0$. Combined with $\overline{D_1}=0$ and $\overline{H_3}=0$, we get 
\[
\am((\pp^1)^3,G) = \langle \overline{H_1}\rangle \subseteq \zz/2\zz.
\]
To realize the nontrivial possibility, consider 
\[ 
D_1 := \{ x_0y_1-x_1y_0=0\} , \qquad D_2:= \{z_0x_1y_1-z_1x_0y_0 =0\},
\]
and let $\mathcal{G}:=\langle \sigma, \tau\rangle \subset \aut((\pp^1)^3)$, where 
\begin{equation}    \label{groupforAm(P1)3}
\begin{split}
    \sigma : ([x_0:x_1],[y_0:y_1],[z_0:z_1]) &\mapsto ([x_0:-x_1],[y_0:-y_1],[z_0:z_1])\\
   \tau :  ([x_0:x_1],[y_0:y_1],[z_0:z_1]) &\mapsto ([x_1:x_0],[y_1:y_0],[z_1:z_0]). 
\end{split}
\end{equation} 
The curve $C:=D_1\cap D_2$ is $\mathcal{G}$-invariant since each $D_1,D_2$ are individually $\mathcal{G}$-invariant. Moreover, $C$ is smooth and irreducible \cite{magma}. Thus, $\mathcal{G}$ lifts to the blow-up $X$ of $(\pp^1)^3$ in $C$ and $\autp(X,\mathcal{G})=0$.

Observe that the induced $\mathcal{G}$-action on $\pp^1_{x}$ and $\pp^1_{y}$ does not lift to a linear representation, however, the $\mathcal{G}$-action on $\pp^3$ via the Segre embedding $\pp^1_x\times\pp^1_y \hookrightarrow \pp^3$ does have a linear lift. So, $H_1$ and $H_2$ are not $\mathcal{G}$-linearizable and they have same obstruction classes in $H^2(\mathcal{G},\cc^\times)$, but $H_1+H_2$ is $\mathcal{G}$-linearizable. Using \eqref{eqn:Am-equal-for-blowupsofP1cube}, we conclude
\[
\am(X,\mathcal{G}) = \am((\pp^1)^3,\mathcal{G}) \cong \mathbb{Z}/2\mathbb{Z}.
\]
\end{proof}

Let $X$ be a Fano threefold in \textnumero 4.8, then the center of the blow-up is a curve $C$ with tridegree $(0,1,1)$, and it is a complete intersection of divisors $D_1\in |H_1|$ and $D_2 \in |H_2+H_3|$ \cite[(7.22)]{MM86}. 
Recall that $\autp(X)=\zz/2\zz$, where the nontrivial $G$-action preserves the divisor classes $D_1$ and $D_2$. The curve $C$ is $G$-invariant, so Lemma \ref{lemma:G-invariance-for-CI} (i) and the following Remark \ref{remark:CI} imply 
\begin{equation}   \label{eqn:relations-for-4.8}
   \overline{D_1}= \overline{H_1}=0, \qquad \overline{D_2}=\overline{H_2+H_3}=0. 
\end{equation}

\begin{lemma}   \label{Am:4.8}
    Let $X$ be a Fano threefold in \textnumero $4.8$. Let $G\subseteq \aut(X)$ be a finite group. Then $\am(X,G)$ is isomorphic to $\zz/2\zz$ or trivial.
    \end{lemma}
\begin{proof}
Suppose $\autp(X,G)= \autp((\pp^1)^3,G)=\zz/2\zz$, then 
\[
\pic(Y)^G = \zz[H_1, H_2+H_3]. 
\]
Since both generators are $G$-linearizable from \eqref{eqn:relations-for-4.8}, we get $\am(X,G)=0$.

Now suppose $\autp(X,G)=\autp((\pp^1)^3,G)=0$. Then $\operatorname{pr}_i$ is $G$-equivariant, consequently, $\overline{2H_i}=0$. Using \eqref{eqn:relations-for-4.8}, we get 
\[
\am((\pp^1)^3,G)=\langle \overline{H_2}\rangle = \langle \overline{H_3}\rangle \subseteq \zz/2\zz. 
\] 
The upper bound can be realized as follows. Consider
\[
D_1 :=  \{ x_0+x_1     =0\}, \hspace{1.0cm}  D_2:=  \{ y_0z_1-y_1z_0 =0 \}
\]
and the group $\mathcal{G}:= \langle \sigma, \tau \rangle \subset \aut((\pp^1)^3)$, where 
\begin{equation*}
   \begin{split}
\sigma : ([x_0:x_1],[y_0:y_1],[z_0:z_1]) &\mapsto ([x_0:x_1],[y_0:-y_1],[z_0:-z_1]) \\
\tau :  ([x_0:x_1],[y_0:y_1],[z_0:z_1]) &\mapsto ([x_1:x_0],[y_1:y_0],[z_1:z_0]).
   \end{split} 
\end{equation*}
Then the intersection $C:=D_1\cap D_2$ is $\mathcal{G}$-invariant since the hypersurfaces defining $D_1$, $D_2$ are individually $G$-invariant. Moreover, the curve $C$ is also smooth and irreducible \cite{magma}. Finally, the induced $\mathcal{G}$-action on $\pp^1_y$ and $\pp^1_z$ does not lift to linear representation, but the $\mathcal{G}$-action on $\pp^3$ via the Segre embedding $\pp^1_y\times\pp^1_z \hookrightarrow\pp^3$ has a linear lift. From \eqref{eqn:Am-equal-for-blowupsofP1cube}, it follows
\[
\am(X,\mathcal{G}) = \am((\pp^1)^3,\mathcal{G}) \cong \zz/2\zz.
\]
\end{proof}

A smooth Fano threefold $X$ in family \textnumero 4.13 is a blow-up of $(\pp^1)^3$ in a curve $C$ of degree $(1,1,3)$ \cite{MMerratum}, where $C$ is a complete intersection of divisors $D_1 \in |H_1+H_2|$ and $D_2\in |2H_1+H_2+H_3|$. 

By Remark \ref{remark:blow-P1cube-families}, it suffices to work over $(\pp^1)^3$. 

We recall that $\autp(X)=\zz/2\zz$ (Table \ref{table:amitsur}). So, there are two cases to consider for groups $G$.

\begin{remark}
\hfill \normalfont \begin{enumerate}
        \item[(i)] Suppose $\autp(X,G)=\autp((\pp^1)^3,G)=0$. So, each $\operatorname{pr}_i$ is $G$-equivariant, and we get $\overline{2H_i}=0$ for $1\leq i \leq 3$. By Lemma \ref{lemma:G-invariance-for-CI} (ii), the divisor classes $D_1$ and $D_2$ admit $G$-linearization by Remark \ref{remark:CI}. 
Consequently, $\overline{H_2+H_3}=0$ and $\overline{H_1+H_3}=0$, and it follows that 
\[
\overline{H_1}=\overline{H_2}=\overline{H_3}.
\] 
Hence, 
\[
\am((\pp^1)^3,G) = \langle \overline{H_1}\rangle \subseteq \zz/2\zz.
\]
\item[(ii)] Suppose $\autp(X,G)=\autp((\pp^1)^3,G)=\zz/2\zz$. Then $G$ swaps $H_1$ and $H_2$ and
$\pic((\pp^1)^3)^G= \zz[H_1+H_2, H_3]$. Thus, $\operatorname{pr}_3$ is $G$-equivariant, and consequently, $\overline{2H_3}=0$. Furthermore, the $G$-invariant divisor class $D_1=H_1+H_2$ is $G$-linearizable due to Remark \ref{remark:CI}. Consequently,
\[
\am((\pp^1)^3,G)= \langle \overline{H_3}\rangle \subseteq \zz/2\zz.
\]
    \end{enumerate}
\end{remark}
We now show that these bounds can be realized. 

\begin{lemma}     \label{Am:4.13}
    Let $X$ be a Fano threefold in \textnumero $4.13$. Let $G\subseteq \aut(X)$ be a finite group. Then $\am(X,G)$ is isomorphic to $\zz/2\zz$ or trivial.
\end{lemma}
\begin{proof}
Consider a group $\mathcal{G}_1$ generated by 
\begin{align*}
    \tau_1 : ([x_0:x_1],[y_0:y_1],[z_0:z_1]) &\mapsto ([x_0:-x_1],[y_0:-y_1],[z_0:-z_1]) \\
    \sigma_1 : ([x_0:x_1],[y_0:y_1],[z_0:z_1]) &\mapsto ([x_1:x_0],[y_1:y_0],[z_1:z_0]),
\end{align*}
where $\autp((\pp^1)^3,\mathcal{G}_1)=0$. 
Let
\begin{align}  \label{eqn:for-4.13}
    D_1 :=\{x_0y_1-x_1y_0=0\}, \hspace{0.5cm}  D_2 :=\{x_0^2y_0z_1-x_1^2y_1z_0=0\},
\end{align}
and let $C:=D_1\cap D_2$. Then $C$ is $\mathcal{G}_1$-invariant, and it can be verified to be both smooth and irreducible (\cite{magma}). Thus, $\mathcal{G}_1 \subset \aut((\pp^1)^3;C)$. Here the $\mathcal{G}_1$-action on each factor $\pp^1$ does not lift to $\operatorname{GL}_2(\mathbb{C})$. In addition, for each pair $1\leq i,j\leq 3$, $H_i+H_j$ is $\mathcal{G}_1$-linearizable. Thus, we obtain 
\[
\am(X,\mathcal{G}_1)= \am((\pp^1)^3,\mathcal{G}_1)\cong \mathbb{Z}/2\mathbb{Z}. 
\] 

Now consider the automorphisms
\begin{align*}
    \tau_2 : ([x_0:x_1],[y_0:y_1],[z_0:z_1]) &\mapsto ([y_0:-y_1],[x_0:-x_1],[z_0:-z_1]) \\
    \sigma_2 : ([x_0:x_1],[y_0:y_1],[z_0:z_1]) &\mapsto ([y_1:y_0],[x_1:x_0],[z_1:z_0]),
\end{align*}
set $\mathcal{G}_2:=\langle \sigma_2,\tau_2\rangle$, and note that $\autp((\pp^1)^3,\mathcal{G}_2)=\zz/2\zz$. 
Then $\mathcal{G}_2$ fixes $D_1$, and 
\[
\tau_2 D_2 = \{y_1^2x_1z_0-y_0^2x_0z_1 =0\} =\sigma_2 D_2.
\]
A direct computation shows that $D_1\cap \tau_2 D_2 = D_1\cap \sigma_2 D_2=C$, this can also be verified using Macaulay2 \cite{M2}. Thus, $\mathcal{G}_2\subset \aut((\pp^1)^3;C)$. Note that the $\mathcal{G}_2$-action on $\pp^1_z$ does not lift to a group homomorphism, so $H_3$ is not $\mathcal{G}_2$-linearizable. Hence, we get 
\[
\am(X,\mathcal{G}_2)=\am(((\pp^1)^3,\mathcal{G}_2)\cong \zz/2\zz.
\]
\end{proof}

A Fano threefold $X$ in family \textnumero $4.6$
is the blow-up of $(\pp^1)^3$ in the intersection $C$ of divisors $D_1, D_2$, and $D_3$ where $D_1 \in |H_1+H_2|$, $D_2 \in |H_2+H_3|$, and $D_3 \in |H_1+H_3|$. 
Let $G\subseteq \aut(X)$ be a finite group. By Remark \ref{remark:blow-P1cube-families}, we may assume that $G$ is the lift of a subgroup of $\aut((\pp^1)^3;C)$. 

Recall that $\autp(X)=S_3$. There are three cases to consider for $G$. 
\begin{remark}
\normalfont \hfill \begin{enumerate}
        \item[(i)] Suppose $\autp(X,G)=0$. Then $\autp((\pp^1)^3,G)=0$, so $\operatorname{pr}_i$ is $G$-equivariant for $1\leq i \leq 3$. Thus, the pullbacks of anticanonical divisors to $(\pp^1)^3$ and $-K_{(\pp^1)^3}$ are $G$-linearizable, thus, $\overline{2H_i}=0$ by Theorem \ref{KXlinearizable}. We observe that $D_1$, $D_2$, $D_3$ are $G$-linearizable on $(\pp^1)^3$ by Lemma \ref{lemma:G-invariance-for-CI} (i) and Remark \ref{remark:CI}. Thus, 
\begin{equation}   
    \overline{H_1+H_2}=0, \quad \overline{H_2+H_3}=0, \quad \overline{H_1+H_3}=0.
\end{equation} We obtain
        \[
        \am(X,G)=\am((\pp^1)^3,G)= \langle \overline{H_1}\rangle,
        \]
        where the latter group is isomorphic to $\zz/2\zz$ or trivial. 
        
        \item[(ii)] Suppose $\autp(X,G)=\zz/2\zz$. We may assume that $G$ permutes the generators $H_1$ and $H_2$ and keeps $H_3$ invariant. Then $\pic((\pp^1)^3)^G=\zz[H_1+H_2,H_3]$ and since $\operatorname{pr}_3$ is $G$-equivariant, we have $\overline{2H_3}=0$. In this case, the divisor class $D_1$ is $G$-invariant, and the argument in Lemma \ref{lemma:G-invariance-for-CI} (i) implies that the divisor itself is $G$-invariant. Then from the argument in Remark \ref{remark:CI}, $D_1$ is $G$-linearizable. Thus, 
        \[
        \am(X,G) = \am((\pp^1)^3,G) = \langle H_3\rangle,
        \]
        which is isomorphic to $\zz/2\zz$ or trivial. 
        \item[(iii)] Suppose $\autp(X,G)$ contains a $3$-cycle. Then $\pic((\pp^1)^3)^G= \zz[H_1+H_2+H_3]$, and we obtain 
        \[
        \am((\pp^1)^3,G) \subseteq \langle \overline{H_1+H_2+H_3}\rangle,
        \]
        where the latter group is isomorphic to $\zz/2\zz$ or trivial, since $-K_{(\pp^1)^3}=2(H_1+H_2+H_3)$. 
    \end{enumerate}
\end{remark}
In the following lemma, we verify whether the upper bounds can be realized in each case. 
\begin{lemma}  \label{Am:4.6}
    Let $X$ be a Fano threefold in family \textnumero $4.6$. Let $G \subseteq \aut(X)$ be a finite group. Then $\am(X,G)$ is isomorphic to $\zz/2\zz$ or trivial.
\end{lemma}
\begin{proof}
Let $\mathcal{G}_1 \subset \aut((\pp^1)^3)$ be a group generated by involutions
\begin{align*}
    \sigma_1 : ([x_0:x_1],[y_0:y_1],[z_0:z_1]) &\mapsto ([x_0:-x_1],[y_0:-y_1],[z_0:-z_1])\\
  \tau_1:   ([x_0:x_1],[y_0:y_1],[z_0:z_1]) &\mapsto ([x_1:x_0],[y_1:y_0],[z_1:z_0]).
\end{align*}
Observe that $\autp((\pp^1)^3,\mathcal{G}_1)=0$. Consider the hypersurfaces given by 
\begin{equation}  \label{eqn:for4.6}
D_1:= \{x_0y_1-x_1y_0 = 0\},\hspace{0.1cm}  D_2:=\{y_0z_1-y_1z_0 =0\}, \hspace{0.1cm}  D_3:=\{x_0z_1-x_1z_0 =0\}. 
\end{equation}
Then the intersection $D_1\cap D_2\cap D_3$ is $\mathcal{G}_1$-invariant since $D_1, D_2, D_3$ are individually $\mathcal{G}_1$-invariant. We also verify that $C$ is smooth and irreducible (\cite{magma}). Thus, $\mathcal{G}_1 \subset \aut((\pp^1)^3,C)$. Now since the induced $\mathcal{G}_1$-action on $\pp^1_{x_0,x_1}$ does not lift to a homomorphism to $\operatorname{GL}_2(\mathbb{C})$, $H_1$ is not $\mathcal{G}_1$-linearizable. However, $H_1+H_2$ is $\mathcal{G}_1$-linearizable. Thus, we obtain 
\[
\am(X,G)=\am((\pp^1)^3,\mathcal{G}_1)\cong \zz/2\zz.
\]

Suppose that $\autp(X,G)$ contains a $2$-cycle. In particular, let $\mathcal{G}_2$ be a group generated by 
\begin{align*}
    \sigma_2 : ([x_0:x_1],[y_0:y_1],[z_0:z_1]) &\mapsto ([y_0:-y_1],[x_0:-x_1],[z_0:-z_1])\\
  \tau_2:   ([x_0:x_1],[y_0:y_1],[z_0:z_1]) &\mapsto ([y_1:y_0],[x_1:x_0],[z_1:z_0]),
\end{align*}
where $\autp(X,\mathcal{G}_2)=\zz/2\zz$. Consider the curve defined by the intersection of hypersurfaces in \eqref{eqn:for4.6}, it is $\mathcal{G}_2$-invariant. Thus, $\mathcal{G}_2 \subset \aut((\pp^1)^3;C)$. Since the $\mathcal{G}_2$-action on $\pp^1_{z_0,z_1}$ does not lift to a homomorphism to $\operatorname{GL}_2(\cc)$, $H_3$ is not $\mathcal{G}_2$-linearizable. But $H_1+H_2$ is $\mathcal{G}_2$-linearizable. Hence, 
\[
\am(X,\mathcal{G}_2)= \am((\pp^1)^3,\mathcal{G}_2)\cong \zz/2\zz.
\]

Finally, the intersection given by equations in \eqref{eqn:for4.6} is also invariant under the group $\mathcal{G}_3=\langle \sigma_3, \tau_3\rangle $, where 
\begin{align*}
    \sigma_3 : ([x_0:x_1],[y_0:y_1],[z_0:z_1]) &\mapsto ([y_0:-y_1],[z_0:-z_1],[x_0:-x_1])\\
  \tau_3:   ([x_0:x_1],[y_0:y_1],[z_0:z_1]) &\mapsto ([y_1:y_0],[z_1:z_0],[x_1:x_0]).
\end{align*}
such that $\autp(X,\mathcal{G}_3)=\zz/3\zz$. Thus, $\mathcal{G}_3 \subset \aut((\pp^1)^3;C)$. Here the $\mathcal{G}_3$-action induced on $\pp^7$ via the Segre embedding $(\pp^1)^3 \hookrightarrow \pp^7$ does not lift to a linear representation, therefore, $H_1+H_2+H_3$ is not $\mathcal{G}_3$-linearizable. Hence, 
\[
\am(X,\mathcal{G}_3) = \am((\pp^1)^3,\mathcal{G}_3) \cong \zz/2\zz.
\]
\end{proof}

\section{Blow-ups of cones} 
We study the Amitsur groups of Fano threefolds belonging to families 
\begin{center}
    \textnumero 3.9,\; \textnumero 4.2.
\end{center}
these arise from the blow-ups of the Veronese cone and the quadric cone, respectively. 

We recall their common construction from \cite[\S 4.6]{calabi}. For family \textnumero 3.9, let $\mathscr{S}:=\pp^2$ and $C\subset\mathscr{S}$ be a smooth quartic curve. For family \textnumero 4.2, let $\mathscr{S}:=\pp^1\times\pp^1$ and $C\subset $ be a smooth curve of bidegree $(2,2)$. Then starting with a double cover of $\pp^1\times \mathscr{S}$, resolving its singularities, followed by a contraction produces $X$, together with two birational contractionss 
\[
\phi: X \to V, \quad \phi': X \to V',
\]
where $V\cong V'$. The involution $\iota \in \aut(\pp^1\times \mathscr{S})$ defined by $([u:v],p)\mapsto ([v:u],p)$ lifts to an involution of $X$ that swaps $\phi$ and $\phi'$, thus it does not descend to an automorphism of either of them.

\subsection{Blow-up of the Veronese cone} Let $V:=\pp(\mathcal{O}\oplus\mathcal{O}(2))$. This is the Fano threefold in family \textnumero 2.36, obtained by blowing-up the Veronese cone at its vertex. Let $\operatorname{pr}: Y \to \pp^2$ denote the natural projection, and set $H:=\operatorname{pr}^*\mathcal{O}_{\pp^2}(1)$. Denote the positive and negative sections by $S_+$ and $S_-$, respectively. 

\smallskip
A Fano threefold $X$ in family \textnumero 3.9 is the blow-up of $V$ along a smooth curve $\mathcal{C}\subset S_+$, which corresponds to $C$. Following Matsuki's notation (\cite[III]{Mat95}), let $\ell_1, \dots, \ell_4$ denote the four extremal generators of $\overline{\operatorname{NE}}(X)$, and $E_1, \dots , E_4$ denote the exceptional divisors of the corresponding contractions $\operatorname{cont}_{\ell_i}$, $1\leq i \leq 4$. With this labeling, $\phi=\operatorname{cont}_{\ell_1}$ and $E_1$ is its exceptional divisor. 

Set $h=\operatorname{cont}_{\ell_1}^*H$. Then $\pic(X)=\zz[h, E_1,E_4]$ and $-K_X = 5h+2E_4-E_1$. From \cite[Lemma 4.21]{sharma2026actionspicardgroupsmooth}, we recall that $\autp(X)=\zz/2\zz$ and $\autp(X,\langle \iota\rangle)=\zz/2\zz$, where 
\[
\iota^* : h\mapsto h,\quad E_1\mapsto 4h-E_1,\quad E_4 \mapsto 2h-E_1+E_4.
\]

\begin{lemma} \label{Am:3.9} 
    Let $X$ be a smooth Fano threefold in family \textnumero 3.9. Let $G\subseteq \aut(X)$ be a finite group. Then $\am(X,G)=0$.  
\end{lemma}
\begin{proof}
    Suppose first that $\autp(X,G)=0$. Since $E_1$ is $G$-invariant, Lemma \ref{Lemma:autblowup} allows us to regard $G$ as a subgroup of $\aut(V;C)$ so that $\phi$ is $G$-equivariant. Since $\am(V,G)=0$ from Table \ref{table:amitsur}, Theorem \cite{sharma2026amitsurgroupsprimitivefano} gives
    \[
    \am(X,G)=\am(V,G)=0.
    \]
    Now suppose that $\autp(X,G)=\zz/2\zz$. Then $\pic(X)^G=\zz[h,E_1-2E_4]$. Set $D:= E_1-2E_4$. Using the intersection product table in \cite{Mat95} and \cite[Lemma 1.22]{prokhorov-fano-threefolds}, we find
   \[
\begin{aligned}
hE_1^2 &= -4, & h^2E_1 &= 0, & E_1^3 &= -24,\\
h^2E_4 &= 1, & hE_4^2 &= -2, & E_4^3 &= 4,\\
E_1^2E_4 &= 0, & E_1E_4^2 &= 0, & hE_1E_4 &= 0.
\end{aligned}
\]
Moreover, $h^3=0$, since $h$ is the pullback of a hyperplane on $\pp^2$. Thus,  
\begin{align*}
(h-D)^2K_X =11h^2D-7hD^2+D^3 = -22+84-56=6. 
\end{align*}
We also have $h^2K_X=-2$. Applying Proposition \ref{prop:amitsur-period-congruences}, we get $m(h)\;|\;3$ and $m(h-D)=1$. Equivalently, $\overline{3h}=0$ and $\overline{h}=\overline{D}$. Since $-K_X=5h-D$ is $G$-linearizable, we also have $m(5h-D)=1$, or 
\[
\overline{5h}=\overline{2h}=\overline{D}.
\]
Combined with the relation $\overline{h}=\overline{D}$, we get $\overline{h}=0$. Thus both $h$ and $D$ are $G$-linearizable, and $\am(X,G)=0$. 
\end{proof}


\subsection{Blow-up of the quadric cone}
Let $\mathscr{S}:=\pp^1\times\pp^1$,  $V:=\pp(\mathcal{O}\oplus \mathcal{O}(1,1))$, and $\operatorname{pr}:V\to\mathscr{S}$ be the natural projection. Then $V$ belongs to family \textnumero 3.31. Set $H_1:= \operatorname{pr}^*\mathcal{O}(1,0)$ and $H_2:= \operatorname{pr}^*\mathcal{O}(0,1)$. 
Let $S_+$ and $S_-$ denote the positive and negative sections of $\operatorname{pr}$, respectively.

A Fano threefold $X$ in family \textnumero 4.2 is the blow-up of $V$ in a curve $C\subset S_+$, corresponding to a smooth curve of bidegree $(2,2)$ on $\mathscr{S}$. Let $\phi: X \to V$ denote this blow-up. 

In the notation from \cite{Mat95,Mat23}, let $\ell_1,\dots , \ell_6$ denote the extremal generators of $\overline{\operatorname{NE}}(X)$, and $E_i$ denote the exceptional divisor corresponding to the extremal contraction $\operatorname{cont}_{\ell_i}$, for $1\leq i \leq 6$. We identify $\phi$ with $\operatorname{cont}_{\ell_1}$ so that $E_1$ is its exceptional divisor. For the simplicity in notation, let $h_1:=\phi^*H_1$ and $h_2:=\phi^*H_2$. Then  
 \[
\pic(X)= \zz[h_1]\oplus\zz[h_2]\oplus\zz[E_5]\oplus\zz[\widetilde{S_+}],
\]
where $E_5=\phi^*S_-$, $\widetilde{S_+}=\phi^*S_+-E_1$. In particular, 
\begin{align}
\widetilde{S_+}
    &= E_5+h_1+h_2-E_1,
    \label{eqn:for-4.2}\\
-K_X
    &= 2h_1+2h_2+E_5+\widetilde{S_+}.
    \label{eqn:K_X-for-4.2}
\end{align}

Recall from \cite[Lemma 4.22]{sharma2026actionspicardgroupsmooth}, 
\[
\autp(X)= \langle \widetilde{\sigma}^*,\iota^*\rangle \cong(\zz/2\zz)^2,
\]
where 
\[
\widetilde{\sigma}^*: \begin{cases}
    h_1 \mapsto h_2,\\
    h_2 \mapsto h_1,\\
    E_5 \mapsto E_5,\\
    \widetilde{S_+} \mapsto \widetilde{S_+}
\end{cases} \qquad 
\iota^*: \begin{cases}
    h_1 \mapsto h_1,\\
    h_2 \mapsto h_2,\\
    E_5 \mapsto \widetilde{S_+},\\
    \widetilde{S_+} \mapsto E_5.
\end{cases}
\]
So, we have four cases to consider for $G$. We obtain an upper bound on $\am(X,G)$ in all four cases. 

\begin{lemma}   \label{Am:4.2}
    Let $X$ be a smooth Fano threefold in \textnumero 4.2. Let $G\subseteq \aut(X)$ be a finite group. Then $\am(X,G) \subseteq (\zz/2\zz)^2$. 
\end{lemma}
\begin{proof}
Suppose first that $\autp(X,G)=0$. Then from \eqref{eqn:for-4.2}, $E_1$ is $G$-invariant. So, by Lemma \ref{Lemma:autblowup} allows us to regard $G \subseteq \aut(V;C)$, and $\phi$ is $G$-equivariant. By Theorem \ref{imprimitivebcdp} and the classification for family \textnumero 3.31 (Table \ref{table:amitsur}), it follows that 
\[
\am(X,G)=\am(V,G) \subseteq \zz/2\zz. 
\]

Next, suppose that $\autp(X,G)= \zz/2\zz$. There are three cases to consider for $G$.

If $\autp(X,G)=\langle \widetilde{\sigma}^*\rangle$, then $\pic(X)^G=\zz[h_1+h_2,E_5,\widetilde{S_+}]$. 
Thus, from the relation \eqref{eqn:for-4.2}, we find that $E_1$ is $G$-invariant. So, $\phi$ is $G$-equivariant, and Theorem \ref{imprimitivebcdp} yields $\am(X,G)=\am(V,G)$. 
From the classification of $\am(V,G)$ using \cite[Proposition 3.1]{sharma2026amitsurgroupsprimitivefano}, we get $\am(V,G)=0$. Thus, $\am(X,G)=0$. 

\smallskip
Now if $\autp(X,G)= \langle \iota^*\rangle$, then  $\pic(X)^G=\zz[h_1,h_2, E_5+\widetilde{S_+}]$. Note that $m(h_i) \; | \; 2$. 
Let $D:= S_- +\widetilde{S_+}$. Since $-K_X$ is $G$-linearizable, relation \eqref{eqn:K_X-for-4.2} gives $\overline{E_5+\widetilde{S_+}} =0$. We obtain 
\[
\am(X,G) = \langle \overline{h_1}, \overline{h_2}\rangle \subseteq (\zz/2\zz)^2.
\]

Finally, if $\autp(X, G)=\langle \widetilde{\sigma}^*\iota^*\rangle$. Then $\pic(X)^G = \zz[h_1+h_2, E_5 +\widetilde{S_+}]$. Set $D_1:= h_1+h_2$ and $D_2:= E_5 +\widetilde{S_+}$.  Combined with the $G$-linearizability of $-K_X=2D_1+D_2$, the relation \eqref{eqn:K_X-for-4.2} gives $\overline{2D_1}+\overline{D_2}=0$. 
Since $m(D_1)\;|\;2$, we get $\overline{D_2}=0$, and 
\[
\am(X,G)= \langle \overline{D_1}\rangle \subseteq \zz/2\zz.
\]

\end{proof}

\section{Other blow-ups}
In this section, we discuss the Amitsur groups for the following families of smooth Fano threefolds
\begin{center}
\textnumero 3.7,\; \textnumero 3.8,\; \textnumero 4.4,\; \textnumero 5.1. 
\end{center}

\subsection{Blow-up of the Fano threefold from \textnumero 2.32} We recall that the unique Fano threefold $W$ in family \textnumero $2.32$ is a divisor of bidegree $(1,1)$ on $\pp^2\times\pp^2$, and from Table \ref{table:amitsur}, $\am(W,G)$ is isomorphic to $\zz/3\zz$ or trivial. 
Let $C:= D_1\cap D_2$ denote the smooth intersection of two divisors $D_1,D_2 \in |-\frac{1}{2}K_W|$. Let $\operatorname{pr}_i$ denote the projection of $\pp^2\times\pp^2$ onto its $i$th factor, for $i=1,2$. Set $H_i:= \big(\operatorname{pr}_i^*\mathcal{O}_{\pp^2}(1)\big)|_W$. Then $D_1, D_2\in |H_1+H_2|$. 

A Fano threefold $X$ in family \textnumero 3.7 is the blow-up of $W$ along the curve $C$. Let $\pi : X \to W$ be the blow-up map. Then $\pic(X)=\zz[\pi^*H_1,\pi^*H_2,E]$, where $E$ is the exceptional divisor, and $-K_X = \pi^*(2H_1+2H_2)-E$. 

Let $G\subseteq \aut(X)$ be a finite group. From the $G$-action on $\pic(X)$ noted in \cite{sharma2026actionspicardgroupsmooth}, we observe that $E$ is $G$-invariant. So, Lemma \ref{Lemma:autblowup} implies that we may regard $G$ as a subgroup of $\aut(W;C)$, and consequently, $\pi$ is $G$-equivariant. It now follows from Theorem \ref{imprimitivebcdp} that 
\begin{equation}   \label{eqn:bound-3.7}
   \am(X,G)=\am(W,G), 
\end{equation}
and $E$ is $G$-linearizable.

Thus, $\am(X,G)\subseteq \zz/3\zz$. In the following lemma, we show that this upper bound is sharp. 

Let $([x_0:x_1:x_2],[y_0:y_1:y_2])$ be projective coordinates on  $\pp^2\times\pp^2$. 
 
\begin{lemma}   \label{Am:3.7}
    Let $X$ be a smooth Fano threefold in \textnumero 3.7. Let $G\subseteq \aut(X)$ be a finite group. Then $\am(X,G)$ is isomorphic to  $\zz/3\zz$ or trivial. 
\end{lemma}
\begin{proof}
Suppose $\autp(X,G)=\zz/2\zz$. Then $\pic(X)^G=\zz[\pi^*(H_1+H_2),E]$. We recall a part of the proof of \cite[Lemma 5.4]{sharma2026amitsurgroupsprimitivefano}: the natural embedding $W \hookrightarrow (\pp^2)^2$ is $G$-equivariant due to Remark \ref{remark:equivarince-of-embedding}, and since
\[
H_1+H_2 = -K_{\pp^2\times\pp^2}|_W+K_W,
\]
$H_1+H_2$ is $G$-linearizable by Proposition \ref{KXlinearizable}. Since $\pi$ is $G$-equivariant, it follows that $\pi^*(H_1+H_2)$ is $G$-linearizable. Hence, $\am(X,G)=0$.  

If $\autp(X,G)=0$, then we show that that the bound on $\am(X,G)$ is sharp. We use the following construction of $X$ from \cite[\S 4.3]{calabi}--- let $W$ be given by the hypersurface $x_0y_0 + x_1y_1 + x_2y_2=0$ in $\mathbb{P}^2 \times \mathbb{P}^2$, and consider the automorphisms of $\pp^2\times\pp^2$
\begin{align*}
    \sigma : ([x_0:x_1:x_2],[y_0:y_1:y_2]) &\mapsto ([x_2:x_0:x_1],[y_2:y_0:y_1]) \\
    \tau : ([x_0:x_1:x_2],[y_0:y_1:y_2]) &\mapsto ([x_0:\omega x_1:\omega^2 x_2],[y_0:\omega^2 y_1:\omega y_2]),
\end{align*}
where $\omega$ is a primitive cube root of unity. Set $\mathcal{G}:=\langle \sigma, \tau \rangle \cong (\zz/3\zz)^2$. Observe that $W$ is $\mathcal{G}$-invariant, and $\autp(W,\mathcal{G})=0$. Consider the curve $C$ be given by the intersection of hypersurfaces 
\[
\begin{cases}
    x_0y_1 + \omega x_1y_2 + \omega^2x_2y_0 &=0\\
    x_0y_2+\omega x_1y_0+\omega^2x_2y_1 &= 0.
\end{cases}
\]
Then $C$ is $\mathcal{G}$-invariant. We also verify the intersection $C\cap W$ to be smooth and irreducible \cite{magma}. Then the blow-up of $W$ in $C$ gives a Fano threefold $X$ in \textnumero 3.7. Thus, $\mathcal{G}$ lifts to a subgroup of $\aut(X)$, we identify it with $\mathcal{G}$.

Finally, since the induced $\mathcal{G}$-action on $\pp^2_x$ and $\pp^2_y$ does not lift to a group homomorphism to $\operatorname{GL}_3(\cc)$, $H_1$ and $H_2$ are not $\mathcal{G}$-linearizable. Hence, $\am(W,\mathcal{G})\cong \zz/3\zz$, and from \eqref{eqn:bound-3.7}, we obtain
\[\am(X,\mathcal{G}) \cong \zz/3\zz.\]
\end{proof}

\subsection{Blow-up of a Fano threefold from \textnumero 2.24} Let $Y$ be a Fano threefold in family \textnumero 2.24, realized as a divisor of bidegree $(1,2)$ in $\pp^2\times\pp^2$. Let $\mathscr{C}:=(\{p\}\times \mathbb{P}^2) \cap Y$ for $p\in \pp^2$; then $\mathscr{C}$ is a smooth conic in the second factor $\mathbb{P}^2$. A smooth Fano threefold $X$ in family \textnumero 3.8 is the blow-up of $Y$ in $\mathscr{C}$ \cite[Table 3]{MM81classtable}. 

Below we use an alternative description of $X$ from \cite[Lemma 8.6]{izv19}. Let $\operatorname{pr}_i : \mathbb{P}^1\times\mathbb{P}^2$ be natural projections to $\pp^1$ and $\pp^2$, respectively. Set $H_1:=\operatorname{pr}_1^*\mathcal{O}_{\mathbb{P}^1}(1)$ and $H_2:=\operatorname{pr}_2^*\mathcal{O}_{\mathbb{P}^2}(1)$.
Let $D_1\in |2H_2|$ and $D_2\in |H_1+2H_2|$ be two divisors on $\pp^1\times\pp^2$, and $C:=D_1\cap D_2$ be their complete intersection. Then $X$ is the blow-up of $\mathbb{P}^1\times \mathbb{P}^2$ in $C$. Let $\pi: X \to \mathbb{P}^1\times\mathbb{P}^2$ be the blow-up map and $E$ denote the exceptional divisor. Then $\pic(X)=\zz[\pi^*H_1,\pi^*H_2,E]$.

Recall that $\autp(X)=0$, so it suffices to assume that $G \subseteq \aut(\pp^1\times\pp^2;C)$, and $\pi$ is $G$-equivariant. Furthermore, since $G$ acts trivially on each factor of $\pp^1\times\pp^2$, the projection mappings $\operatorname{pr}_i$ are also $G$-equivariant. Therefore, the pullbacks of $-K_{\pp^1}$ and $-K_{\pp^2}$ are $G$-linearizable, and we have 
\begin{equation}   \label{eqn:relations-for-3.8}
    2\overline{\pi^*H_1} =0,\qquad 3\overline{\pi^*H_2}=0.
\end{equation}

\smallskip
\begin{lemma}  \label{Am:3.8}
    Let $X$ be a smooth Fano threefold in \textnumero 3.8. Let $G\subseteq \aut(X)$ be a finite group. Then $\am(X,G)=0$. 
\end{lemma}
\begin{proof}
Since $C$ is $G$-invariant, it follows from Lemma \ref{lemma:G-invariance-for-CI} (ii) that $D_1$ is $G$-invariant and there is divisor $D_2'\sim D_2$ such that $D_2'$ is $G$-invariant. Moreover, by Remark \ref{remark:CI}, both $D_1$ and $D_2'$ are $G$-linearizable on $\pp^1\times\pp^2$. Consequently, 
\[
2\overline{\pi^*H_2}= 0, \qquad \overline{\pi^*H_1+2\pi^*H_2}=0.
\]
Together with relations from \eqref{eqn:relations-for-3.8}, we get $\overline{\pi^*H_1}=\overline{\pi^*H_2}=0$. Hence, $\am(X,G)=0$.
\end{proof}

\subsection{Blow-up of a Fano threefold from \textnumero 3.19}
Recall the family \textnumero 3.19 from \S 3. A Fano threefold $Y$ in this family is the blow-up of a quadric $Q\subset \pp^4$ in two points $p$ and $q$. Let $\phi: Y\to Q$ be the blow-up map with exceptional divisor $E_p$ and $E_q$, and $C$ denote a conic containing $p,q$. 

Mori-Mukai describe a Fano threefold $X$ in family \textnumero 4.4 as the blow-up of $Y$ in the strict transform of $C$. Let $\pi:X \to Y$ denote this blow-up, and $F$ be its exceptional divisor. 

Let $H:=\pi^*\phi^*\mathcal{O}_Q(1)$, and $\widetilde{E_p}, \widetilde{E_q}$ denote the pullbacks to $X$ of $E_p,E_q$, respectively. Then 
\[
\pic(X)=\zz[H, \widetilde{E_p},\widetilde{E_q},F].
\]

\begin{lemma}   \label{Am:4.4}
     Let $X$ be a smooth Fano threefold in \textnumero 4.4. Let $G\subseteq \aut(X)$ be a finite group. Then $\am(X,G)=0$. 
\end{lemma}
\begin{proof}
From Table \ref{table:amitsur}, $\autp(X)=\zz/2\zz$. By the description in \cite[Lemma 4.23]{sharma2026actionspicardgroupsmooth}, we recall that the only nontrivial $G$-action on $\pic(X)$ swaps $\widetilde{E_p}$ and $\widetilde{E_p}$, while fixing $H$ and $F$. Thus, $F$ is $G$-invariant for all $G$. By Lemma \ref{Lemma:autblowup}, we may regard $G$ as a subgroup of $\aut(Y;C)$ so that $\pi$ is $G$-equivariant. Hence, by Theorem \ref{imprimitivebcdp} and Lemma \ref{Am:3.19}, it follows that 
\[
\am(X,G)=\am(Y,G)=0.
\]
\end{proof}

\subsection{Blow-up of a Fano threefold from \textnumero 2.29}
Let $C$ be a smooth conic on a quadric $Q\subset \pp^4$, and let $p_1, p_2, p_3 \in Q$ be three distinct points on $C$. 
The blow-up $\pi_1: Y\to Q$ along $C$ is a Fano threefold $Y$ in family \textnumero 2.29 \cite[Table 2]{MM81classtable}. 

A Fano threefold $X$ in \textnumero 5.1 is obtained by blowing-up $Y$ along $e_1,e_2,e_3$, where $e_i$ is the exceptional line on $Y$ contracted by $\pi_1$ to $p_i$, for $1\leq i \leq 3$. 

Using \cite[Corollary 2.6]{MM86}, the following is an equivalent construction of $X$: let $f: \widetilde{Y} \to Q$ denote the blow-up of $Q$ in $p_1,p_2,p_3$, with exceptional divisors $E_1,E_2,E_3$ and $\widetilde{C}$ denote the strict transform of the conic $C$. Then $X$ is obtained as the blow-up $\pi: X\to \widetilde{Y}$ along $\widetilde{C}$. Setting $E_7$ the exceptional divisor of this blow-up, and $H:=\pi^*f^*\mathcal{O}_Q(1)$, we have
\[
\pic(X) = \zz[H]\oplus\zz[\pi^*E_1]\oplus\zz[\pi^*E_2]\oplus\zz[\pi^*E_3]\oplus\zz[E_7].
\] 

\begin{lemma}\label{Am:5.1}
    Let $X$ be a smooth Fano threefold in family \textnumero 5.1. Let $G\subseteq \aut(X)$ be a finite group. Then $\am(X,G)=0$. 
\end{lemma}
\begin{proof}
From \cite[Lemma 4.24]{sharma2026actionspicardgroupsmooth}, we recall 
that $\autp(X)=S_3$. So, for groups $G\subseteq \aut(X)$, the divisor classes $H$ and $E_7$ are preserved and $G$ acts on $\{\pi^*E_1,\pi^*E_2,\pi^*E_3\}$ by permutation. Let $\iota\colon Q\hookrightarrow\pp^4$ denote the natural embedding. 
Note that $|H|$ defines the map 
$\phi_{|H|}:=\iota\circ f\circ\pi: X \to \pp^4$. Since $H$ is $G$-invariant, $\phi_{|H|}$ is a $G$-equivariant map. Thus, $\phi_{|H|}(X)=Q$ and the set $\{p_1,p_2,p_3\}$ are preserved under the $G$-action on $\pp^4$, and we may regard $G\subseteq \aut(Q,\{p_1,p_2,p_3\})$. Consequently, the composition $f\circ\pi:X \to Q$ is $G$-equivariant. By Theorem \ref{imprimitivebcdp} and \eqref{eqn:AmQ}, it follows that $\am(X,G) =0$.
\end{proof} 

\section{Direct Products}   \label{section:directproducts}
We study Amitsur groups for Fano threefolds that are products of the form $X\times Y$ and for which $\aut(X\times Y) \cong \aut(X)\times \aut(Y)$. This will yield an upper bound for the Amitsur groups of Fano threefolds belonging to families
 \begin{center}
 \textnumero 6.1, \textnumero 7.1, \textnumero 8.1, \textnumero 9.1, \textnumero 10.1.
 \end{center}
By Mori-Mukai classification, these families consist of products $\pp^1\times S_d$, where $S_d$ is a del Pezzo surface of degree $d=5, \dots , 1$. Each $S_d$ is obtained by blowing up $\pp^2$ at $9-d$ points in general position. 

\smallskip
We identify $\aut(X\times Y)$ with 
$\aut(X)\times \aut(Y)$ under the isomorphism 
\[
f: \aut(X)\times\aut(Y)\xrightarrow{\sim}\aut(X\times Y), \qquad (h,k) \to ((x,y) \mapsto (hx,ky)).
\]

Let $\operatorname{pr}_1: \aut(X \times Y)\to \aut(X)$ and $\operatorname{pr}_2: \aut(X \times Y)\to \aut(Y)$ denote the canonical projections. For a finite group $G\subseteq \aut(X\times Y)$, set $H:=\operatorname{pr}_1(G)$ and $K:=\operatorname{pr}_2(G)$. Then we may regard $G$ as a subgroup of $H\times K$. 

Finally, let $\pi_1 : X \times Y \to X$ and $\pi_2: X \times Y \to Y$ denote the projections. These are equivariant with respect to $\operatorname{pr}_1|_G$ and $\operatorname{pr}_2|_G$ .

\begin{lemma}  \label{lemma:directproducts-pic}
    Let $X$ and $Y$ be smooth projective varieties, and let $\pi_1,\pi_2$ denote the projections from $X\times Y$. Assume that the natural homomorphisms $\aut(X)\times \aut(Y) \to \aut(X\times Y)$ and 
    \[
    \pic(X \times Y) \to \pic(X)\oplus \pic(Y), \qquad (\mathcal{L},\mathcal{M})\mapsto \pi_1^*\mathcal{L}\otimes\pi_2^*\mathcal{M}
    \]
    are isomorphisms. 
    Let $G\subseteq \aut(X \times Y)$ be a finite group, and let $H\subseteq \aut(X)$ and $K\subseteq \aut(Y)$ denote its images under the corresponding projections. Then 
    \[
    \pic(X\times Y)^G \cong \pic(X)^H \oplus \pic(Y)^K.
    \]
\end{lemma}
\begin{proof}
Since $G$ can be identified with a subgroup of $H\times K$, we get
\[
(\pic(X)\oplus \pic(Y))^{H\times K} \hookrightarrow (\pic(X)\oplus \pic(Y))^{G} \cong \pic(X\times Y)^G.
\]
Conversely, let $(\mathcal{L},\mathcal{M})\in (\pic(X)\oplus\pic(Y))^G$. Then for any $g\in G$, we have 
\[
h^*\mathcal{L}\cong \mathcal{L}, \quad k^*\mathcal{M}\cong \mathcal{M},
\]
where $h=\operatorname{pr}_1(g)$ and $k=\operatorname{pr}_2(g)$. Finally, the result follows from the fact that $(\pic(X)\oplus \pic(Y))^{H\times K} \cong \pic(X)^H \oplus \pic(Y)^K$. 
\end{proof}

Under the hypothses of the previous lemma, we obtain an upper bound for $\am(X\times Y,G)$.
\begin{prop}
    Let $G\subseteq \aut(X\times Y)$ be a finite group. Then 
    \[
    \am(X,H) \oplus \am(Y,K) \twoheadrightarrow \am(X\times Y,G).
    \]
\end{prop}
\begin{proof}
Consider the following diagram, 
\begin{center}
\resizebox{0.95\linewidth}{!}{$
\begin{tikzcd}[ampersand replacement=\&, column sep=large, row sep=large]
    {\operatorname{Pic}(X,H)\oplus\operatorname{Pic}(Y,K)}
    \&
    {\operatorname{Pic}(X)^H\oplus\operatorname{Pic}(Y)^K}
    \&
    {H^2(H,\mathbb{C}^*)\oplus H^2(K,\mathbb{C}^*)}
    \\
    {\operatorname{Pic}(X\times Y,G)}
    \&
    {\operatorname{Pic}(X\times Y)^G}
    \&
    {H^2(G,\mathbb{C}^*)}
    \arrow["{\epsilon_X\times\epsilon_Y}", from=1-1, to=1-2]
    \arrow[from=1-1, to=2-1]
    \arrow["{\delta_X\times\delta_Y}", from=1-2, to=1-3]
    \arrow["\sim",
        shift right, from=1-2, to=2-2]
    \arrow[from=1-3, to=2-3]
    \arrow["{\epsilon_{XY}}", from=2-1, to=2-2]
    \arrow["{\delta_{XY}}", from=2-2, to=2-3]
\end{tikzcd}
$}
\end{center}
Let $\mathcal{L}\in \pic(X)^H$ be $H$-linearizable line bundle on $X$ and $\mathcal{M}\in \pic(Y)^K$ be a $K$-linearizable line bundle on $Y$. Then $\pi_1^*\mathcal{L}\otimes\pi_2^*\mathcal{M} \in \pic(X\times Y)^G$ admits a $G$-linearization. Thus, $\operatorname{Im}\epsilon_X \oplus \operatorname{Im}\epsilon_Y \hookrightarrow \operatorname{Im}\epsilon_{XY}$. Combined with Lemma \ref{lemma:directproducts-pic}, we obtain
\[
\am(X,H)\oplus \am(Y,K) \twoheadrightarrow \am(X\times Y,G). 
\]
\end{proof}

Let $X=\pp^1\times S_d$ be a Fano threefold with $d\leq 5$. Then we have
\begin{equation}   \label{eqn:aut-for-delPezzos}
\aut(\pp^1\times S_d) \cong \aut(\pp^1) \times \aut(S_d).
\end{equation}

If $S_d$ is a $\widetilde{G}$-minimal del Pezzo surface, the Amitsur groups $\am(S_d,\widetilde{G})$ have been classified in \cite[Proposition A.7]{BCDP}. In particular, when $\operatorname{rk}\pic(X)^G=1$ and $d\leq 6$, then $\am(S_d,\widetilde{G})=0$. 
\begin{coro}  \label{amforFanodirectproducts}
    Let $X$ be a smooth Fano threefold belonging to one of the families
    \begin{center}
        \textnumero 6.1,\; \textnumero 7.1,\; \textnumero 8.1,\; \textnumero 9.1,\;, \textnumero 10.1.
    \end{center} 
    Let $G\subseteq \aut(X)$ be finite group. Then we have the following possibilities for  $\am(X,G)$
   \[
        (\mathbb{Z}/2\mathbb{Z})^3,\; \zz/2\zz\times\zz/3\zz,\; (\mathbb{Z}/2\mathbb{Z})^2,\; \zz/3\zz,\; \zz/2\zz,\; 0.
    \]
\end{coro}

\newpage

\section{Table for Amitsur groups}   \label{section:tableforAmitsur}
In the following tables, column 1 lists the Mori–Mukai names of the deformation
families of smooth Fano threefolds. Column 2 lists the possible groups AutP(X,G), that is, the images of finite groups $G \subseteq \aut(X)$ in $\aut(\pic(X))$. For each entry in
column 2, column 3 records the largest possible Amitsur group $\am(X,G)$. Column 5 lists the references where $\am(X,G)$ has been classified for the
corresponding family.
\small{\begin{longtable} {|p{1.8cm}|p{2.0cm}|p{3.0cm}|p{3.5cm}|}
 \caption{Fano threefolds}
 \label{table:amitsur}\\
    \hline
     MM \textnumero & $\autp(X,G)$ & $\am(X,G)$ &  Refer\\
\hline
     1.1, \dots , 1.10    & 0   & 0 & \cite{sharma2026amitsurgroupsprimitivefano} \\
\hline
     1.11    & 0  & $0$ & \cite[Lemma 4.3]{sharma2026amitsurgroupsprimitivefano} \\
\hline
     1.12  & 0   & $\zz/2\zz$ & \cite[Lemma 4.4]{sharma2026amitsurgroupsprimitivefano} \\
\hline
    1.13  & 0    & $0$ & \cite[Lemma 4.6]{sharma2026amitsurgroupsprimitivefano} \\
\hline
    1.14  & 0  & $\zz/2\zz$ & \cite[Lemma 4.7]{sharma2026amitsurgroupsprimitivefano} \\
\hline
   1.15  &  0  & $0$ & \cite[Lemma 4.8]{sharma2026amitsurgroupsprimitivefano} \\
\hline
   1.16  & 0  & $0$ & \cite[Lemma 4.6]{sharma2026amitsurgroupsprimitivefano} \\
\hline
   1.17  &  0  & $\zz/4\zz$ & \cite[Example 2.3]{sharma2026amitsurgroupsprimitivefano} \\
\hline
 2.1  &  0   & $0$ & Corollary \ref{Am:2.1}\\
\hline
 2.2  & $0$  & 0 & \cite[Lemma 7.2]{sharma2026amitsurgroupsprimitivefano} \\
\hline
 2.3  & 0  & $\zz/2\zz$ & Lemma \ref{Am:2.3} \\
\hline
2.4 &  0 & $\zz/2\zz$ & Lemma \ref{Am:2.4}\\
\hline
 2.5  & 0  & $0$ & Corollary \ref{Am:2.5,2.11}\\
\hline
 2.6 (a)  &  $\zz/2\zz$ & $0$ & \cite[Lemmas 5.2, 7.5]{sharma2026amitsurgroupsprimitivefano} \\
         &    $0$       &   $\zz/3\zz$  &      \\
\hline
 2.7  & 0   & $0$ & Lemma \ref{Am:2.7,2.13,2.17,2.21,2.23,2.29,2.31} \\
\hline
 2.8  & 0   & $0$ & \cite[Lemma 7.6]{sharma2026amitsurgroupsprimitivefano} \\
\hline
 2.9  & 0  & $0$ & Lemma \ref{Am:2.9} \\
\hline
 2.10  & 0  & $\zz/2\zz$ & Lemma \ref{Am:2.10} \\
\hline
 2.11  &  0  & $0$ & Corollary \ref{Am:2.5,2.11} \\
\hline
 2.12  & $\zz/2\zz$  & $0$ & Lemma \ref{Am:2.12}\\
       &   $0$      & $\zz/4\zz$ &      \\
\hline
 2.13  & 0  & $0$ & Lemma \ref{Am:2.7,2.13,2.17,2.21,2.23,2.29,2.31} \\
\hline
 2.14 &  0   & $0$ & Lemma \ref{Am:2.14,2.20,2.22,2.26} \\
\hline
 2.15  & 0   & $0$ & Lemma \ref{Am:2.15}\\
\hline
 2.16  & 0    & $0$ & Lemma \ref{Am:2.16} \\
\hline
 2.17  &  0  & $0$ & Lemma \ref{Am:2.7,2.13,2.17,2.21,2.23,2.29,2.31} \\
\hline
 2.18  &  0   & $0$ & \cite[Lemma 7.3]{sharma2026amitsurgroupsprimitivefano} \\
\hline
 2.19  & 0   & $\zz/2\zz$ & Lemma \ref{Am:2.19}\\
\hline
2.20  & $0$  & $0$ & Lemma \ref{Am:2.14,2.20,2.22,2.26} \\
\hline
 2.21  &  $\zz/2\zz$   & $0$ & Lemma \ref{Am:2.7,2.13,2.17,2.21,2.23,2.29,2.31} \\
       &    $0$        &  $0$  &   \\
\hline
 2.22  &  $0$   & $0$ & Lemma \ref{Am:2.14,2.20,2.22,2.26} \\
\hline
 2.23  & $0$    & $0$ & Lemma \ref{Am:2.7,2.13,2.17,2.21,2.23,2.29,2.31} \\
\hline
 2.24  & $0$   & $\zz/3\zz$ & \cite[Lemma 5.3]{sharma2026amitsurgroupsprimitivefano} \\
\hline
2.25  & $0$   & $\zz/2\zz$ & Lemma \ref{Am:2.25}\\
\hline
 2.26  & $0$    & $0$ & Lemma \ref{Am:2.14,2.20,2.22,2.26} \\
\hline
 2.27  & $0$     & $\zz/2\zz$ & Lemma \ref{Am:2.27}\\
\hline
 2.28  & $0$    & $0$ & Lemma \ref{Am:2.28}\\
\hline
 2.29  & $0$   & $0$ & Lemma \ref{Am:2.7,2.13,2.17,2.21,2.23,2.29,2.31} \\
\hline
 2.30  & $0$   & $0$ & Lemma \ref{Am:2.30}\\
\hline
 2.31  & $0$  & $0$ & Lemma \ref{Am:2.7,2.13,2.17,2.21,2.23,2.29,2.31} \\
\hline
 2.32  & $\zz/2\zz$   & $0$ & \cite[Lemma 5.4]{sharma2026amitsurgroupsprimitivefano}\\ 
      &   $0$      &   $\zz/3\zz$ &  \\
\hline
 2.33  & $0$   & $\zz/2\zz$ & Lemma \ref{Am:3.6,3.12}; Toric, \cite{sharma2026amitsurgroupsprimitivefano} \\ 
\hline
 2.34  & $0$    & $\zz/6\zz$  & \\  
    &     &     & Toric, \cite{sharma2026amitsurgroupsprimitivefano} \\
\hline
 2.35  & $0$   & $0$ & Lemma \ref{Am:2.35}, Toric \\
\hline
 2.36  & $0$   & $0$ & Toric, \cite{sharma2026amitsurgroupsprimitivefano} \\
\hline
 3.1  & $0$   &  $(\zz/2\zz)^2$ & \cite[Lemma 7.8]{sharma2026amitsurgroupsprimitivefano}\\
     &  $\zz/2\zz$   & $\zz/2\zz$  &  \\
     &  $\zz/3\zz$   & $0$   &    \\
     &  $S_3$       & $0$    &    \\
\hline
 3.2  & $0$   & $\zz/2\zz$ & \cite[Lemma 6.1]{sharma2026amitsurgroupsprimitivefano}\\
\hline
 3.3  & $\zz/2\zz$   & $0$ & Lemma \ref{Am:3.3} \\
       &  $0$      &  $\zz/2\zz$   &    \\
\hline
 3.4  & $0$   & $0$ & \cite[Corollary 7.4]{sharma2026amitsurgroupsprimitivefano}\\
\hline
 3.5  & $0$   & $\zz/2\zz$ &  Lemma \ref{Am:3.5}\\
\hline
 3.6  & $0$     & $\zz/2\zz$ & Lemma \ref{Am:3.6,3.12}\\
\hline
 3.7  & $\zz/2\zz$   & $0$ &  Lemma \ref{Am:3.7}\\
      &         & $\zz/3\zz$    &  \\
\hline
 3.8  & $0$  & $0$ & Lemma \ref{Am:3.8}\\
\hline
 3.9  & $\zz/2\zz$  & 0 & Lemma \ref{Am:3.9} \\
\hline
 3.10  & $\zz/2\zz$   & $0$ & Lemma \ref{Am:3.10,3.15,3.20} \\
\hline
 3.11  & $0$   & $0$ & Corollary \ref{Am:blowup_of_2.35}\\
\hline
 3.12  & $0$    & $\zz/2\zz$ & Lemma \ref{Am:3.6,3.12} \\
\hline
 3.13  & $S_3$    & $0$  & Lemma \ref{Am:3.13}\\
       & $\zz/3\zz$   & $0$ &              \\
       & $\zz/2\zz$    & $0$ &             \\
       &   $0$         & $0$ &             \\
\hline
 3.14  & $0$    & $0$ &  Corollary \ref{Am:blowup_of_2.35}\\
\hline
 3.15  & $0$    & $0$ & Lemma \ref{Am:3.10,3.15,3.20} \\
\hline
 3.16  & $0$   & $0$ & Corollary \ref{Am:blowup_of_2.35}\\
\hline
 3.17  & $\zz/2\zz$   & $0$ & Lemma \ref{Am:3.17}\\
       &   $0$      & $\zz/2\zz$      &    \\
\hline
 3.18  & $0$   & $0$ &  Lemma \ref{Am:3.18}\\
\hline
 3.19  & $\zz/2\zz$  & $0$ & Lemma \ref{Am:3.19} \\
\hline
 3.20  & $\zz/2\zz$   & $0$ &  Lemma \ref{Am:3.10,3.15,3.20}\\
\hline
 3.21  & $0$   & 0 & Lemma \ref{Am:3.21} \\
\hline
 3.22  &  $0$   & $0$ & Lemma \ref{Am:3.22}\\
\hline
 3.23  & $0$   & $0$ & Corollary \ref{Am:blowup_of_2.35} \\
\hline
 3.24  & $0$   & $0$  & Lemma \ref{Am:3.24} \\
\hline
 3.25  & $\zz/2\zz$   & $\zz/4\zz$  & Toric, \cite{sharma2026amitsurgroupsprimitivefano}\\
       &  $0$      & $\zz/2\zz$  &    \\
\hline
 3.26  & $0$    & $0$ & Corollary \ref{Am:blowup_of_2.35}; Toric, \cite{sharma2026amitsurgroupsprimitivefano} \\
\hline
 3.27  & $S_3$   & $\zz/2\zz$ & Toric, \cite{sharma2026amitsurgroupsprimitivefano} \\
   
 & $\zz/3\zz$    & $\zz/2\zz$ &  \\
  
 & $\zz/2\zz$  & $(\zz/2\zz)^2$  &   \\
 
 & $0$ &  $(\zz/2\zz)^3$  &   \\
\hline
 3.28  & $0$   & $\zz/2\zz$ & Toric, \cite{sharma2026amitsurgroupsprimitivefano} \\
\hline
 3.29  & $0$   & $0$ & Corollary \ref{Am:blowup_of_2.35}; Toric, \cite{sharma2026amitsurgroupsprimitivefano} \\
\hline
 3.30  & $0$   & $0$ & Corollary \ref{Am:blowup_of_2.35}; Toric, \cite{sharma2026amitsurgroupsprimitivefano}\\
\hline
 3.31  & $\zz/2\zz$   & $0$ & Toric, \cite{sharma2026amitsurgroupsprimitivefano}\\
 &     $0$  &   $\zz/2\zz$  &  \\
\hline
 4.1  & $0$ &  $(\zz/2\zz)^3$ & Lemma \ref{Am:4.1}\\
    &  $\zz/2\zz$  &   $(\zz/2\zz)^2$     &     \\   & $\zz/3\zz$  &   $\zz/2\zz$     &    \\
    &  $\zz/4\zz$   &     $0$   &  \\ 
    &  $(\zz/2\zz)^2$      & ?    &  \\
    &  $S_3$   &  $\zz/2\zz$     &  \\
    &  $A_4$   &  $0$   &  \\
    &  $S_4$   &   $0$   &  \\
\hline
 4.2  & $(\zz/2\zz)^2$ &  ?  & Lemma \ref{Am:4.2}\\
\hline
 4.3  & $\zz/2\zz$    & $0$ &  Lemma \ref{Am:4.3}\\
 &  $0$   &     $\zz/2\zz$   &  \\
\hline
 4.4  & $\zz/2\zz$   & $0$ &  Lemma \ref{Am:4.4}\\
\hline
 4.5  & $0$   & $0$ & Corollary \ref{Am:4.5} \\
\hline
 4.6  & $S_3$   & $\zz/2\zz$ & Lemma \ref{Am:4.6}\\
      &  $\zz/3\zz$   & $\zz/2\zz$     &                 \\
      &  $\zz/2\zz$   & $\zz/2\zz$      &                 \\
      &  $0$   &  $\zz/2\zz$     &                 \\
\hline
 4.7  & $\zz/2\zz$  & $0$ & Lemma \ref{Am:4.7} \\
\hline
 4.8  & $\zz/2\zz$  & $0$ &  Lemma \ref{Am:4.8} \\
 &  $0$   &    $\zz/2\zz$   &  \\
\hline
 4.9 &  $0$ &  $0$ & Toric, \cite{sharma2026amitsurgroupsprimitivefano} \\
\hline
 4.10  & $\zz/2\zz$  & $\zz/2\zz$ & Toric, \cite{sharma2026amitsurgroupsprimitivefano} \\
\hline
 4.11  & $\zz/2\zz$    & $\zz/2\zz$ & Toric, \cite{sharma2026amitsurgroupsprimitivefano} \\
 &  $0$  &    $0$   & \\
\hline
 4.12  & $0$   & $0$ & Toric, \cite{sharma2026amitsurgroupsprimitivefano}\\
     &                                 &  &  \\
\hline
 4.13  & $0$   & $\zz/2\zz$ & Lemma \ref{Am:4.13} \\
       & $\zz/2\zz$    &  $\zz/2\zz$  &   \\
\hline
 5.1  &  $S_3$      &  $0$ & Lemma \ref{Am:5.1}\\
      & $\zz/3\zz$    & $0$ &  \\ 
      & $\zz/2\zz$   &  $0$  &   \\
      & $0$        &   $0$  &   \\
\hline
 5.2  &  $\zz/2\zz$   & $\zz/2\zz$ & Toric, \cite{sharma2026amitsurgroupsprimitivefano} \\
     &   $0$        &   $0$   &    \\
\hline
 5.3  & 0 &  $\zz/2\zz$ & Toric, \cite{sharma2026amitsurgroupsprimitivefano} \\
 &   $\zz/2\zz$ &   $\zz/2\zz$  & \\
 &    $\zz/2\zz$ &   $\zz/2\zz$  &  \\
 &   $\zz/2\zz$ &   $(\zz/2\zz)^3$  & \\
 &  $\zz/3\zz$ &    $\zz/6\zz$  &  \\
 &    $(\zz/2\zz)^2 $   &  $(\zz/2\zz)^2$  & \\
 &  $S_3$ &   $\zz/2\zz$  & \\
 &   $S_3$ &    $\zz/6\zz$  &  \\
 &    $\zz/6\zz$    & $\zz/2\zz$   &  \\
 &   $D_6$ &    $\zz/2\zz$  &   \\
\hline
 6.1  & $S_5$   & - & Corollary \ref{amforFanodirectproducts} \\
\hline
 7.1  & -   & - & Corollary \ref{amforFanodirectproducts}\\
\hline
 8.1  &  -   & - & Corollary \ref{amforFanodirectproducts}\\
\hline
 9.1  & -    & - & Corollary \ref{amforFanodirectproducts}\\
\hline
 10.1  & -   & - & Corollary \ref{amforFanodirectproducts} \\
\hline  
\end{longtable}}

\newpage

\bibliographystyle{alphaurl} 
\bibliography{bibis} 

\end{document}